\documentclass[sn-mathphys-num]{sn-jnl}% Math and Physical Sciences Numbered Reference Style 
\usepackage{graphicx}%
\usepackage{multirow}%
\usepackage{amsmath,amssymb,amsfonts}%
\usepackage{amsthm}%
\usepackage[title]{appendix}%
\usepackage{xcolor}%
\usepackage{textcomp}%
\usepackage{manyfoot}%
\usepackage{booktabs}%
\usepackage{algorithm}%
\usepackage{algorithmicx}%
\usepackage{algpseudocode}%
\usepackage{listings}%
\usepackage{lmodern} % to remove the warning
\usepackage{dsfont} % replace of bbm
\usepackage{siunitx}

\newtheorem{theorem}{Theorem}
\newtheorem{lem}{Lemma}
\newtheorem{rmk}{Remark}
\newtheorem{prop}{Proposition}
\newtheorem{assump}{Assumption}
\newcommand{\ve}{\varepsilon}

\newcommand{\EE}{\mathbb{E}}
\newcommand{\RR}{\mathbb{R}}
\newcommand{\PP}{\mathbb{P}}
\newcommand{\ZZ}{\mathbb{Z}}
\newcommand{\mF}{\mathcal{F}}
\newcommand{\mI}{\mathcal{I}}
\newcommand{\mL}{\mathcal{L}}

\newcommand{\mS}{\mathcal{S}}
\newcommand{\mU}{\mathcal{U}}
\newcommand{\mV}{\mathcal{V}}
\newcommand{\mX}{\mathcal{X}}

\newcommand{\rd}{\mathrm{d}}

\newcommand{\bt}{\overline{t}}
\newcommand{\bx}{\overline{x}}
\newcommand{\btx}{(\overline{t}, \overline{x})}
\newcommand{\mG}{\mathcal{G}}
\newcommand{\nbu}{\nabla_u}
\newcommand{\nx}{\nabla_x}
\newcommand{\ny}{\nabla_y}
\newcommand{\pt}{\partial_t}
\newcommand{\ps}{\partial_s}
\newcommand{\pve}{\partial_{\ve}}
\newcommand{\hp}{\widehat{p}}
\newcommand{\hq}{\widehat{q}}
\newcommand{\htk}{\widehat{t}_k}
\newcommand{\hsk}{\widehat{s}_k}
\newcommand{\hxk}{\widehat{x}_k}
\newcommand{\hyk}{\widehat{y}_k}

\newcommand{\xyet}{x^{1,\ve}_t}
\newcommand{\xeet}{x^{2,\ve}_t}

\newcommand{\pyet}{\phi^{1,\ve}_t}
\newcommand{\peet}{\phi^{2,\ve}_t}
\newcommand{\hPhik}{\widehat{\Phi}_k}

\newcommand{\TD}{\mathrm{TD}}
\newcommand{\tp}{^{\top}}
\newcommand{\parentheses}[1]{\left(#1\right)}
\newcommand{\sqbra}[1]{\left[#1\right]}
\newcommand{\curlybra}[1]{\left\{#1\right\}}
\newcommand{\abs}[1]{\left|#1\right|}
\newcommand{\norm}[1]{\left\|#1\right\|}

\newcommand{\inttx}[1]{\int_0^T \int_\mX #1 \rd x \, \rd t}
\newcommand{\inner}[2]{\left\langle#1,\,#2\right\rangle}
\newcommand{\fd}[2]{\dfrac{\delta #1}{\delta #2}}
\newcommand{\pd}[2]{\dfrac{\partial #1}{\partial #2}}
\newcommand{\rom}[1]{\uppercase\expandafter{\romannumeral #1\relax}}

\DeclareMathOperator{\Tr}{Tr}

\DeclareMathOperator*{\argmax}{arg\,max}

\begin{document}

\title[Actor-Critic Flow for Optimal Control]{Solving Time-Continuous Stochastic Optimal Control Problems: Algorithm Design and Convergence Analysis of Actor-Critic Flow}

\author*[1]{\fnm{Mo} \sur{Zhou}}\email{mozhou366@math.ucla.edu}

\author[2]{\fnm{Jianfeng} \sur{Lu}}\email{jianfeng@math.duke.edu}

\affil*[1]{\orgdiv{Department of Mathematics}, \orgname{University of California}, \orgaddress{\city{Los Angeles}, \postcode{90095}, \state{CA}, \country{USA}}}

\affil[2]{\orgdiv{Department of Mathematics, Department of Physics, and Department of Chemistry}, \orgname{Duke University}, \orgaddress{ \city{Durham}, \postcode{27708}, \state{State}, \country{USA}}}

\abstract{We propose an actor-critic framework to solve the time-continuous stochastic optimal control problem. A least square temporal difference method is applied to compute the value function for the critic. The policy gradient method is implemented as policy improvement for the actor. Our key contribution lies in establishing a linear rate of convergence for our proposed actor-critic flow. Theoretical findings are further validated through numerical examples, showing the efficacy of our approach in practical applications.}




\maketitle

\section{Introduction}\label{sec:intro}

The optimal control problem has an exciting long history since last century. In the early 1950s, Pontryagin introduced the maximum principle \citep{boltyanski1956theory, boltyanski1960maximum}, which provided a general framework for solving optimal control problems. Around the same time, Bellman developed dynamic programming \citep{bellman1966dynamic}, a mathematical framework for solving complex optimization problems, which allows for the efficient solution of problems with overlapping subproblems. Then, Kalman made significant contributions to the field of control theory \citep{kalman1960contributions}, particularly in the area of linear quadratic (LQ) control. The combination of Pontryagin's maximum principle, Bellman's dynamic programming, and Kalman's linear quadratic control formed the basis of optimal control theory. Since then, researchers have been exploring upon these foundational ideas \citep{kushner1990numerical}, leading to further developments and applications in various fields, including aerospace \citep{longuski2013optimal}, robotics \cite{bobrow1985time}, economics \citep{leonard1992optimal}, finance \citep{forsyth2007numerical} and beyond.

In addition to these traditional methods mentioned above, there is an increasing number of works that focus on the combination of machine learning and optimal control in recent years \citep{tzen2020mean,nakamura2021adaptive,bachouch2022deep,meng2022sympocnet,onken2022neural,bokanowski2023neural,domingo2023stochastic}, motivated by the impressive success of machine learning. Surveys on recent development for optimal control problems have been conducted in \cite{jin2022survey,hu2023recent}.

Reinforcement learning (RL) \citep{sutton2018reinforcement} has a close relationship with the optimal control problem. While one usually maximizes the reward with discrete time and state in RL, the goal of an optimal control problem is to minimize a cost functional with continuous time and state space. These two topics share many common concepts such as value function, the Bellman equation, and the dynamic programming principle. While traditionally research in these two areas follow separate routes, in recent years, there have been increasing cross-fertilization between the two fields \citep{recht2019tour,wang2020reinforcement, munos1997reinforcement,hure2021deep,quer2022connecting,frikha2023actor}.

\subsection*{Our contribution.} This work lies at the intersection of reinforcement learning and optimal control problem. We develop a general framework to solve the optimal control problem, borrowing ideas from the actor-crtic framework of RL. A deterministic feedback control without entropy regularization is applied in order to obtain  exact control function. We design a least square temporal difference (LSTD) method to compute the value function for the critic. This temporal difference (TD) is obtained from It\^o's calculus, reducing the error compared with a vanilla TD from RL. For the actor part, we implement a policy gradient method, with an explicit expression for the variation of the objective. This policy gradient method outperforms a vanilla gradient descent method. The main theoretical result of this paper is the linear convergence guarantee for our algorithm at the level of continuous gradient flow. The result is also validated through numerical examples.

\subsection*{Related works.}

In the analysis of numerical algorithms for optimal control problems, many studies focus on specific setting such as the LQ problem \citep{darbon2023neural}, because of its clear structure. Various works have proposed policy gradient methods for linear quadratic regulator (LQR), with global convergence analysis \citep{wang2021global,giegrich2022convergence,hambly2021policy}. A study of Newton's method for control problems with linear dynamics and quadratic convergence is conducted in \cite{gobet2022newton}. There are also works that focus on the Hamilton--Jacobi--Bellman (HJB) equation for the LQ system \cite{zhou2024deep}.

It has been an common research theme to consider soft policies in optimal control, i.e., the control is not deterministic. The soft policy enables exploration for global search, making the convergence analysis more tractable \citep{zhou2021curse, ma2024convergence}.
For instance, in the context of mean-field games, the convergence of soft policies to the Nash equilibrium is guaranteed under mild assumptions \citep{domingo2020mean, firoozi2022exploratory, guo2022entropy}. The convergence for soft policy gradient methods in finite action spaces is studied in \cite{reisinger2021regularity}. A rate of convergence for general action spaces is then established in \cite{sethi2024entropy}. Jia and Zhou propose an actor-critic method for the stochastic problem \cite{jia2022policy}. The use of soft policies however leads to sub-optimality, to overcome this, Tang et al. analyzes the property of a class of soft policy algorithms where the rate of exploration decreases to zero \cite{tang2022exploratory}.

In more general settings, there have been several recent works on the convergence of algorithms for various optimal control problems. s. Carmona and Lauri\`ere analyze approximation errors with neural networks for LQ mean-field games \cite{carmona2021convergence} and general mean-field games \cite{carmona2022convergence} . The mean field games is also studied in \cite{lauriere2022convergence,lauriere2023policy}.
Kerimkulov et al. study the convergence and stability of a Howard’s policy improvement algorithm \cite{kerimkulov2020exponential}. Ito et al. investigate an iterative method with a superlinear convergence rate  \cite{ito2021neural}. Kerimkulov et al. study the convergence rate of a method of successive approximations (MSA) algorithm with controlled diffusion \cite{kerimkulov2021modified}. A modified MSA method with convergence proof is proposed in \cite{sethi2022modified}. Huang et al. propose the policy iteration algorithm (PIA) and study its convergence properties \cite{huang2022convergence}. Reisinger et al. study the condition for linear convergence of policy gradient method \cite{reisinger2022linear}.

\medskip 

The rest of the paper is organized as follows. The background for the stochastic optimal control problem is introduced in Section \ref{sec:background}. Then we introduce our actor-critic algorithm  in Section \ref{sec:actor_critic}. In section \ref{sec:flow}, we analyze the continuous limit of our algorithm and describe the actor-critic flow. We establish the global convergence property of this actor-critic flow in Section \ref{sec:theory}, with proofs deferred to the appendix. We validate our algorithm through numerical examples in Section \ref{sec:example}. Finally, we conclude the paper and give future directions of research in \ref{sec:future}.

\section{The stochastic optimal control problem} \label{sec:background}
 
Let us denote $\mX$ an $n$-dimensional unit torus, which can be understood as $[0,1]^n$ with periodic boundary condition. We use $|\cdot|$ to denote the absolute value of a scalar, the $\ell^2$ norm of a vector, or the Frobenius norm of a matrix, depending on the context. $\norm{\cdot}_{L^1}$ and $\norm{\cdot}_{L^2}$ are the $L^1$ and $L^2$ norms of a function. $\norm{\cdot}_2$ represents the $\ell^2$ operator norm (i.e. the largest singular value) of a matrix. $\inner{\cdot}{\cdot}$ is the inner product between two vectors, and $\inner{\cdot}{\cdot}_{L^2}$ denotes the inner product between two $L^2$ functions. $\Tr(\cdot)$ means the trace of a squared matrix. Throughout the work, we fix $\left(\Omega, \mF, \{\mF_t\}, \PP \right)$ as a filtered probability space.

We consider the optimal control problem on state space $\mX$ with time $t\in[0,T]$. The objective is to minimize the following cost functional
\begin{equation}\label{eq:cost}
    J[u] = \EE\sqbra{\int_0^T r(x_t, u_t) \,\rd t + g(x_T)},
\end{equation}
subject to the dynamic
\begin{equation}\label{eq:SDE_X}
\rd x_t = b(x_t, u_t) \rd t + \sigma(x_t) \rd W_t.
\end{equation}
The functions $r(x,u): \mX \times \RR^{n'} \to \RR$ and $g(x): \mX \to \RR$ in \eqref{eq:cost} are the running cost and the terminal cost respectively. In the state dynamic \eqref{eq:SDE_X}, $b(x,u): \mX \times \RR^{n'} \to \RR^{n}$ and $\sigma(x): \mX \to \RR^{n \times m}$ are the drift and diffusion coefficients, $u_t \in \RR^{n'}$ is an $\mF_t$-adapted control process, and $W_t$ is an $m$-dimensional $\mF_t$-Brownian motion. In this work, we assume the initial state is uniformly distributed on $\mX$. In a more general setting, the diffusion $\sigma$ might also depend on the control $u_t$, while we will limit ourselves to only $x_t$ dependence in this work.

% The goal is to minimize the cost functional \eqref{eq:cost} over all adapted control processes $u_t$. 
In this work, we consider the smooth feedback control, i.e. it is a smooth function of time and state $u_t = u(t,x_t)$, because it is the case for the optimal control (see the verification theorem in \cite{yong1999stochastic} or an explanation in \cite{zhou2023policy}). This smoothness frees us form the technical definition for the viscosity solution. We will also denote the control process by $u$ when there is no confusion.

Throughout the paper, we assume that the matrix valued function
$$D(x) := \frac12 \sigma(x) \sigma(x)\tp \in \RR^{n \times n}$$
is uniformly elliptic with its smallest eigenvalue $\lambda_{\min}(x) \ge \sigma_0$ for some uniform $\sigma_0 > 0$. Let $\rho^{u}(t,x)$ be the density function of $x_t$ under control function $u$, then $\rho^{u}$ satisfies the Fokker Planck (FP) equation \citep{risken1996fokker}
\begin{equation}\label{eq:FokkerPlanck}
    \partial_t \rho^u(t,x) = -\nx \cdot \sqbra{b(x, u(t,x)) \rho^u(t,x)} + \sum_{i,j=1}^n \partial_i \partial_j \sqbra{D_{ij}(x) \rho^u(t,x)},
\end{equation}
where we denote $\partial_i = \partial_{x_i}$ for simplicity. The initial condition $\rho^u(0,\cdot)$ is the density of $x_0$. For example, $\rho^u(0,\cdot) \equiv 1$ if $x_0 \sim \text{Unif}(\mX)$ and $\rho^u(0,\cdot) = \delta_{x_0}$ if $x_0$ is deterministic. We will assume $\rho^u(0,\cdot)$ follows uniform distribution in the sequel if not specified. We denote $\mI_{u}$ the infinitesimal generator of the stochastic differential equation (SDE) \eqref{eq:SDE_X} under control $u$, then, the Fokker Planck equation \eqref{eq:FokkerPlanck} becomes $\partial_t \rho^u = \mI_u^{\dagger} \rho^u$, where $\mI_u^{\dagger}$ is the adjoint of $\mI_{u}$. (We reserve the more common generator notation $\mL$ for loss and Lyapunov functions)

Next, we introduce the value function, an important tool in optimal control. It is defined as the expected cost under a given control starting from a given time and state: 
\begin{equation}\label{eq:value}
    V_u(t,x) = \EE\sqbra{\int_t^T r(x_s, u_s) \,\rd s + g(x_T) ~\Big|~ x_t=x},
\end{equation}
where the subscript $u$ indicates that the value function is w.r.t. the control function $u$.
Using the Markov property, it can be shown that $V_u(t, x)$ satisfies the Bellman equation
\begin{equation}\label{eq:Bellman}
    V_u(t_1,x) = \EE \sqbra{\int_{t_1}^{t_2} r(x_t, u_t) \,\rd t + V_u(t_2, x_{t_2}) ~\Big|~ x_{t_1} = x}
\end{equation}
for any $0 \le t_1 \le t_2 \le T$ and $x \in \mX$. Analogous concepts for the value function and the Bellman equation are also common in RL \citep{sutton2018reinforcement}. %Taking the limit $t_2 \to t_1^+$ yields the Hamilton--Jacobi (HJ) equation
If we rearrange every term to the right, divide this equation by $t_2-t_1$, and then take the limit $t_2 \to t_1^+$, we obtain the Hamilton--Jacobi (HJ) equation
\begin{equation}\label{eq:HJ}
\left\{ \begin{aligned}
    -\partial_t V_u(t,x) + G\parentheses{t, x, u(t,x), -\nx V_u(t,x), -\nx^2 V_u(t,x)} & = 0 \\
    V_u(T, x) & = g(x).
\end{aligned} \right.
\end{equation}
for the value function, where $G: \RR \times \mX \times \RR^{n'} \times \RR^n \times \RR^{n\times n} \to \RR$ is the generalized Hamiltonian \citep{yong1999stochastic}, defined as
\begin{equation}\label{eq:generalized_Hamiltonian}
G\parentheses{t,x,u,p,P} := \frac12 \Tr\parentheses{P \sigma(x) \sigma(x)\tp} + \inner{p}{b(x,u)} - r(x,u).
\end{equation}

Another important quantity for optimal control is the adjoint state. Let $p_t = -\nx V_u(t,x_t)$ and $q_t = - \nx^2 V_u(t,x_t) \sigma(x_t)$. Then, $(p_t, q_t) \in \RR^n \times \RR^{n \times m}$ is the unique solution to the following backward stochastic differential equations (BSDE)  
\begin{equation}\label{eq:adjoint1}
\left\{ \begin{aligned}
    \rd p_t & = -\sqbra{\nx b(x_t,u(t, x_t))\tp p_t + \nx \Tr\parentheses{\sigma(x_t)\tp q_t} - \nx r(x_t,u(t, x_t))} \rd t + q_t ~ \rd W_t \\
    p_T & = -\nx g(x_T).
\end{aligned} \right.
\end{equation}

Extensive research has been conducted on the existence and uniqueness of the optimal control, see, for example \citep{yong1999stochastic}. We denote $u^*(\cdot,\cdot)$ the optimal control function, $x^*_t$ the optimal state process, and $V^*(t, x) = V_{u^*}(t,x)$ the optimal value function. Through the paper, we assume existence and uniqueness of the optimal control function, which will be further elaborated in Section \ref{sec:theory}. By the dynamic programming principle \citep{fleming2012deterministic},
\begin{equation*}
    V^*(t_1, x) = \inf_{u} \EE\sqbra{\int_{t_1}^{t_2} r(x_t, u_t) \,\rd t + V^*(t_2, x_{t_2}) ~\Big|~ x_{t_1}=x},
\end{equation*}
where the infimum is taken over all the control that coincide with $u^*$ in $[0,t_1] \cup [t_2,T]$. This dynamic programming principle can be also viewed as the optimal version of the Bellman equation \eqref{eq:Bellman}. %Again, taking the limit $t_2 \to t_1^+$, we arrive at the Hamilton--Jacobi--Bellman (HJB) equation
Similarly, if we rearrange all the terms to the right, divide the equation by $t_2-t_1$, and then take the limit $t_2 \to t_1^+$ \cite[Ch.1]{krylov2008controlled}, we arrive at the HJB equation
\begin{equation}\label{eq:HJB}
\left\{ \begin{aligned}
    -\partial_t V(t,x) + \sup_{u \in \RR^{n'}} G\parentheses{t, x, u, -\nx V(t,x), -\nx^2 V(t,x)} & = 0 \\
    V(T, x) & = g(x)
\end{aligned} \right.
\end{equation}
for the optimal value function. From this equation, we observe that a necessary condition for an optimal control $u$ is the maximum condition
\begin{equation}\label{eq:max1}
u(t,x) = \argmax_{u' \in \RR^{n'}} G\parentheses{t, x, u', -\nx V_u(t,x), -\nx^2 V_u(t,x)}.
\end{equation}
This maximization problem makes the semi-linear HJB equation \eqref{eq:HJB} difficult to solve.

\section{The actor-critic framework}\label{sec:actor_critic}

The actor-critic method, firstly proposed by Konda and Tsitsiklis \cite{konda1999actor}, is a popular approach in reinforcement learning (RL), which naturally also applies to optimal control problems. In this method, we solve for both the value function and the control field. The control (i.e.,  the policy in RL) is known as the \emph{actor}. The value function for the control is known as the \emph{critic}, because it is used to evaluate the optimality of the control. Accordingly, the actor-critic method consists of two parts: policy evaluation for the critic and policy improvement for the actor.

In this section, we establish an actor-critic framework to solve the optimal control problem. We develop a variant of TD learning for the critic and propose a policy gradient method for the actor.

\subsection{Policy evaluation for the critic}\label{sec:critic}

In this section, we consider the policy evaluation process for a fixed control $u_t = u(t, x_t)$, that is, we want to compute $V_u(\cdot,\cdot)$.
For the purpose, we will use a variant of TD learning.

The TD learning method \citep{sutton2018reinforcement} is the most popular method for evaluating the value function in RL. The term TD refers to a quantity that indicates the discrepancy between our current approximation of the value function and one of its sample estimate. 
By definition \eqref{eq:value}, we have
\begin{equation*}
V_u(0,x_0) = \EE \sqbra{\int_0^T r(x_t, u_t) \,\rd t + g(x_T) ~\Big|~ x_0},
\end{equation*}
which motivate us to define the TD as
\begin{equation}\label{eq:TD_RL}
\TD_{RL} = \int_0^T r(x_t, u_t) \,\rd t + g(x_T) - \mV_0(x_0).
\end{equation}
Here, $\mV_0$ is some estimate of $V_u(0,\cdot)$ and $\int_0^T r(x_t, u_t) \,\rd t + g(x_T)$ is an unbiased sample of $V_u(0,x_0)$ by definition. We use the subscript \emph{RL} to indicate that this TD in \eqref{eq:TD_RL} is a direct analog of the TD method used in RL. In this work, we will instead use an alternative version of TD derived from stochastic calculus. Applying It\^o's Lemma and the HJ equation \eqref{eq:HJ}, we obtain
\begin{equation}\label{eq:Ito}
\begin{aligned}
& \quad g(x_T) = V_u(T, x_T) \\
& = V_u(0, x_0) + \int_0^T \nx V_u(t, x_t)\tp \sigma(x_t) \,\rd W_t \\
& \quad + \int_0^T \sqbra{\partial_t V_u(t,x_t) + \frac12 \Tr\parentheses{\sigma(x_t) \sigma(x_t)\tp \nx^2 V_u(t,x_t)} + b(x_t, u_t)\tp \nx V_u(t,x_t)} \rd t\\
& = V_u(0, x_0) + \int_0^T \nx V_u(t, x_t)\tp \sigma(x_t) \,\rd W_t - \int_0^T r(x_t, u_t)\, \rd t.
\end{aligned}
\end{equation}
Therefore, we define a modified TD as 
\begin{equation}\label{eq:TD}
\TD = \int_0^T r(x_t, u_t) \,\rd t + g(x_T) - \mV_0(x_0) - \int_0^T \mG(t, x_t)\tp \sigma(x_t) \,\rd W_t,
\end{equation}
where $\mG$ is some estimate of $\nx V_u$. The temporal difference can be used for policy evaluation via LSTD method \citep{bradtke1996linear}, which minimizes the following expected squared loss
\begin{equation}\label{eq:critic_loss}
\mL_c = \frac12 \EE\sqbra{\TD^2}.
\end{equation}

There are two important advantages using \eqref{eq:TD} instead of \eqref{eq:TD_RL}. The first is that we are able to obtain an estimate of the gradient of value function, which is necessary in the actor part (see the next section). The second is that we can expect higher accuracy of the estimate due to the low-variance. If the estimate for the value function is precise, we have $\TD = 0$ \emph{almost surely} while we only have $\EE [\TD_{RL}] = 0$. The modified TD vanishes without taking the expectation because of the martingale term $- \int_0^T \mG(t, x_t)\tp \sigma(x_t) \,\rd W_t$. In comparison, $\TD_{RL}$ is nonzero, even for exact value function, which leads to higher noise of the estimate. This has been confirmed numerically in our previous work \citep{zhou2021actor}. We also remark that the $\TD_{RL}$ will be more useful for online learning, where we do not have access to the functions $b$, $\sigma$, and $r$.

Let $\tau$ denote the training time of the algorithm, which should be distinguished from the time $t$ in the dynamics of the control problem. The critic loss along the algorithm is denoted accordingly by $\mL_c^\tau$. In numerical implementation, we can parametrize $V_u(0,\cdot)$ and $\nx V_u$ as two neural networks, denoted by $\mV_0^\tau$ and $\mG^\tau$. We apply gradient based optimization algorithm such as Adam \citep{Kingma2015adam} on $\mL_c$ for the critic parameters $\theta_c$, which consists of all the trainable parameters for $\mV_0^\tau$ and $\mG^\tau$.

To be more specific, at each iteration of the algorithm, we sample $N$ state trajectories of $x_t$ using Euler-Maruyama scheme first. Let $h = \frac{T}{N_T}$ denote the step size. For $i=1,2,\ldots,N$, we sample 
\begin{equation}\label{eq:EM}
x^{(i)}_0 \sim \text{Unif}(\mX) \hspace{0.2in} x^{(i)}_{t + h} = x^{(i)}_t + b(x^{(i)}_t, u^\tau(t,x^{(i)}_t)) h + \sigma(x^{(i)}_t) \xi^{(i)}_t
\end{equation}
for $t=0,h,\ldots,(N_T-1)h$, where $\xi^{(i)}_t \sim N(0,h I_m)$ are i.i.d. sampled Brownian increments. Let us denote the set of these samplings by
\begin{equation}\label{eq:samplings}
\mS_\tau := \{(j h,x_{j h}^{(i)}) ~|~ j=0,1,\ldots N_T ~~ i=1,\ldots N \},
\end{equation}
which is the Monte Carlo sampling from the density $\rho^{u^\tau}$. Then we approximate the TD \eqref{eq:TD} and critic loss using these samples of discretized trajectories through
\begin{equation}\label{eq:TD_numeric}
\TD_i = \sum_{j=0}^{N_t-1} r\parentheses{x^{(i)}_{jh}, u^\tau(jh, x^{(i)}_{jh})} h + g(x^{(i)}_T) - \mV^\tau_0(x^{(i)}_0) - \sum_{j=0}^{N_t-1} \mG^\tau(jh,x^{(i)}_{jh})\tp \sigma(x^{(i)}_{jh}) \xi^{(i)}_{jh}
\end{equation}
and  
\begin{equation}\label{eq:critic_loss_numeric}
\mL_c^\tau \approx \mL_c^{\tau, N} := \dfrac{1}{N} \sum_{i=1}^N \TD_i^2.
\end{equation}

Then, Adam method is applied to update the critic parameters $\theta_c$:
\begin{equation}\label{eq:critic_update}
\theta_c^{\tau + \Delta \tau} = \theta_c^\tau - \alpha_c \Delta \tau \, \text{Adam} \parentheses{\dfrac{\partial \mL_c^{\tau,N}}{\partial \theta}},
\end{equation}
where $\alpha_c$ is the learning speed for the critic, and $\text{Adam}(\cdot)$ denotes the Adam update direction.
\begin{rmk}\label{rmk:critic_BSDE}
The construction of the critic loss end up as a similar method to the deep BSDE method proposed in \cite{han2018solving}. In their work, the authors parametrize the gradient of the solution as a neural network and use a shooting method to match the terminal condition based on Feynman-Kac formulation. In comparison, we start with the TD from RL formalism and minimized the squared expectation of TD for the critic.
\end{rmk}

\subsection{Policy gradient for the actor} \label{sec:policy_gradient}
Policy gradient \citep{sutton1999policy} is a widely-used gradient based method to improve the parametrized policy strategy in RL,  which can be viewed as a local dynamic programming method.

Our starting point is the following proposition (with proof in Appendix \ref{sec:props}), which gives the explicit functional derivative (i.e., gradient) of the cost w.r.t.{} the control function.

\begin{prop}\label{prop:cost_derivative}
Let $b,r \in C^2(\mX \times \RR^{n'})$ and $g,\sigma \in C^2(\mX)$. Let $u \in C^{1,2}([0,T];\mX)$ be a control function and $V_u$ be the corresponding value function. Let $\rho^u(t, \cdot)$ be the density for the state process $x_t$ starting with uniform distribution in $\mX$ and following the SDE \eqref{eq:SDE_X} under control $u$. Then, for any $t, x \in \mX$
\begin{equation}\label{eq:actor_derivative}
\fd{J[u]}{u}(t, x)  = - \rho^u(t,x) ~\nbu G\parentheses{t, x, u(t, x), -\nx V_u(t,x)},
\end{equation}
where $\fd{}{u}$ denotes the $L^2$ first order variation w.r.t. the function $u(\cdot,\cdot)$. As a generalization, 
\begin{equation}\label{eq:dVdu}
\fd{V_u(s,y)}{u}(t, x)  = - \mathds{1}_{\{t \ge s\}} \, p^u(t,x;s,y) ~\nbu G\parentheses{t, x, u(t, x), -\nx V_u(t,x)}
\end{equation}
for all $t,s$ and $x,y \in \mX$, where $p^u(t,x;s,y)$ is the fundamental solution to the Fokker-Planck equation \eqref{eq:FokkerPlanck}.
\end{prop}

\begin{rmk}
Recall that $G$ is the generalized Hamiltonian defined in \eqref{eq:generalized_Hamiltonian}, with five arguments as inputs
\begin{equation*}
G\parentheses{t,x,u,p,P} = \frac12 \Tr\parentheses{P \sigma(x) \sigma(x)\tp} + \inner{p}{b(x,u)} - r(x,u)
\end{equation*}
and $\nbu G$ denotes the partial derivative of $G$ w.r.t.{} its third argument (as a vector in $\RR^{n'}$). In \eqref{eq:actor_derivative}, $\nbu G$ only has four inputs since in our setting that $\sigma$ not depending on $u$, $\nbu G$ does not depend on $P = -\nx^2 V_u(t,x)$. 
More generally, the diffusion coefficient $\sigma(x,u)$ also depends on the control, a more general version of Proposition~\ref{prop:cost_derivative} in such setting is considered in \cite{zhou2023policy}.
\end{rmk}

Let us denote $u^\tau$ the control function along the learning process, with $\tau$ being the algorithm time.
Motivated by Proposition \ref{prop:cost_derivative}, the idealized policy gradient dynamic is the $L^2$ gradient flow of $J[u]$, given by
\begin{equation}\label{eq:PG_ideal}
\dfrac{\rd}{\rd \tau} u^{\tau}(t, x) = - \fd{J}{u^{\tau}}(t, x) = \rho^{u^{\tau}}(t,x) ~\nbu G\parentheses{t, x, u^\tau(t, x), -\nx V_{u^\tau}(t,x)}.
\end{equation}
However, this requires instantaneous evaluation of the gradient of value function $\nx V_{u^\tau}$, which is usually not accessible. Therefore, in the actor-critic framework, we use the estimate of the gradient $\mG^\tau$ instead:
\begin{equation}\label{eq:actor_dynamic}
\dfrac{\rd}{\rd \tau} u^{\tau}(t, x) = \alpha_a \rho^\tau(t,x) ~\nbu G\parentheses{t, x, u^\tau(t, x), -\mG^\tau(t,x)},
\end{equation}
where we add $\alpha_a > 0$ as the learning rate for the actor, and we have used the shorthand $\rho^\tau$ for $\rho^{u^\tau}$. If we denote $\Delta \tau$ as the step size, then a gradient descent discretization for \eqref{eq:actor_dynamic} is
\begin{equation}\label{eq:actor_dynamic_GD}
u^{\tau+\Delta\tau}(t,x) = u^\tau(t,x) + \Delta \tau \alpha_a \rho^\tau(t,x) ~\nbu G\parentheses{t, x, u^\tau(t, x), -\mG^\tau(t,x)}.
\end{equation}

It is expensive to implement the gradient descent method \eqref{eq:actor_dynamic_GD} directly, because it involves solving the (high dimensional) Fokker-Planck equation for $\rho^\tau$. Therefore, in practice, we consider a stochastic version of gradient descent that updates the control function through
\begin{equation}\label{eq:actor_dynamic_SGD}
\begin{aligned}
u^{\tau + \Delta\tau}(t,x_t) 
& = u^\tau(t,x_t) + \Delta \tau \alpha_a \nbu G\parentheses{t, x_t, u^\tau(t, x_t), -\mG^\tau(t,x_t)},
\end{aligned}
\end{equation}
where $x_t \sim \rho^\tau(t, \cdot)$. Note that on average the update of the stochastic gradient is the same with \eqref{eq:actor_dynamic_GD}. Then, we can further use the scheme \eqref{eq:EM} to sample trajectories for $x_t$, which gives the set $\mS_\tau$ (see \eqref{eq:samplings}).

Let us parametrize the control function by $u(t,x; \theta_a)$, hence $u^\tau(t,x)$ is represented by $u(t,x;\theta_a^\tau)$. Then, the stochastic gradient descent \eqref{eq:actor_dynamic_SGD} could be further approximated through least square using the following loss function w.r.t. $\theta$ 
\begin{equation}\label{eq:LSreg}
\mL_a(\theta; \theta_a^\tau) = \sum_{(t,x) \in \mS_\tau} \abs{u(t,x;\theta) - u(t,x;\theta_a^\tau) - \Delta \tau \alpha_a\, \nbu G\parentheses{t,x,u(t,x;\theta_a^\tau),-\mG^{\tau}(t,x)}}^2.
\end{equation}

Consequently, the update for actor parameters is given by 
\begin{equation}\label{eq:actor_update}
\theta_a^{\tau+\Delta\tau} = \theta_a^\tau - \alpha_a \Delta \tau \, \text{Adam}\parentheses{\pd{\mL_a}{\theta}(\theta = \theta^\tau_a; \theta^\tau_a)}.
\end{equation}

We remark that there is an alternative approach of directly solving the optimal control problem. In this approach, we discretize the state dynamic \eqref{eq:SDE_X} and the objective function \eqref{eq:cost} directly and then apply optimization algorithm like Adam on the control parameter. This vanilla method doesn't even require policy evaluation and can be viewed as a discretize-optimize method.
In comparison, we give the expression for policy gradient \eqref{eq:actor_dynamic} first and then discretize it with \eqref{eq:actor_update}, so our algorithm is an optimize-discretize method. 
Our optimize-discretize method fits the continuous gradient flow better because we discretize the gradient flow directly. As Stuart points out, the guiding principle in numerical analysis is to avoid discretization until the last possible moment \cite{stuart2010inverse}. Onken and Ruthotto also studies this issue and find that a discretize-optimize method can achieve performance similar optimize-discretize method only with careful numerical treatment \cite{onken2020discretize}. We also verify that our optimize-discretize method gives better performance in the numerical example, as in Section~\ref{sec:example}.

We summarize the actor-critic algorithm in the pseudo-code Algorithm \ref{alg:AC}.
\begin{algorithm}[ht]
\caption{Actor-critic solver for the optimal control problem}\label{alg:AC}
\begin{algorithmic}
\Require Number of iterations $k_{end}$, learning rates $\alpha_a$, $\alpha_c$, step size $\Delta \tau$, batch size $N$, number of time intervals $N_T$
\Ensure value function and control 
\State initialization: $\theta_a$, $\theta_c$
\For{$k=0$ \textbf{to} $k_{end}-1$}
\State $\tau = k \Delta \tau$
\State For $i=1,\ldots, N$,  \Comment{sample state trajectories and get $\mS_\tau$}\\
\hspace{0.4in} sample $x^{(i)}_0 \sim \text{Unif}(\mX)$ and $x^{(i)}_{t+h} = x^{(i)}_t + b(x^{(i)}_t, u(t,x^{(i)}_t; \theta^\tau_a)) h + \sigma(x^{(i)}_t) \xi^{(i)}_t$ 
\State $ \TD_i = \sum_{j=0}^{N_t-1} r\parentheses{x^{(i)}_{jh}, u(jh, x^{(i)}_{jh};\theta_a^\tau)} h + g(x^{(i)}_T) $ \Comment{Compute the TD}\\
\hspace{0.6in} $-\mV^\tau_0(x^{(i)}_0) -\sum_{j=0}^{N_t-1} \mG^\tau(jh,x^{(i)}_{jh})\tp \sigma(x^{(i)}_{jh}) \xi^{(i)}_t $ % \Comment{Compute the TD}
% \State $\text{Grad}_c = \partial_{\theta_c} \dfrac{1}{N} \sum_{i=1}^N \TD_i^2$ 
\State $\theta_c^{\tau + \Delta \tau} = \theta_c^\tau - \alpha_c \Delta \tau \, \text{Adam}\parentheses{\partial_{\theta_c} \frac{1}{N} \sum_{i=1}^N \TD_i^2}$ \Comment{update critic parameter}
\medskip
% \State compute $\partial_\theta u(t,x) = \pd{u(t,x;\theta_a^\tau)}{\theta}$ for all $(t,x)\in\mS_\tau$  \Comment{actor steps}
% \State $\text{Grad}_a = \partial_{\theta_a'} \mL_a(\theta^\tau_a)$ \Comment{actor steps}
\State $\theta_a^{\tau+\Delta\tau} = \theta_a^\tau - \alpha_a \Delta \tau \, \text{Adam} \parentheses{\partial_{\theta} \mL_a(\theta^\tau_a;\theta^\tau_a)}$ \Comment{update actor parameter}
\EndFor
\end{algorithmic}
\end{algorithm}

% For the actor, we want to simulate the gradient flow \eqref{eq:actor_dynamic}. The sample-then-differentiate strategy is no longer our preferred choice because we have an explicit expression \eqref{eq:actor_derivative}, which can be approximated directly. 

\section{The actor-critic flow} \label{sec:flow}
Let $\Delta\tau \to 0$, then the discrete algorithm converges to the actor-critic flow with continuous time, for which we will carry out the analysis.
%Let us temporally omit the neural network parametrization, just like the tabular case in discrete RL, and focus on the continuous gradient flow.
First, we define the critic flow as the $L^2$-gradient flow of the critic loss $\mL_c$: 
\begin{equation}\label{eq:critic_dynamic}
\begin{aligned}
\dfrac{\rd}{\rd \tau} \mV_0^\tau(x) &= - \alpha_c \fd{\mL_c^\tau}{\mV_0}(x), \\
\dfrac{\rd}{\rd \tau} \mG^\tau(t,x) &= -\alpha_c \fd{\mL_c^\tau}{\mG}(t,x),
\end{aligned}
\end{equation}
which is the analog of the critic gradient descent \eqref{eq:critic_update} in continuous time. Here, we omit the neural network parametrization, and focus on the gradient flow of the two functions. Actually, the reason we use Adam method for optimization is that it is scale free \citep{Kingma2015adam}. This property ensures that different weight and bias parameters in a neural network are changing with a similar speed in training, which better approximates the gradient flow.

Note that we are treating $\mV_0$ and $\mG$ as two independent functions (with their own parameterization) in the gradient flow, without imposing any constraint for consistency, while the ground truth should satisfy $\nx \mV_0^\tau(x) = \mG^\tau(0,x)$. We will justify this treatment in Proposition \ref{prop:critic0} and Theorem \ref{thm:critic_improvement} below. To understand the critic flow, we observe that 
\begin{equation}\label{eq:critic_loss2}
\begin{aligned}
& \quad \mL_c^\tau = \frac12 \EE \sqbra{TD^2} \\
& = \frac12 \EE\sqbra{\parentheses{V_{u^\tau}(0,x_0) - \mV_0^\tau(x_0) + \int_0^T \parentheses{\nx V_{u^\tau}(t, x_t) - \mG^\tau(t,x_t)} \tp \sigma(x_t) \,\rd W_t}^2} \\
& = \frac12 \EE \sqbra{ \parentheses{\mV_0^\tau(x_0) - V_{u^\tau}(0, x_0)}^2 + \int_0^T \abs{ \sigma(x_t)\tp \parentheses{  \mG^\tau(t, x_t) - \nx V_{u^\tau}(t,x_t) } }^2 \rd t }\\
& = \frac12 \int_\mX \parentheses{\mV_0^\tau(x) - V_{u^\tau}(0,x)}^2 \,\rd x \\
& \qquad\qquad + \frac12 \inttx{\rho^{u}(t,x) \abs{\sigma(x)\tp \parentheses{\mG^\tau(t, x) - \nx V_{u^\tau}(t,x)}}^2} \\
& =: \mL_0^\tau + \mL_1^\tau,
\end{aligned}
\end{equation}
where we have substituted \eqref{eq:Ito} into \eqref{eq:TD} in the second equality, and we have applied It\^o's isometry in the third equality. Thus, the critic loss is decomposed into $\mL_0^\tau$ and $\mL_1^\tau$, which characterize the error for $\mV_0$ and $\mG$ respectively. This decomposition leads to the following proposition.

\begin{prop}\label{prop:critic0}
Suppose that $\mV_0^\tau$ and $\mG^\tau$ are two functions such that $\mL_c^\tau = \frac12 \EE[\TD^2] = 0$, where
\begin{equation*}
\TD =  \int_0^T r(x_t, u_t) \,\rd t + g(x_T) - \mV_0^\tau(x_0) - \int_0^T \mG^\tau(t, x_t)\tp \sigma(x_t) \,\rd W_t.
\end{equation*}
Then $\mV_0^\tau = V_{u^\tau}(0,\cdot)$ and $\mG^\tau = \nx V_{u^\tau}$.
\end{prop}
\begin{proof} According to \eqref{eq:critic_loss2}, $\mL_c^\tau=0$ implies $\mL_0^\tau=\mL_1^\tau=0$. Accordingly, $\mL_0^\tau=0$ indicates $\mV_0^\tau = V_{u^\tau}(0,\cdot)$. The uniform ellipticity of $\sigma(x)$ and $\mL_1^\tau=0$ implies $\mG^\tau = \nx V_{u^\tau}$. (Here, we used the fact that $\rho(t,x) > 0$, which will be justified in Proposition \ref{prop:rho}.)
\end{proof}

Moreover, the decomposition \eqref{eq:critic_loss2} also implies that the critic dynamic \eqref{eq:critic_dynamic} is equivalent to two separate $L^2$ gradient flow of $\mL_0^\tau$ and $\mL_1^\tau$ respectively.
Combining the critic \eqref{eq:critic_dynamic} and actor \eqref{eq:actor_dynamic} dynamics, we arrive at the actor-critic flow.
\begin{subequations}\label{eq:joint_dynamic}
\begin{align}
\dfrac{\rd}{\rd \tau} \mV_0^\tau(x) & = - \alpha_c \fd{\mL_0^\tau}{\mV_0}(x) = - \alpha_c \parentheses{\mV_0^\tau(x) - V_{u^\tau}(0,x)} \label{eq:V0dynamic}\\
\dfrac{\rd}{\rd \tau} \mG^\tau(t,x) & = - \alpha_c \fd{\mL_1^\tau}{\mG}(t,x) = - \alpha_c \, \rho^\tau(t,x) \,\sigma(x)\sigma(x)\tp \parentheses{\mG^\tau(t, x) - \nx V_{u^\tau}(t,x)} \label{eq:Gdynamic}\\
\dfrac{\rd}{\rd \tau} u^{\tau}(t, x) &= \alpha_a \rho^\tau(t,x) ~\nbu G\parentheses{t, x, u^{\tau}(t, x), -\mG^\tau(t,x)}.
\end{align}
\end{subequations}
We recall that $\rho^\tau = \rho^{u^\tau}$ is the density function for the state dynamic under control $u^\tau$, which satisfies the FP equation.

\section{Theoretical analysis for the actor-critic flow}\label{sec:theory}
In this section, we give theoretical analysis for the actor-critic flow \eqref{eq:joint_dynamic}, showing that it converges linearly to the optimum. Most of the proofs are deferred to the appendix. Let us start with some technical assumptions. We will give several remarks on these assumptions at the beginning of the appendix.

\begin{assump}\label{assump:basic}
Assume the followings hold.
\begin{enumerate}
\item $r$ and $b$ are smooth with $C^{4}(\mX \times \RR^{n'})$ norm bounded by some constant $K$. 
\item $g$ and $\sigma$ are smooth with $C^{4}(\mX)$ norm bounded by $K$.
\end{enumerate}
\end{assump}

\begin{assump}\label{assump:u_smooth} The parametrizations and their derivatives are bounded by $K$:
$$\norm{u^{\tau}}_{C^{2,4}([0,T]; \mX)}, \, \norm{\mV_0^\tau}_{C^2(\mX)} , \, \norm{\mG^\tau}_{C^{1,1}([0,T]; \mX)} \le K.$$
\end{assump}

Let us define a set 
$$\mU = \curlybra{u(t,x) ~\Big|~ u \text{ is smooth and} \norm{u}_{C^{2,4}([0,T]; \mX)} \le K} $$
to include all the regular control functions we consider.

We present a proposition about the density function next.

\begin{prop}\label{prop:rho}
Let Assumption \ref{assump:basic} hold. Let $u \in \mU$ and $\rho^u$ be the solution to the Fokker Planck equation \eqref{eq:FokkerPlanck} with initial condition $\rho^u(0,x) \equiv 1$. Then $\rho^u(t,x)$ has a positive lower bound $\rho_0$ and an upper bound $\rho_1$ that only depend on $n$, $T$, and $K$. $\abs{\nx\rho(t,x)}$ also has a upper bound $\rho_2$ that only depends on $n$, $T$, and $K$.
\end{prop}

Next, we will present the analysis for the actor-critic flow. We give our specific choice of parameter here. We set the speed ratio as 
\begin{equation}\label{eq:speed_ratio}
\dfrac{\alpha_a}{\alpha_c} = \kappa := \min\curlybra{\dfrac{\rho_0^2 \sigma_0}{C_5} ,\, \dfrac{\rho_0}{2C_c} ,\, \dfrac{c_c \rho_0 \sigma_0}{C_a} }
\end{equation}
where $\rho_0$ is the lower bound for $\rho^u(t,x)$ given in Proposition \ref{prop:rho}, $\sigma_0$ is the constant for uniform ellipticity. $C_5$, $C_c$, $c_c$, and $C_a$ are constants that only depend on $K, \sigma_0, n, n', m, T$, which will be introduced in Lemma \ref{lem:rho_speed}, Theorem \ref{thm:critic_improvement}, and Theorem \ref{thm:actor_improvement}.

\smallskip 

\begin{rmk}
The choice of a constant speed ratio indicates that our algorithm has single time scale \citep{chen2021single}, meaning that the critic and actor dynamics have speeds of the same order. The analysis of single time scale is more challenging than two time scale bilevel optimization \citep{dalal2018finite}, where we can view the slower dynamic as static when considering the faster one, and the faster dynamics reaches optimum when the slower dynamics is concerned \citep{hong2020two}. The difference of analysis for single- and two- time scale methods have been also reviewed in the context of actor-critic method for LQR in \cite{zhou2023single}.
\end{rmk}

Next, we present our main results, establishing linear convergence  of our actor-critic flow. We characterize the optimality of critical point for the flow first. We will assume that the HJB equation \eqref{eq:HJB} admits a unique solution $V^* \in C^{1,2}([0,T]; \mX)$ and the optimal control function $u^* \in\mU$. The existence, uniqueness, and regularity of HJB equations have been well studied, see for example \cite{yong1999stochastic, mou2019remarks}.

\smallskip 
\begin{theorem}[Critical point for the joint dynamic]
Assume that $G$ is concave in $u$. Let the actor-critic joint dynamic \eqref{eq:joint_dynamic} reaches critical point $\mV^\infty_0$, $\mG^\infty$, and $u^\infty$, then $u^\infty$ is the optimal control and $\mV^\infty_0$ and $\mG^\infty$ are the optimal value function (at $t=0$) and its gradient.
\end{theorem}
\begin{proof}
We denote $V_{u^\infty}$ the true value function w.r.t. $u^{\infty}$, i.e. the solution to the HJ equation \eqref{eq:HJ} with $u^\infty$. Since the critic dynamic is static, we know that $\dfrac{\rd}{\rd \tau} \mV^\tau_0=0$ and $\dfrac{\rd}{\rd \tau} \mG^\tau=0$. By \eqref{eq:V0dynamic}, $\mV^{\infty}_0 = V_{u^\infty}(0,\cdot)$. Next,  \eqref{eq:Gdynamic}, Proposition \ref{prop:rho}, and the uniform ellipticity assumption imply $\mG^\infty = \nx V_{u^\infty}$ is the true gradient. Since the actor dynamic \eqref{eq:actor_dynamic} is static, we have $\nbu G\parentheses{t,x,u^\infty(t,x), \mG^\infty(t,x)}=0$. So the concavity of $G$ in $u$ implies $u^\infty(t,x)$ maximizes $G(t,x,\cdot, \nx V_{u^\infty}(t,x), \nx^2 V_{u^\infty}(t,x))$. i.e. the maximize condition \eqref{eq:max1} is satisfied. Therefore, $V_{u^\infty}$ is also the solution of the HJB equation \eqref{eq:HJB}, i.e. the optimal value function.
\end{proof}

Next, we state two theorems for the convergence properties of the critic and the actor dynamics respectively.
\begin{theorem}[Critic improvement]\label{thm:critic_improvement}
Let assumptions \ref{assump:basic}, \ref{assump:u_smooth}, hold. Then
\begin{equation}\label{eq:critic_improvement}
\dfrac{\rd}{\rd \tau} \mL_c^\tau \le -2 \alpha_c \, c_c \, \mL_c^\tau + C_c \dfrac{\alpha_a^2}{\alpha_c} \inttx{\rho^\tau(t,x) \abs{\nbu G \parentheses{t,x,u^\tau(t,x), -\mG^\tau(t,x)}}^2}
\end{equation}
where $c_c= \min\curlybra{\frac12, \rho_0 \sigma_0}$ and $C_c =  \dfrac{C_2 \,\rho_1}{2} + \dfrac{\rho_1^2 K^2 C_2^2}{2 \rho_0 \sigma_0}$.
\end{theorem}
Here, $C_2$ is a constant that will be introduced in Lemma \ref{lem:regularity_Vu} in Appendix \ref{sec:lemmas}. 

\smallskip 
\begin{rmk}\label{rmk:critic_rate}
As a direct corollary for Theorem \ref{thm:critic_improvement}, if $\alpha_a=0$, then by Gronwall inequality, we have
\begin{equation*}
\mL_c^\tau \le \mL_c^0 \exp(-2\alpha_c c_c \tau).
\end{equation*}
This corresponds to a pure policy evaluation process with a fixed control, which shows that the critic flow \eqref{eq:critic_dynamic} has a linear convergence rate. The last term in \eqref{eq:critic_improvement} is an additional error due to the actor (control) dynamic.
\end{rmk}

Next, we present the convergence analysis for the policy gradient flow. We first state another assumption.

\smallskip 
\begin{assump}\label{assump:actor_rate}
There exists a modulus of continuity $\omega: [0,\infty) \to [0,\infty)$ such that
$$\norm{u-u^*}_{L^2} \le \omega(J[u] - J[u^*])$$
for any $u \in \mU$. Here $u^*$ is the optimal control.
\end{assump}

\smallskip 

\begin{theorem}[Actor improvement]\label{thm:actor_improvement}
Let Assumptions \ref{assump:basic}, \ref{assump:u_smooth}, \ref{assump:actor_rate} hold. Assume that $G$ is uniformly strongly concave in $u$. Then the actor dynamic satisfies
\begin{equation}\label{eq:actor_improvement}
\begin{aligned}
\dfrac{\rd}{\rd \tau} J[u^\tau] & \le - \alpha_a \, c_a \parentheses{J[u^\tau] - J[u^*]} + \alpha_a C_a \norm{\nx V_{u^\tau} - \mG^\tau}_{L^2}^2\\
& \quad - \frac12 \alpha_a \rho_0 \inttx{\rho^\tau(t,x) \abs{\nbu G\parentheses{t,x,u^\tau(t,x),-\mG^\tau(t,x)}}^2 \, } 
\end{aligned}
\end{equation}
for positive constants $c_a$ and $C_a=\frac12\rho_1^2 K^2$. The constant $c_a$ depends on $n$, $n'$, $m$, $T$, $K$, and $\omega$.
\end{theorem}
We clarify that by uniformly strongly concave, we mean there exists a constant $\mu_G>0$ such that the family of functions $G(t,x,\cdot,p,P)$ is $\mu_G-$strongly concave for all $(t,x,p,P)$ within the range given by Assumption \ref{assump:basic}. Note that even with such assumption, the optimal control problem itself is not convex, i.e., $J[\cdot]$ is not convex in $u$. Therefore, the linear convergence does not directly follow from convexity and is nontrivial. We also remark that our assumption on the concavity of $G$ is weaker than those imposed in \cite{reisinger2022linear}, which require that the running cost is sufficiently convex or the time span $[0,T]$ is sufficiently short. 

\smallskip 

\begin{rmk}
This result establishes that the policy is improving at a linear rate $\alpha_a c_a$, but with an additional term $\alpha_a C_a \norm{\nx V_{u^\tau} - \mG^\tau}_{L^2}^2$ coming from the critic error.
An explicit expression for $c_a$ will require more information, such as the explicit formula of $\omega(\cdot)$ in Assumption \ref{assump:actor_rate}.

Additionally, without Assumption \ref{assump:actor_rate}, we can still show that the actor-critic flow converges, but without a rate, i.e., $\lim_{\tau\to0} \mL^\tau = 0$ in Theorem \ref{thm:joint_rate}. We refer the reader to \cite[Theorem 2]{zhou2023policy} for details.
\end{rmk}

Finally, we combine Theorem \ref{thm:critic_improvement} and Theorem \ref{thm:actor_improvement} and present our main theorem, giving a global linear convergence for the joint actor-critic dynamic.

\smallskip 
\begin{theorem}[Convergence of the actor-critic flow]\label{thm:joint_rate}
Let assumptions \ref{assump:basic}, \ref{assump:u_smooth}, \ref{assump:actor_rate} hold. Assume that $G$ is uniformly strongly concave in $u$. Let
\begin{equation*}
\mL^\tau = (J[u^\tau] - J[u^*]) + \mL_c^\tau
\end{equation*}
be the sum of the actor gap and critic loss. Then,
\begin{equation}\label{eq:main_thm}
\mL^\tau \le \mL^0 \exp(-c \tau)
\end{equation}
with $c = \min\{ \alpha_c c_c, \alpha_a  c_a \} >0$, where $c_c$ and $c_a$ are given by Theorem \ref{thm:critic_improvement} and \ref{thm:actor_improvement} respectively.
\end{theorem}
\begin{proof}%[Proof for theorem \ref{thm:joint_rate}]
Since $\dfrac{\alpha_a}{\alpha_c} \le \dfrac{c_c \rho_0 \sigma_0}{C_a}$, we have
\begin{equation}\label{eq:thm4_temp}
\begin{aligned}
 \alpha_c \, c_c \, \mL_c^\tau & \ge \alpha_c \, c_c \, \mL_1^\tau \\
& = \frac12 \alpha_c c_c \inttx{\rho^{u^\tau}(t,x) \abs{\sigma(x)\tp \parentheses{\mG^\tau(t, x) - \nabla_x V_{u^\tau}(t,x)}}^2} \\
& \ge \alpha_c \, c_c \, \rho_0 \, \sigma_0 \inttx{ \abs{\parentheses{\mG^\tau(t, x) - \nabla_x V_{u^\tau}(t,x)}}^2} \\
& \ge \alpha_a C_a \norm{\nx V_{u^\tau} - \mG^\tau}_{L^2}^2,
\end{aligned}
\end{equation}
where we used the lower bound in Proposition \ref{prop:rho} and uniform ellipticity condition in the second inequality. Combining the results \eqref{eq:critic_improvement} in Theorem \ref{thm:critic_improvement} and \eqref{eq:actor_improvement} in Theorem \ref{thm:actor_improvement}, we obtain
\begin{equation*}
\begin{aligned}
& \quad \dfrac{\rd}{\rd \tau}\mL^\tau = \dfrac{\rd}{\rd \tau} \parentheses{J[u^\tau] + \mL_c^\tau} \\
& \le -2 \alpha_c \, c_c \, \mL_c^\tau + C_c \dfrac{\alpha_a^2}{\alpha_c} \inttx{\rho^\tau(t,x) \abs{\nbu G \parentheses{t,x,u^\tau(t,x), -\mG^\tau(t,x)}}^2} \\
& \quad - \alpha_a \, c_a \parentheses{J[u^\tau] - J[u^*]} + \alpha_a C_a \norm{\nx V_{u^\tau} - \mG^\tau}_{L^2}^2\\
& \quad - \frac12 \alpha_a \rho_0 \inttx{\rho^\tau(t,x) \abs{\nbu G\parentheses{t,x,u^\tau(t,x),-\mG^\tau(t,x)}}^2 \, }\\
& \le -2 \alpha_c \, c_c \, \mL_c^\tau - \alpha_a \, c_a \parentheses{J[u^\tau] - J[u^*]} + \alpha_a C_a \norm{\nx V_{u^\tau} - \mG^\tau}_{L^2}^2 \\
& \le - \alpha_c \, c_c \, \mL_c^\tau - \alpha_a \, c_a \parentheses{J[u^\tau] - J[u^*]} \le - c \mL^\tau
\end{aligned}
\end{equation*}
where the second inequality is because $\dfrac{\alpha_a}{\alpha_c} \le \dfrac{\rho_0}{2 C_c}$, and the third inequality is due to \eqref{eq:thm4_temp}. This establishes \eqref{eq:main_thm} and finishes the proof.
\end{proof}

% We comment that Proposition \ref{prop:cost_derivative}, \ref{prop:rho} and Theorem \ref{thm:actor_improvement} could be generalized to the setting with controlled diffusion where $\sigma$ also depends on $u$. For details, please refer to \cite{zhou2023policy}. The only restriction to extend Theorem \ref{thm:critic_improvement} to the controlled diffusion setting is the lack of Hessian estimate with theoretical guarantee for the critic. 

\section{Numerical examples}\label{sec:example}

In this section, we present two numerical examples to illustrate the effectiveness of our actor-critic algorithm. The first one is the widely studied LQ problem. The second example explores the diffusion form of Aiyagari's model in economic growth. The errors reported below are $L^2$ relative errors, calculated from the average of multiple runs. The default parameters for the PyTorch Adam optimizer are used for model training.

\subsection{The LQ problem}
We consider a LQ type example with $n=n'=m$ first. We lengthen the torus to $[0,2\pi]^n$ for simplicity of implementation in this example. Note that a simple rescaling could pull us back to the unit torus $[0,1]^n$. The objective is
\begin{equation*}
J[u] = \EE\sqbra{\int_0^T \parentheses{\frac12 \abs{u(t,x_t)}^2 + \widetilde{r}(x_t)} \,\rd t + g(x_T)},
\end{equation*}
subject to the dynamic
\begin{equation*}
\rd x_t = u(t,x_t) \,\rd t + \overline{\sigma} \,\rd W_t.
\end{equation*}
Here, $\widetilde{r}(x) = \beta_0 + \frac12 \sum_{i=1}^n \parentheses{\overline{\sigma}^2  \beta_i \sin(x_i) + \beta_i^2 \cos^2(x_i)}$
and $g(x) = \sum_{i=1}^n \beta_i \sin(x_i)$.
The HJB equation for this problem is
$$-\pt V + \sup_{u\in\RR^n}\sqbra{-\frac12 \overline{\sigma}^2 \Delta_x V - \nx V\tp u - \frac12 |u|^2 - \widetilde{r}(x)}=0,$$
with ground truth solution $V^*(t,x) = \beta_0(T-t) + \sum_{i=1}^n \beta_i \sin(x_i)$. In order to enforce the periodic boundary condition in the neural network, we firstly map the input $x$ into a set of trigonometric basis, and then feed them into a two-layer residual network, with ReLU as the activation function. Similar network structure is also used in \cite{han2020solving}.

The numerical results for this LQ problem in $1$ and $10$ dimensions are presented in Figure \ref{fig:LQ}. In the first row of the figure, three plots depict the value function $V(0,\cdot)$, its spatial derivative $\partial_x V(0,\cdot)$, and the control function $u(0,\cdot)$ in $1$ dimension LQ problem. In all three figures, we compare the true function with its neural network approximation, demonstrating the efficacy of our algorithm in capturing the ground truth solutions. The final $L^2$ relative errors for the initial value function $\mV_0$, gradient of value function $\mG$, and the control function $u$ are $2.25\%$, $2.80\%$, and $4.65\%$ respectively.

The first figure on the second row in Figure \ref{fig:LQ} plots the learning curves for this $1d$ LQ problem, using the logarithm of validation errors during the training. The shadow in the figure represents the standard deviation observed during multiple test runs. We observe a linear decay of errors at the beginning of the training, which coincides with the result in Theorem~\ref{thm:joint_rate}. In the later part of the training, the discretization error and sampling error dominate, leading to a deviation from a linear rate of convergence.

We also test the LQ example in $10$ dimensions. The errors for the three functions are $1.28\%$, $4.85\%$, and $4.30\%$ respectively. The second and third figures in the second row of Figure~\ref{fig:LQ} show the density plots of $V(0,\cdot)$ and the first dimension for $u(0,\cdot)$. Specifically, the second figure on row 2 illustrates the probability density function of $V(0,x_0)$ and its neural network approximation, where $x_0$ is uniformly distributed on the torus. The third figure is obtained in a similar way, but the output $u$ in multidimensional, we only plot the first dimension. The neural networks effectively capture the value function and the control function in this multidimensional setting. 

In order to further explain the numerical results, we compare the errors for $10$ dimensional LQ problem with two baselines. The first one is supervise learning. We use the $L^2$ errors of the three functions $V(0,\cdot)$, $\nx V$, and $u$ as the loss functions and train the neural networks with the same network structures and training scheme as our actor-critic algorithm. The supervise learning gives an approximation for the capacity of the networks. The errors for the three functions are $2.45\%$, $3.44\%$, and $4.72\%$, which is similar to the training errors reported above. This result suggests that our algorithm nearly reaches its capacity determined by the neural network architecture. The second comparison involves a vanilla gradient descent method (i.e., the discretize-optimize framework), where the cost \eqref{eq:cost} is discretized and optimized directly without any policy evaluation. We also use the same network structure and same training scheme for this comparison. The final error for the control is $5.39\%$, which is larger than our actor-critic method ($4.30\%$). Moreover, the final cost for our actor-critic method and a vanilla gradient descent method are $\num{1.90e-2}$ and $\num{2.11e-2}$, where our actor-critic method has about $10\%$ relative improvement. This result validates our argument in Section \ref{sec:policy_gradient} that our optimize-discrete is better than the vanilla discrete-optimize method.

% LQ10d
% Training errors. V0: 0.01278 Grad: 0.04845 u: 0.04303 J: 0.01899
% Net capacity errors. V0: 0.02453 Grad: 0.03438 u: 0.04723 J: 0.0
% Vanilla errors. V0: 0.0 Grad: 0.0 u: 0.05385 J: 0.02108

\begin{figure}[ht]
\centering
\includegraphics[width=0.325\textwidth]{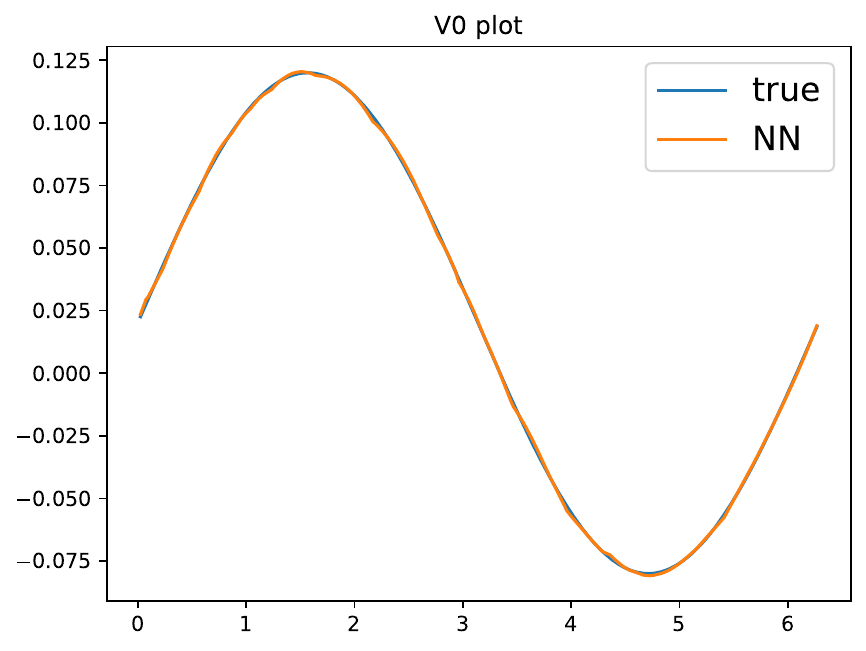}
\includegraphics[width=0.325\textwidth]{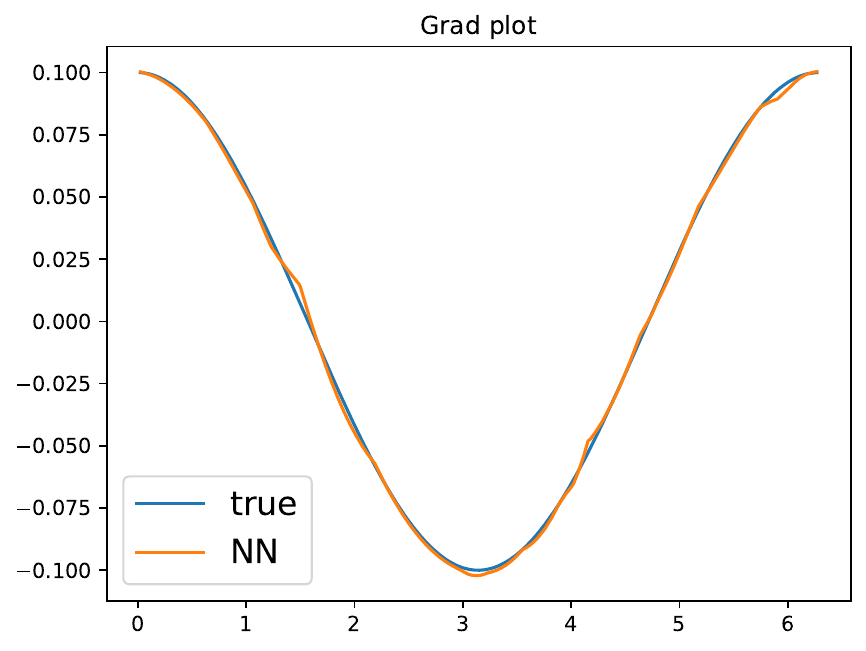}
\includegraphics[width=0.325\textwidth]{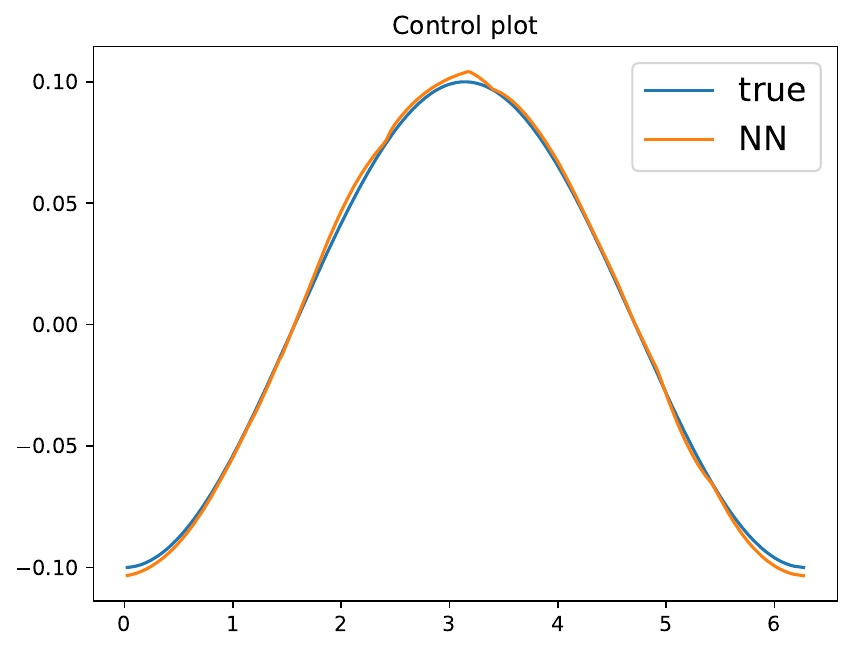}
\includegraphics[width=0.325\textwidth]{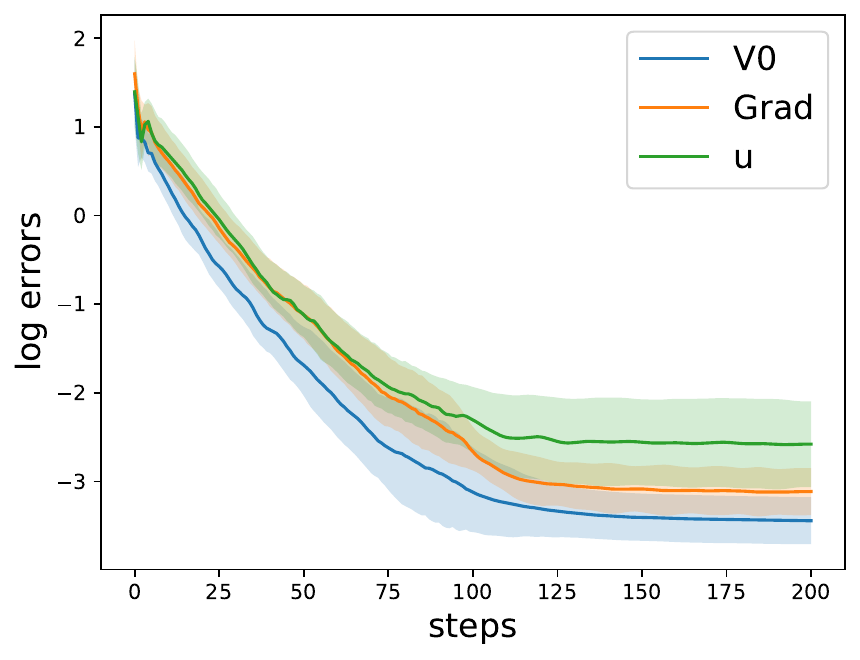}
\includegraphics[width=0.325\textwidth]{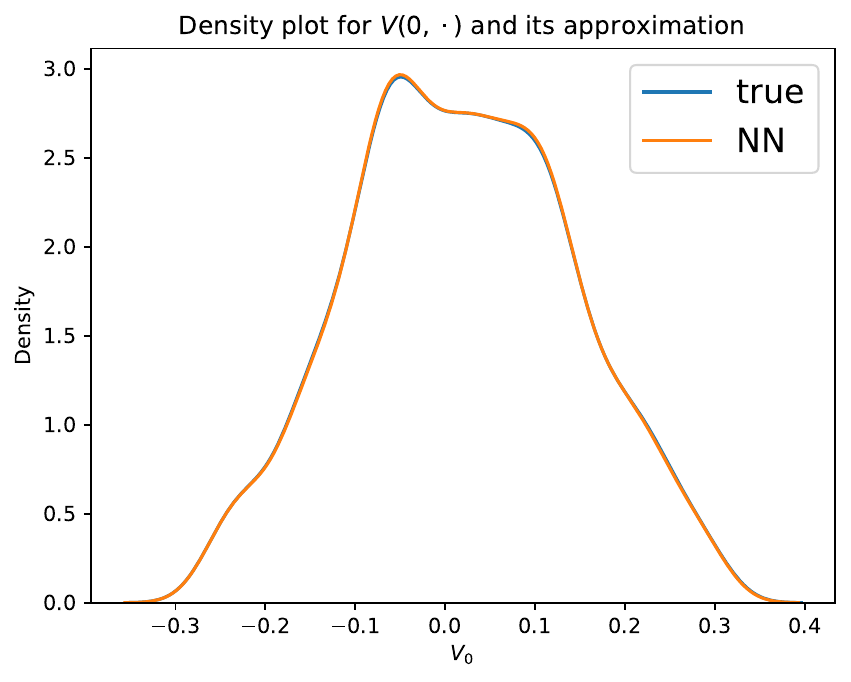}
\includegraphics[width=0.325\textwidth]{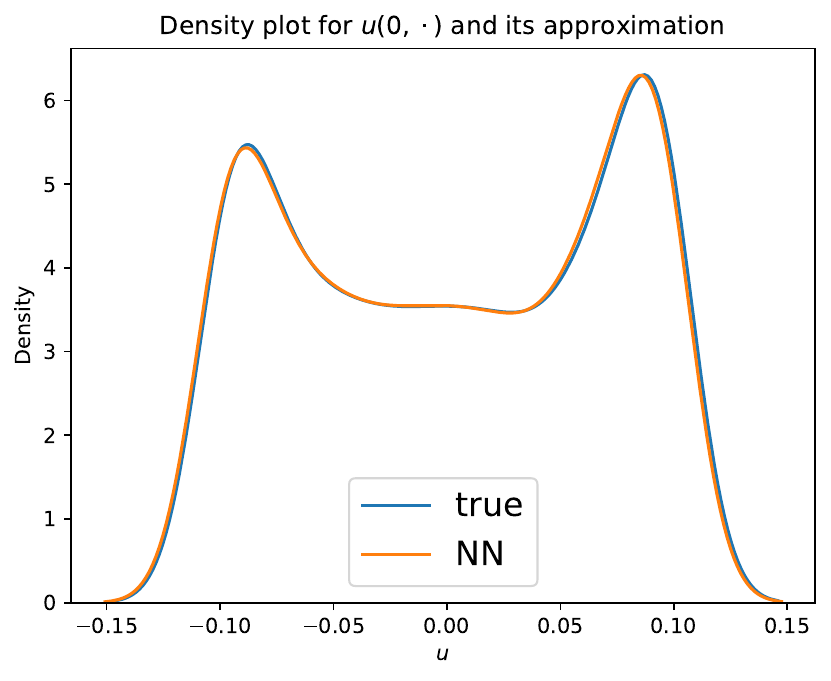}
\caption{Numerical results for the LQ problem. First line: plot for the value function, its spatial gradient and the control function at $t=0$ for $1d$ LQ problem. Each figure compare the true function with its neural network approximation. Second line: training curve for $1d$ LQ problem and the density plots of the value function and the control function for $10d$ LQ problem.}
\label{fig:LQ}
\end{figure}

\subsection{Aiyagari's growth model in economics}

In this section, we consider a modified version of the Aiyagari's economic model \citep{carmona2018probabilistic}. It describes the individual's choice in economy. The model is originally a mean field game problem, and we make some modification to suit our optimal control framework. Everything in this example is in $1$ dimension. The objective of an agent is to maximize the total utility
\begin{equation*}
\max_{c_t} \EE \sqbra{\int_0^T \parentheses{U(c_t) - r(Z_t)} \rd t + A_T - g(Z_T)}
\end{equation*}
subject to the state dynamics
\begin{align*}
\rd Z_t &= -(Z_t - 1) \,\rd t + \sigma_z \,\rd W_t \\
\rd A_t &= \parentheses{(1-\alpha)Z_t + (\alpha - \delta) A_t - c_t} \rd t + \sigma_a A_t \,\rd W_t.
\end{align*}
$Z_t$ denotes the labor productivity, which follows a Ornstein–Uhlenbeck process with mean $1$ and diffusion coefficient $\sigma_z$. $A_t$ represents the total asset of the agent, whose dynamic depends on the income $(1-\alpha)Z_t$, the total asset $A_t$, the consumption $c_t$ that the agent can control, and some random perturbation with volatility $\sigma_a$. Here, $\alpha$ is the tax rate and $\delta$ is the capital depreciation rate. $U(c)$ is the consumption utility function that is increasing and concave. In this example, we choose $U(c) = \log(c)$. $r(z)$ and $g(z)$ are increasing functions that describe the tiredness of the agent during and at the end of the time period.

With a trivial sign shift, we transform the objective into a minimization problem to suit our setting
\begin{equation*}
\min_{c_t} \EE \sqbra{\int_0^T \parentheses{-U(c_t) + r(Z_t)} \rd t - A_T + g(Z_T)}.
\end{equation*}
The corresponding HJB equation is
\begin{multline*}
-V_t + \sup_{c \in \RR} \biggl[-\frac12 \sigma_z^2 V_{zz} - \frac12 \sigma_a^2 a^2 V_{aa} + (z-1)V_z \\ + (c-(1-\alpha)z - (\alpha-\delta)a)V_a + U(c) - r(z)\biggr] = 0
\end{multline*}
with terminal condition $V(T,z,a)=g(z)-a$. In the numerical implementation, we set $\alpha=\delta=0.05$, $\sigma_z=1.0$, $\sigma_a=0.1$ and $g(z)=0.2\,e^{0.2z}$.

We consider Euclidean space instead of a torus in this instance for practical reasons. Consequently, we will not use trigonometric basis in this example. This choice demonstrates the versatility of our algorithm in handling more general settings, whereas our theoretical analysis was conducted on a torus for technical simplicity.

% There are several reasons we choose this example. Firstly, this problem originates from economics, giving it practical meaning. Secondly, many research only studies LQ problems numerical, while this example provides a non-LQ structure, broadening the scope of our research in the numerical aspect. Thirdly, the diffusion coefficient is not constant but depends on the state $A_t$, making the control problem more challenging.

The numerical results are shown in Figure \ref{fig:Aiyagari}. The first figure shows the plot for $V(0,\cdot)$ and its neural network approximation. The second figure shows the plot of the control function when we fix $t=0$, $z=1$ and let $a$ varies. We observe that the neural network approximations capture the shapes of the value function and the control function well. The final errors for $\mV_0$,  $\mG$, and $u$ are $0.54\%$, $1.41\%$, and $1.56\%$ respectively.
% errors: V0: 0.00542 Grad: 0.01414 u: 0.0156 

\begin{figure}[ht]
\centering
\includegraphics[width=0.35\textwidth]{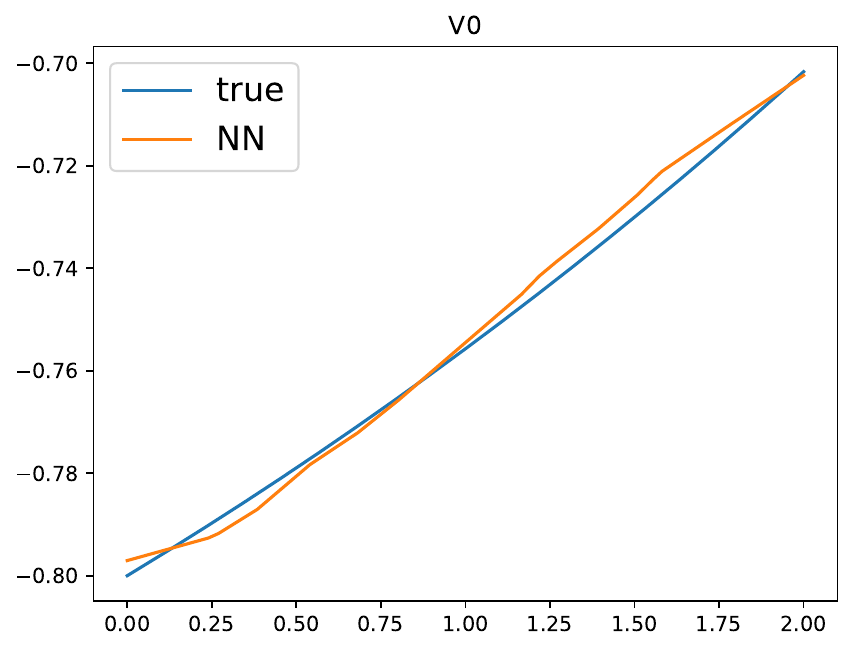}
\includegraphics[width=0.335\textwidth]{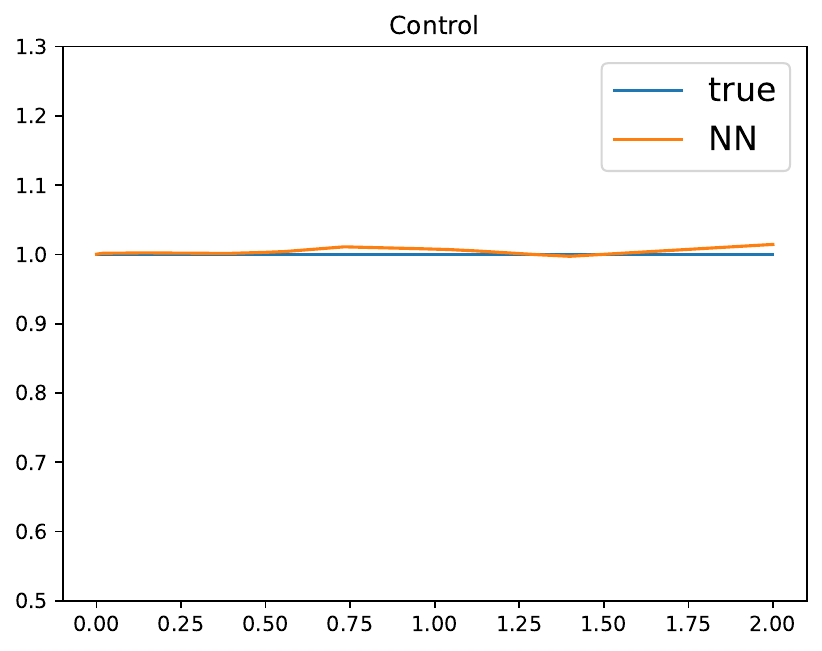}
\caption{Numerical results for the Aiyagari's example. The two figures show the plot of the value function and control function. Each figure compares the true function with its neural network approximation.}
\label{fig:Aiyagari}
\end{figure}

We remark finally that we did not pursue  extensive parameter tuning in our numerical test, and we believe that a more careful design of networks and hyperparameter tuning would further reduce errors. We use the same network structures and same training scheme to make a fair comparison.
We summarize all hyperparameters used in our examples in Table \ref{tab:para}.
\begin{table}[ht]
\centering
\begin{tabular}{c|ccc}
\midrule
&  LQ $1d$&  LQ $10d$& Aiyagari\\
\midrule
learning rate (actor, critic)& $(0.05,0.1)$ & $(0.01,0.05)$ & $(0.01,0.02)$\\
number of iterations& $200$ & $1500$ & $500$ \\
batch size& $500$ & $2000$ & $1000$ \\
time step size $(\Delta t, \Delta \tau)$ & $(0.01,0.5)$ & $(0.01,0.5)$ & $(0.01,0.5)$\\
\midrule
\end{tabular}
\caption{Parameters for the numerical examples.}
\label{tab:para}
\end{table}

\section{Conclusion and future directions}\label{sec:future}

In this paper, we propose an actor-critic framework to solve the optimal control problem. We apply a LSTD method for policy evaluation in the critic and a policy gradient method for the actor. We prove a linear rate of convergence for the actor-critic flow and validate our algorithm with numerical examples.

There are several interesting topics arising from gradient flow and numerical scheme worth further studying. For example, the consistency between the discretized gradient descent and the continuous gradient flow is crucial for ensuring the effectiveness of the algorithm. Additionally, the theoretical analysis for the approximation error of neural networks is an active research area. While our work focuses on convergence analysis for the gradient flow at the continuous time level, numerous other significant topics remain open for future research.

It's noteworthy that Theorem \ref{thm:critic_improvement} still holds with controlled diffusion $\sigma=\sigma(x_t,u_t)$. However, implementing our policy gradient flow becomes challenging under such conditions because $\nbu G$ involves $\nx^2 V_{u^\tau}$ while we do not have estimates for this Hessian. Although using the Jacobian of $\mG^\tau$ as an estimate for $\nx^2 V_{u^\tau}$ is numerically feasible, it lacks theoretical guarantees. Consequently, a promising avenue for future research is the development of a PDE solver with theoretical guarantees for up to second-order derivatives. If the Hessian also has a linear convergence rate as in Remark \ref{rmk:critic_rate}, then our analysis for the actor-critic flow is applicable to the setting with controlled diffusion. 

\textbf{Acknowledgements} The work of J. Lu is partially supported by National Science Foundation under award NSF DMS-2309378.
M. Zhou is partially supported by AFOSR YIP award No. FA9550-23-1-0087, AFOSR MURI FA9550-21-1-0084, AFOSR MURI FA9550-18-502, and N00014-20-1-2787.

\section*{Declarations}

\textbf{Conflict of interest} The authors declare no conflict of interest.

\noindent \textbf{Code and data availability} Our code is available at \url{https://github.com/MoZhou1995/ActorCriticControl}. All data used in this study were randomly generated and can be reproduced directly using the provided code.

\begin{appendices}

\vspace{0.15in}

\hspace{-0.3in} {\Large \textbf{Appendix}}

\vspace{0.15in}

We collect proofs for the propositions, lemmas and theorems in the appendix. We prove Proposition \ref{prop:cost_derivative} and \ref{prop:rho} in Section \ref{sec:props}. Then, we state and prove some auxiliary lemmas in Section \ref{sec:lemmas}. Next, we prove Theorem \ref{thm:critic_improvement} and \ref{thm:actor_improvement} in Section \ref{sec:thms}.

Before proceeding, we give a few comments on Assumptions \ref{assump:basic} - \ref{assump:actor_rate} first.

Recall that $D(x) = \frac12 \sigma(x) \sigma(x)\tp$, so $D$ is also smooth by our assumption on $\sigma$, and we still use $K$ to denote its $C^{4}(\mX)$ bound. Since the control function $u(t,x) \in \mU$ is bounded, we just need $r(x,u)$ and $b(x,u)$ has bounded derivative when their input $u \in \RR^{n'}$ lies in such bounded domain. A similar assumption can also be found in \cite{vsivska2020gradient}. Also, when the regularities in Assumption \ref{assump:basic} hold and $u \in \mU$, the solution $V_u$ to the HJ equation \eqref{eq:HJ} lies in $C^{1,2}([0,T]; \mX)$ and has classical derivatives. 
Similarly, the solution $V^*$ to the HJB equation lies in $C^{1,2}([0,T]; \mX)$ thanks to \cite{mou2019remarks} Theorem 5.3. This smoothness further implies that $V^*$ is a classical solution to the HJB equation. 
These classical solutions possess higher regularity thanks to Schauder's estimate \citep{ladyvzenskaja1988linear}, which informs us that $V_u \in C^{1,4}([0,T]; \mX)$ when Assumption \ref{assump:basic} holds and $u \in \mU$. Then, we observe that $G\parentheses{t, x, u(t,x), -\nx V_u(t,x), -\nx^2 V_u(t,x)}$ in \eqref{eq:HJ} is differentiable in $t$, which implies that $V_u \in  C^{2,4}([0,T]; \mX)$. We will also use $K$ to denote the bound for $C^{2,4}([0,T]; \mX)$ norm of $V_u$. The same argument also holds for $V^*$, so $\norm{V^*}_{C^{2,4}} \le K$.

We make boundedness assumptions for  $r$, $b$, $\sigma$, and $g$ because $\mX$ is compact. In the Euclidean space $\RR^n$, we usually assume that these functions has linear growth in $\abs{x}$. The boundedness assumption for the derivatives also gives us Lipschitz conditions thanks to the mean-value theorem. For example, boundedness of $\abs{\nx b(x,u)}$ implies Lipschitz condition of $b$ in $x$. To make the proofs more reader friendly, we will use $L$ instead of $K$ whenever we use the Lipschitz condition. Throughout the proof, we use $C$ to denote a constant that only depend on $n$, $T$, and $K$, which may change from line to line.

\section{Proofs for the Propositions}\label{sec:props}
\begin{proof}[Proof of Proposition \ref{prop:cost_derivative}]
We fix an arbitrary smooth perturbation function $\phi(t,x)$ (note that smooth functions are dense in $L^2$ space), then
\begin{equation*}
    \dfrac{\rd}{\rd \ve} J[u + \ve \phi] \bigg|_{\ve=0} = \inner{\fd{J}{u}}{\phi}_{L^2}.
\end{equation*}
We denote $x^{\ve}_t$ the SDE \eqref{eq:SDE_X} under control function $u^{\ve} := u + \ve\phi$ that start with $x^{\ve}_0 \sim \text{Unif}(\mX)$. Let $\rho^{\ve}(t,x)$ be its density. We also denote the corresponding value function by $V^{\ve}(t,x) := V_{u^{\ve}}(t,x)$ for simplicity. Then, $\rho^{\ve}$ and $V^{\ve}$ depend continuously on $\ve$ (see \cite{yong1999stochastic} Section 4.4.1). Through further observation on the FP equation (3) the HJ equation (6), we find that $\rho^{\ve}(t,x)$ and $V^{\ve}(t,x)$ is differentiable in $\ve$ for any $(t,x)$. Then, the differentiability of $J[u+\ve\phi]$ in $\ve$ follows from the smooth dependence of $J$ on $u$. By definition of the cost functional \eqref{eq:cost},
\begin{equation*}
\begin{aligned}
    J[u + \ve\phi] & = \EE\sqbra{\int_0^T r(x^{\ve}_t, u^{\ve}(t,x^{\ve}_t)) \rd t + g(x^{\ve}_T)} \\
    & = \int_0^T \int_{\mX} r(x, u^{\ve}(t, x)) \rho^{\ve}(t ,x) \,\rd x \, \rd t + \int_{\mX} V^{\ve}(T, x) \rho^{\ve}(T ,x) \,\rd x.
\end{aligned}
\end{equation*}
Taking derivative w.r.t. $\ve$, and note that $V^{\ve}(T,x)=g(x)$ does not depend on $\ve$, we obtain
\begin{equation}\label{eq:dJdep}
\begin{aligned}
    &\dfrac{\rd}{\rd \ve} J[u + \ve \phi] = \int_{\mX} V^{\ve}(T, x) \pve \rho^{\ve}(T ,x) \,\rd x + \\
    & + \int_0^T \int_{\mX} \sqbra{\nbu r\parentheses{x, u^{\ve}(t, x)}\tp \phi(t,x) \rho^{\ve}(t ,x) + r\parentheses{x, u^{\ve}(t, x)} \pve \rho^{\ve}(t ,x)} \,\rd x \, \rd t.
\end{aligned}
\end{equation}
In order to compute $V^{\ve}(T, x) \pve \rho^{\ve}(T ,x)$ in \eqref{eq:dJdep}, we write down the integral equation
\begin{equation}\label{eq:VT_drho}
    V^{\ve}(T,x) \rho^{\ve}(T,x) = V^{\ve}(0,x) \rho^{\ve}(0,x) + \int_0^T \sqbra{\partial_t \rho^{\ve}(t,x) V^{\ve}(t,x) + \rho^{\ve}(t,x) \partial_t V^{\ve}(t,x)} \rd t.
\end{equation}
We also have
\begin{equation}\label{eq:prop_dJdu_temp}
\begin{aligned}
& \quad \pve V^{\ve}(0,x) \rho^{\ve}(0,x) + \int_0^T \sqbra{ \partial_t \rho^{\ve}(t,x) \pve V^{\ve}(t,x) + \rho^{\ve}(t,x) \pve \partial_t V^{\ve}(t,x)} \rd t \\
& = \pve V^{\ve}(T,x) \rho^{\ve}(T,x) = \pve g(x) \rho^{\ve}(T,x) = 0.
\end{aligned}
\end{equation}
Next, taking derivative of \eqref{eq:VT_drho} w.r.t. $\ve$ (note that $V^{\ve}(T,x)=g(x)$ and $\rho^{\ve}(0,\cdot) \equiv 1$ do not depend on $\ve$), we obtain
\begin{equation}\label{eq:Vdrhodep}
\begin{aligned}
& \quad V^{\ve}(T, x) \, \pve \rho^{\ve}(T ,x) \\
& = \pve V^{\ve}(0,x) \rho^{\ve}(0,x) + \int_0^T \partial_{\ve} \sqbra{ \partial_t \rho^{\ve}(t,x) V^{\ve}(t,x) + \rho^{\ve}(t,x) \partial_t V^{\ve}(t,x) } \rd t \\
& = \int_0^T \sqbra{\pve \partial_t \rho^{\ve}(t,x) V^{\ve}(t,x) + \pve \rho^{\ve}(t,x) \partial_t V^{\ve}(t,x) } \rd t \\
& = \int_0^T \sqbra{\pve \partial_t \rho^{\ve}(t,x) V^{\ve}(t,x) + \pve \rho^{\ve}(t,x) G\parentheses{t, x, u(t, x), -\nx V^{\ve}(t,x), -\nx^2 V^{\ve}(t,x)} } \rd t
\end{aligned}
\end{equation}
where we used \eqref{eq:prop_dJdu_temp} and the HJ equation \eqref{eq:HJ} in the second and third equality respectively.
Substitute \eqref{eq:Vdrhodep} into \eqref{eq:dJdep}, we obtain
\begin{equation}\label{eq:dJdve1}
\begin{aligned}
& \quad \dfrac{\rd}{\rd \ve} J[u + \ve \phi] \\
& = \int_0^T \int_{\mX} \left[\pve \rho^{\ve}(t,x) G\parentheses{t, x, u(t, x), -\nx V^{\ve}(t,x), -\nx^2 V^{\ve}(t,x)}\right.\\
& \hspace{0.5in} \left. + \pve \partial_t \rho^{\ve}(t,x) V^{\ve}(t,x)\right] \,\rd x \, \rd t\\
& \quad + \int_0^T \int_{\mX} \sqbra{\nbu r\parentheses{x, u^{\ve}(t, x)}\tp \phi(t,x) \rho^{\ve}(t ,x) + r\parentheses{x, u^{\ve}(t, x)} \pve \rho^{\ve}(t ,x)} \,\rd x \, \rd t
\end{aligned}
\end{equation}
Taking derivative of the Fokker Planck equation \eqref{eq:FokkerPlanck} w.r.t. $\ve$, we obtain
\begin{equation}\label{eq:dFKdve}
\begin{aligned}
    & \quad \pve \partial_t \rho^{\ve}(t,x)\\
    & = -\nx \cdot \sqbra{\nbu b\parentheses{x, u^{\ve}(t, x)} \phi(t,x) \rho^{\ve}(t,x) + b\parentheses{x, u^{\ve}(t, x)} \pve \rho^{\ve}(t,x) } \\
    & \quad + \sum_{i,j=1}^n \partial_i \partial_j \sqbra{D_{ij} \parentheses{x} \pve \rho^{\ve}(t,x)}
\end{aligned}
\end{equation}
Substituting \eqref{eq:dFKdve} into \eqref{eq:dJdve1}, we get
\begin{equation*}
\begin{aligned}
& \quad \dfrac{\rd}{\rd \ve} J[u + \ve \phi] = \int_0^T \int_{\mX} \left\{ \pve\rho^{\ve}(t,x) G\parentheses{t, x, u(t, x), -\nx V^{\ve}(t,x), -\nx^2 V^{\ve}(t,x)} \right.\\
& + V^{\ve}(t,x) \{ -\nx \cdot \sqbra{\nbu b\parentheses{x, u^{\ve}(t, x)} \phi(t,x) \rho^{\ve}(t,x) + b\parentheses{x, u^{\ve}(t, x)} \pve \rho^{\ve}(t,x) } \\
& \hspace{0.6in} + \sum_{i,j=1}^n \partial_i \partial_j \sqbra{D_{ij} \parentheses{x} \pve \rho^{\ve}(t,x)} \}\\
& + \left. \sqbra{\nbu r\parentheses{x, u^{\ve}(t, x)}\tp \phi(t,x) \rho^{\ve}(t ,x) + r\parentheses{x, u^{\ve}(t, x)} \pve \rho^{\ve}(t ,x)} \right\}\rd x \, \rd t.
\end{aligned}
\end{equation*}
Applying integration by part in $x$, we get 
\begin{equation*}
\begin{aligned}
& \quad \dfrac{\rd}{\rd \ve} J[u + \ve \phi] = \int_0^T \int_{\mX} \left\{ \pve\rho^{\ve}(t,x) G\parentheses{t, x, u(t, x), -\nx V^{\ve}(t,x), -\nx^2 V^{\ve}(t,x)} \right.\\
& \hspace{0.4in} +\nx V^{\ve}(t,x)\tp \sqbra{\nbu b\parentheses{x, u^{\ve}(t, x)} \phi(t,x) \rho^{\ve}(t,x) + b\parentheses{x, u^{\ve}(t, x)} \pve \rho^{\ve}(t,x) } \\
& \hspace{0.4in} + \sum_{i,j=1}^n \partial_i \partial_j V^{\ve}(t,x) \sqbra{D_{ij} \parentheses{x} \pve \rho^{\ve}(t,x)}\\
& \hspace{0.4in} + \left. \sqbra{\nbu r\parentheses{x, u^{\ve}(t, x)}\tp \phi(t,x) \rho^{\ve}(t ,x) + r\parentheses{x, u^{\ve}(t, x)}\tp \pve \rho^{\ve}(t ,x)} \right\}\rd x \, \rd t.
\end{aligned}
\end{equation*}
Making a rearrangement, we get 
\begin{equation*}
\begin{aligned}
& \quad \dfrac{\rd}{\rd \ve} J[u + \ve \phi]  = \int_0^T \int_{\mX} \bigg\{ \pve\rho^{\ve}(t,x) \Big[ G\parentheses{t, x, u(t, x), -\nx V^{\ve}(t,x), -\nx^2 V^{\ve}(t,x)} \\
& \hspace{0.4in} + \nx V^{\ve}(t,x)\tp b\parentheses{x, u^{\ve}(t, x)} + \sum_{i,j=1}^n \partial_i \partial_j V^{\ve}(t,x) D_{ij} \parentheses{x} + r\parentheses{x, u^{\ve}(t, x)} \Big] \\
& \hspace{0.4in} + \rho^{\ve}(t,x) \sqbra{ \nx V^{\ve}(t,x)\tp \nbu b\parentheses{x, u^{\ve}(t, x)} + \nbu r\parentheses{x, u^{\ve}(t, x)}\tp }\tp  \phi(t,x) \bigg\} \rd x \, \rd t.
\end{aligned}
\end{equation*}
Therefore, by the definition of $G$ in \eqref{eq:generalized_Hamiltonian},
\begin{equation*}
\begin{aligned}
& \quad \dfrac{\rd}{\rd \ve} J[u + \ve \phi] \\
& = \int_0^T \int_{\mX} \sqbra{\pve\rho^{\ve}(t,x) \cdot [0] - \rho^{\ve}(t,x) \nbu G\parentheses{t, x, u^{\ve}(t, x), -\nx V^{\ve}(t, x)}\tp \phi(t,x)} \rd x \, \rd t\\
& = -\int_0^T \int_{\mX} \rho^{\ve}(t,x) \nbu G\parentheses{t, x, u^{\ve}(t, x), -\nx V^{\ve}(t, x)}\tp \phi(t,x) \, \rd x \, \rd t.
\end{aligned}
\end{equation*}
Let $\ve = 0$, we get
\begin{equation}\label{eq:dJdve}
\dfrac{\rd}{\rd \ve} J[u + \ve \phi] \bigg|_{\ve=0} = - \int_0^T \int_{\mX} \rho(t,x) \nbu G\parentheses{t, x, u(t, x), -\nx V_u(t, x)} \phi(t,x) \, \rd x \, \rd t.
\end{equation}
Therefore,
\begin{equation*}
\fd{J}{u}(t, x)  = - \rho(t,x) ~\nbu G\parentheses{t, x, u(t, x), -\nx V(t,x) }.
\end{equation*}
i.e. \eqref{eq:actor_derivative} holds.

We remark that the proposition still holds with controlled diffusion, we just need to track $\nbu D$. See \cite{zhou2023policy} for details.

The proof for \eqref{eq:dVdu} is almost the same. Firstly, changing the control function at $t < s$ will not affect $V^u(s,\cdot)$ by definition, so we just need to show \eqref{eq:dVdu} when $t \ge s$. We recall the definition of value function \eqref{eq:value}
\begin{equation*}
\begin{aligned}
V_u(s,y) &= \EE\sqbra{\int_{s}^T r(x_t, u_t) \,\rd t + g(x_T) \;\Big|\; x_s=y} \\
& = \int_s^T \int_{\mX} r(x,u(t,x)) p^u(t,x;s,y) \, \rd x \, \rd t + \int_{\mX} g(x) p^u(T,x;s,y) \,\rd x.
\end{aligned}
\end{equation*}
Here, $p^u(t,x;s,y)$ is the fundamental solution of the Fokker Planck equation \eqref{eq:FokkerPlanck}, so $p^u(t,x;s,y)$, as a function of $(t,x)$, is the density of $x_t$ starting at $x_s=y$. Therefore, we only need to repeat the argument to prove \eqref{eq:actor_derivative}. The only caveat we need to be careful is that $p^{\ve}(s,\cdot;s,y) = \delta_y$, so the classical derivative does not exist. This is not an essential difficulty because we can pick an arbitrary smooth probability density function $\psi(y)$ on $\mX$ and compute
\begin{equation}\label{eq:general_derivative}
\frac{\rd}{\rd \ve} \int_{\mX} V^{\ve}(s,y) \psi(y) \, \rd y \Big|_{\ve=0}.
\end{equation}
For example, when $s=0$ and $\psi \equiv 1$, \eqref{eq:general_derivative} becomes 
\begin{equation*}
\frac{\rd}{\rd \ve} \int_{\mX} V^{\ve}(0,y) \, \rd y \Big|_{\ve=0} = \frac{\rd}{\rd \ve} J[u^{\ve}] \Big|_{\ve=0}.
\end{equation*}
Therefore, we can repeat the argument to prove \eqref{eq:dJdve} and get
\begin{equation}\label{eq:general_derivative2}
\begin{aligned}
& \quad \frac{\rd}{\rd \ve} \int_{\mX} V^{\ve}(s,y) \psi(y) \, \rd y \Big|_{\ve=0} \\
& = - \int_s^T \int_{\mX} \rho^{u,s,\psi}(t,x) \nbu G\parentheses{t, x, u(t, x), -\nx V_u(t,x)} \phi(t,x) \, \rd x \, \rd t.
\end{aligned}
\end{equation}
where $\rho^{u,s,\psi}(t,x) := \int_{\mX} p^u(t,x;s,y) \psi(y) \, \rd y$ is the solution to the Fokker Planck equation with initial condition $\rho^{u,s,\psi}(s,x) = \psi(x)$. It is also the density function of $x_t$, which starts with $x_s \sim \psi$. The only difference between proving \eqref{eq:dJdve} and \eqref{eq:general_derivative2} is that we need to replace $\int_0^T$ by $\int_s^T$, replace $\rho^{\ve}(t,x)$ by $\rho^{\ve,s,\psi}(t,x) := \rho^{u^{\ve},s,\psi}(t,x)$, and replace $\rho^{\ve}(0,x)$ by $\rho^{\ve,s,\psi}(s,x)$. Therefore,
\begin{equation*}
\dfrac{\delta \parentheses{ \int_{\mX} V_u(s,y) \psi(y) \, \rd y}}{\delta u}(t,x) = - \int_{\mX} p^u(t,x;s,y) \psi(y) \, \rd y \, \nbu G\parentheses{t, x, u(t, x), -\nx V_u(t,x)}.
\end{equation*}
Hence, \eqref{eq:dVdu} holds.
\end{proof}

\begin{proof}[Proof for Proposition \ref{prop:rho}]
The Fokker Planck equation has been well-studied. Let $p^u(t,x;s,y)$ denote the fundamental solution to \eqref{eq:FokkerPlanck}. \cite{aronson1967bounds} found that the fundamental solution of a linear parabolic equation can be upper and lower bounded by fundamental solutions of heat equation (i.e. Gaussian functions) with different thermal diffusivity constant. For example, let $\widetilde{p}^u(t,x;s,y)$ be the fundamental solution of the Fokker Planck equation \eqref{eq:FokkerPlanck} in $\RR^n$ (where $b$ and $\sigma$ are extended periodically), then
\begin{equation}\label{eq:prop_rho_temp}
C^{-1} (t-s)^{-n/2} \exp\Bigl(-\dfrac{C \abs{x-y}^2}{t-s}\Bigr) \le \widetilde{p}^u(t,x;s,y) \le C (t-s)^{-n/2} \exp\Bigl(-\dfrac{C^{-1} \abs{x-y}^2}{t-s}\Bigr)
\end{equation}
for all $s<t\le T$ and $x,y \in \RR^n$, where $C$ only depends on $n$, $T$, and $K$. We are in the unit torus $\mX$ instead of $\RR^n$, so 
$$p^u(t,x;s,y) = \sum_{z \in \ZZ^n} \widetilde{p}^u(t,x + z;s,y),$$
where the $x,y$ on the left is in $\mX$, and the $x,y$ on the right can be viewed as their embedding into $\RR^n$.
%So, $p^u(t,x;s,y)$ also satisfies \eqref{eq:prop_rho_temp} with the same lower bound and a upper bound with a larger constant $C$, which still only depends on $n$, $T$, and $K$.
Our solution to the Fokker Planck equation, starting at uniform distribution $\rho(0,x) \equiv 1$, can be represented by
\begin{align*}
& \quad \rho^u(t,x) = \int_{\mX} p^u(t,x;0,y) \,\rd y = \int_{[0,1]^n} \sum_{z \in \ZZ^n} \widetilde{p}^u(t,x+z;0,y) \,\rd y \\
& = \int_{[0,1]^n} \sum_{z \in \ZZ^n} \widetilde{p}^u(t,x;0,y-z) \,\rd y = \int_{\RR^n} \widetilde{p}^u(t,x;0,y) \,\rd y 
\end{align*}

Substituting the lower and upper bound \eqref{eq:prop_rho_temp}, we obtain
$$\rho^u(t,x) \ge \int_{\RR^n}  C^{-1} t^{-n/2} \exp\parentheses{-\dfrac{C \abs{x-y}^2}{t}} \,\rd y =: \rho_0$$
and
$$\rho^u(t,x) \le \int_{\RR^n}  C t^{-n/2} \exp\parentheses{-\dfrac{C^{-1} \abs{x-y}^2}{t}} \,\rd y =: \rho_1.$$
Here, the two integrals above are invariant w.r.t. $t$ because of a simple change of variable. Therefore, we obtain a uniform lower bound $\rho_0$ and upper bound $\rho_1$ for $\rho^u(t,x)$, which depend only on $T$, $n$, and $K$.

Next, we show an upper bound for $\abs{\nx \rho^u(t,x)}$. It is sufficient to show
$$\abs{\rho^u(t,x_1) - \rho^u(t,x_2)} \le \rho_2 \abs{x_1-x_2}$$ 
for any $t \in [0,T]$ and $x_1,x_2 \in \mX$. We will use Feynman Kac representation of the backward parabolic equation for $\rho_b^u(t,x) := \rho^u(T-t,x)$ show this bound. Here the subscript ``$b$'' is short for backward. $\rho_b^u(t,x)$ satisfies $\rho_b^u(T,x) = 1$ and
\begin{equation}
\begin{aligned}
-\pt \rho_b^u(t,x) &=  -\nx \cdot \sqbra{b(x, u(T-t,x)) \rho_b^u(t,x)} + \sum_{i,j=1}^n \partial_i \partial_j \sqbra{D_{ij}(x) \rho_b^u(t,x)}  \\
& = \Tr\sqbra{D(x) \nx^2 \rho_b^u(t,x)} - \inner{b(x, u(T-t,x))}{\nx \rho_b^u(t,x)}\\
& \quad + \sum_{i,j=1}^n \parentheses{\partial_i D_{ij}(x) \partial_j \rho_b^u(t,x) + \partial_j D_{ij}(x) \partial_i \rho_b^u(t,x)} \\
& \quad + \sqbra{\sum_{i,j=1}^n \partial_i \partial_j D_{i,j}(x) - \nx \parentheses{ b(x, u(T-t,x))}} \rho_b^u(t,x)\\
& =: \Tr\sqbra{D(x) \nx^2 \rho_b^u(t,x)} + \inner{b_b(t,x)}{\nx \rho_b^u(t,x)} + V_b(t,x) \rho_b^u(t,x).
\end{aligned}
\end{equation}
Therefore, if we define an SDE on $\mX$
\begin{equation}\label{eq:backSDE}
\rd x^b_t = b_b(t,x^b_t)\, \rd t + \sigma(x^b_t) \,\rd W_t,
\end{equation}
then we have the Feynman Kac representation
$$\rho^u(T-t,x) = \rho_b^u(t,x) = \EE \sqbra{\exp\parentheses{-\int_t^T V_b(s,x^b_s) \,\rd s} ~\bigg|~ x^b_t=x}.$$
Here,
$$ V_b(s,x) = \sum_{i,j=1}^n \partial_i \partial_j D_{i,j}(x) - \nx \parentheses{ b(x, u(T-s,x))}$$
is bounded and Lipschitz in $x$ because Assumption \ref{assump:basic} holds and $u \in \mU$. We will use $K_b$ and $L_b$ to denote its bound and Lipschitz constant. Let $x^{i,t,b}_s$ be the SDE dynamics \eqref{eq:backSDE} with initializations $x^{i,t,b}_t = x_i ~~ (i=1,2)$, then
\begin{align*}
& \quad \abs{\rho^u(T-t,x_1) - \rho^u(T-t,x_2)} \\
& = \abs{ \EE\sqbra{ \exp\parentheses{-\int_t^T V_b(s,x^{1,t,b}_s) \,\rd s} - \exp\parentheses{-\int_t^T V_b(s,x^{2,t,b}_s) \,\rd s}}} \\
& \le \EE \abs{\exp\parentheses{-\int_t^T V_b(s,x^{1,t,b}_s) \,\rd s} \parentheses{1-\exp\parentheses{\int_t^T V_b(s,x^{1,t,b}_s) - V_b(s,x^{2,t,b}_s) \,\rd s}}} \\
& \le \exp(K_b T) ~ \EE \abs{1 - \exp\parentheses{\int_t^T L_b \, \abs{x^{1,t,b}_s - x^{2,t,b}_s} \rd s}} \\
& \le \exp(K_b T) ~ \EE \sqbra{\exp\parentheses{\sqrt{n} \, T \, L_b} / (\sqrt{n} \, T) \int_t^T \, \abs{x^{1,t,b}_s - x^{2,t,b}_s} \rd s } \le \rho_2 \abs{x_1-x_2}
\end{align*}
In the second inequality, we used the Lipschitz condition of $V_b$. In the third inequality, we used the inequality $\exp(Lx) - 1 \le \exp(LK)x / K$ if $x \in [0,K]$ (recall that $x^{i,t,b}_s \in \mX$, which implies $\abs{x^{1,t,b}_s - x^{2,t,b}_s} \le \sqrt{n}$). The last inequality comes from Lemma \ref{lem:gronwall} (in the next section), which is a direct corollary of \eqref{eq:Gronwall2}. Therefore, $\abs{\nx \rho^u(t,x)}$ is bounded by some constant $\rho_2$.
\end{proof}

\section{Some auxiliary lemmas}\label{sec:lemmas}
\begin{lem}\label{lem:gronwall}[Stochastic Gronwall inequality] 
Under Assumption \ref{assump:basic}, there exists a positive constant $C_1$ s.t. for any two control functions $u_1, u_2 \in \mU$, we have
\begin{equation}\label{eq:Gronwall1}
\sup_{t \in [0,T]} \EE  \abs{x^1_t - x^2_t}^2 \le C_1 \EE\abs{x^1_0-x^2_0}^2 + C_1 \EE \sqbra{ \int_0^T \abs{u_1(t,x^1_t) - u_2(t,x^1_t)}^2 \rd t},
\end{equation}
where $x^1_t$ and $x^2_t$ are the state process \eqref{eq:SDE_X} under controls $u_1$ and $u_2$ respectively. As a direct corollary, if $u_1=u_2$, then
\begin{equation}\label{eq:Gronwall2}
\sup_{t \in [0,T]} \EE  \abs{x^1_t - x^2_t}^2 \le C_1 \EE\abs{x^1_0-x^2_0}^2.
\end{equation}
Moreover, if $x^1_0=x^2_0 \sim \text{Unif}(\mX)$, then
\begin{equation}\label{eq:Gronwall3}
\sup_{t \in [0,T]} \EE \abs{x^1_t - x^2_t}^2 \le C_1 \norm{u_1-u_2}_{L^2}^2.
\end{equation}
\end{lem}
\begin{proof}
We denote $b^i_t=b(x^i_t, u_i(t,x^i_t))$, $\sigma^i_t=\sigma(x^i_t)$ for $i=1,2$, so $\rd x^i_t = b^i_t \,\rd t + \sigma^i_t \,\rd W_t$. By It\^o's lemma,
$$\rd \abs{x^1_t-x^2_t}^2 = \sqbra{\abs{\sigma^1_t-\sigma^2_t}^2 + 2\inner{x^1_t-x^2_t}{b^1_t-b^2_t}} \rd t + 2(x^1_t-x^2_t)\tp (\sigma^1_t-\sigma^2_t) \,\rd W_t.$$
Integrate and take expectation, we obtain
\begin{equation}\label{eq:square_diffx}
\EE \abs{x^1_T-x^2_T}^2 = \EE\abs{x^1_0-x^2_0}^2 + \EE \int_0^T \sqbra{\abs{\sigma^1_t-\sigma^2_t}^2 + 2\inner{x^1_t-x^2_t}{b^1_t-b^2_t}} \rd t.
\end{equation}
By the Lipschitz condition in Assumption \ref{assump:basic},
\begin{align*}
& \quad \abs{b^1_t - b^2_t} \le L \abs{x^1_t-x^2_t} + L \abs{u_1(t,x^1_t) - u_2(t,x^2_t)} \\
& \le (L+L^2) \abs{x^1_t-x^2_t} + L \abs{u_1(t,x^1_t) - u_2(t,x^1_t)}.
\end{align*}
So, 
\begin{equation}\label{eq:diffb_bound}
\abs{b^1_t - b^2_t}^2 \le 2(L+L^2)^2 \abs{x^1_t-x^2_t}^2 + 2L^2 \abs{u_1(t,x^1_t) - u_2(t,x^1_t)}^2.
\end{equation}
Also,
\begin{equation}\label{eq:diffsigma_bound}
\abs{\sigma^1_t - \sigma^2_t}^2 \le L^2 \abs{x^1_t-x^2_t}^2.
\end{equation}
Applying Cauchy's inequality, and substituting \eqref{eq:diffb_bound} and \eqref{eq:diffsigma_bound} into \eqref{eq:square_diffx}, we obtain
\begin{equation}\label{eq:square_diffx2}
\begin{aligned}
\EE \abs{x^1_T-x^2_T}^2 \le \EE\abs{x^1_0-x^2_0}^2 + \EE \int_0^T \sqbra{\abs{\sigma^1_t-\sigma^2_t}^2 + \abs{x^1_t-x^2_t}^2 + \abs{b^1_t-b^2_t}^2 }\rd t\\
\le \EE\abs{x^1_0-x^2_0}^2 + \EE \int_0^T \sqbra{10L^4 \abs{x^1_t-x^2_t}^2 + 2L^2 \abs{u_1(t,x^1_t) - u_2(t,x^1_t)}^2}\rd t
\end{aligned}
\end{equation}
Note that \eqref{eq:square_diffx2} still holds if we replace $T$ by some $T'<T$, so we can apply Gronwall's inequality and obtain
\begin{equation}\label{eq:square_diffx3}
\EE \abs{x^1_T-x^2_T}^2 \le e^{10L^4T} \EE\abs{x^1_0-x^2_0}^2 + 2L^2 e^{10L^4T} \EE \sqbra{ \int_0^T \abs{u_1(t,x^1_t) - u_2(t,x^1_t)}^2 \rd t}.
\end{equation}
Again, \eqref{eq:square_diffx3} still holds if we replace $T$ by some $T'<T$, so \eqref{eq:Gronwall1} holds. Moreover, if $x_0^1 = x_0^2 \sim \text{Unif}(\mX)$, then by Proposition \ref{prop:rho},
\begin{align*}
& \quad \EE \int_0^T \abs{u_1(t,x^1_t) - u_2(t,x^1_t)}^2 \rd t \\
& = \int_{\mX}\int_0^T \rho^{u_1}(t,x) \abs{u_1(t,x) - u_2(t,x)}^2 \rd t \rd x \le \rho_1 \norm{u_1-u_2}_{L^2}^2.
\end{align*}
Therefore, \eqref{eq:Gronwall3} holds.
\end{proof}

\begin{lem}\label{lem:regularity_Vu}
Under Assumption \ref{assump:basic}, there exists a positive constant $C_2$ s.t. for any two control functions $u_1, u_2 \in \mU$, we have
\begin{equation}\label{eq:regularity_Vu}
\norm{\nx V_{u_1} - \nx V_{u_2}}_{L^2} \le C_2 \norm{u_1-u_2}_{L^2}.
\end{equation}
\end{lem}
\begin{proof}
We firstly give some notations. Following the notations in the previous lemma, $x^1_t$ and $x^2_t$ are the state process w.r.t.{} controls $u_1$ and $u_2$, starting at $x^1_0 = x^2_0 = x_0 \sim \text{Unif}(\mX)$. For $i=1,2$, we have $b^i_t=b(x^i_t, u_i(t,x^i_t))$, $\sigma^i_t=\sigma(x^i_t)$, and $r^i_t=r(x^i_t, u_i(t,x^i_t))$. 
% We also define the following gradient processes $\nx b^i_t=\nx b(x^i_t, u_i(t,x^i_t))$, $\nx \sigma^i_t =\nx \sigma(x^i_t)$, and $\nx r^i_t = \nx r(x^i_t, u_i(t,x^i_t))$. Here please note that $\nx$ only operate on the first argument in $b$ and $r$. 
By Assumption \ref{assump:basic}, for $f = b, \sigma, r$, we have
\begin{equation*}
\begin{aligned}
& \quad \abs{f^1_t-f^2_t} = \abs{f(x^1_t, u_1(t,x^1_t))-f(x^2_t, u_2(t,x^2_t))}\\
& \le L \abs{x^1_t-x^2_t} + L \abs{u_1(t,x^1_t) - u_2(t,x^2_t)} \le (L+L^2) \abs{x^1_t-x^2_t} + L \abs{u_1(t,x^1_t) - u_2(t,x^1_t)}
\end{aligned}
\end{equation*}
and hence
\begin{equation*}
\abs{f^1_t-f^2_t}^2 \le 2(L+L^2)^2 \abs{x^1_t-x^2_t}^2 + 2L^2 \abs{u_1(t,x^1_t) - u_2(t,x^1_t)}^2.
\end{equation*}
If we make an integration and apply \eqref{eq:Gronwall3} in Lemma \ref{lem:gronwall}, we obtain
\begin{equation}\label{eq:all1-2}
\EE \int_0^T \abs{f^1_t-f^2_t}^2 \rd t \le C \EE \sqbra{ \int_0^T \abs{u_1(t,x^1_t) - u_2(t,x^1_t)}^2 \rd t} \le C \norm{u_1-u_2}_{L^2}^2.
\end{equation}
Next, we will show \eqref{eq:regularity_Vu} in two steps.

\emph{Step 1.} We want to show
\begin{equation}\label{eq:regVu_0}
\norm{V_{u_1}(0,\cdot) - V_{u_2}(0,\cdot)}_{L^2} \le C \norm{u_1-u_2}_{L^2}.
\end{equation}
%By \eqref{eq:Ito},
Applying It\^o's lemma on $V_{u_i}(t,x^i_t)$ for $i=1,2$, we obtain
$$g(x^i_T) = V_{u_i}(0,x_0) + \int_0^T \parentheses{\partial_t V_{u_i}(t,x^i_t) + \mI_{u_i}V_{u_i}(t,x^i_t) }\rd t + \int_0^T \nx V_{u_i}(t,x^i_t)\tp \sigma^i_t \rd W_t,$$
where $\mI_{u_i}$ is the infinitesimal generator of the SDE \eqref{eq:SDE_X} under control $u_i$. Applying the HJ equation \eqref{eq:HJ} in the drift term and rearranging the terms, we get
\begin{equation}\label{eq:diff_V0}
V_{u_i}(0,x_0) = g(x^i_T) + \int_0^T r^i_t \rd t - \int_0^T \nx V_{u_i}(t,x^i_t)\tp \sigma^i_t \rd W_t.
\end{equation}
So,
\begin{equation}\label{eq:diff_v0EE}
V_{u_1}(0,x_0) - V_{u_2}(0,x_0) = \EE\sqbra{ g(x^1_T) -g(x^2_T) + \int_0^T (r^1_t - r^2_t) \rd t ~\Big|~ x_0}.
\end{equation}
Therefore,
\begin{equation*}
\begin{aligned}
& \quad \int_{\mX} \abs{V_{u_1}(0,x) - V_{u_2}(0,x)}^2 \rd x = \EE \sqbra{\parentheses{V_{u_1}(0,x_0) - V_{u_2}(0,x_0)}^2}\\
& = \EE \sqbra{\parentheses{\EE\sqbra{ g(x^1_T) -g(x^2_T) + \int_0^T (r^1_t - r^2_t) \rd t ~\Big|~ x_0}}^2} \\
& \le \EE \sqbra{\parentheses{\EE\sqbra{ L\abs{x^1_T - x^2_T} + \int_0^T \abs{r^1_t - r^2_t} \rd t ~\Big|~ x_0}}^2} \\
& \le \EE \sqbra{\parentheses{ L\abs{x^1_T - x^2_T} + \int_0^T \abs{r^1_t - r^2_t} \rd t }^2} \\
& \le \EE \sqbra{ 2L^2 \abs{x^1_T - x^2_T}^2 + 2T \int_0^T \abs{r^1_t - r^2_t}^2 \rd t} \le C \norm{u_1-u_2}_{L^2}^2,
\end{aligned}
\end{equation*}
where we have consecutively used: $x_0 \sim \text{Unif}(\mX)$; equation \eqref{eq:diff_v0EE}; Lipschitz condition of $g$ in Assumption \ref{assump:basic}; Jensen's inequality and tower property; Cauchy's inequality; Lemma \ref{lem:gronwall} and \eqref{eq:all1-2}. Therefore, we have shown
\begin{equation}\label{eq:diff_v0norm}
\norm{V_{u_1}(0,\cdot) - V_{u_2}(0,\cdot)}_{L^2}^2 \le C \norm{u_1-u_2}_{L^2}^2
\end{equation}
where this constant $C$ only depends on $K,n,T$. 
%Also, \eqref{eq:diff_v0norm} holds with the same $C$ if the total time span $T$ decreases. Therefore, we can reformulate the control problem such that it start at $t \in (0,T)$ instead of $0$. Then the new state process starts at $x^i_t \sim \text{Unif}(\mX)$ and the new value function coincide with $V_{u_i}$ on $[t,T]$ by definition \eqref{eq:value}. We also remark that the constants $\rho_0, \rho_1$ in Proposition \ref{prop:rho} remain the same because $T$ decreases. Applying the argument for \eqref{eq:diff_v0norm} on the new control problem gives us
% \begin{equation*}
% \norm{V_{u_1}(t,\cdot) - V_{u_2}(t,\cdot)}_{L^2}^2 \le C \int_t^T \norm{u_1(s,\cdot)-u_2(s,\cdot)}_{L^2}^2 \rd s \le C \norm{u_1-u_2}_{L^2}^2.
% \end{equation*}
% Making an integration in $t$ gives us \eqref{eq:regVu_0}.

\emph{Step 2.} We want to show \eqref{eq:regularity_Vu}
\begin{equation*}
\norm{\nx V_{u_1} - \nx V_{u_2}}_{L^2} \le C \norm{u_1 - u_2}_{L^2}.
\end{equation*}
We recall that the first order adjoint equation is given by \eqref{eq:adjoint1}. We denote $p^i_t = -\nx V_{u_i}(t,x^i_t)$ and $q^i_t = -\nx^2 V_{u_i}(t,x^i_t) \sigma^i_t$ for $i=1,2$. They satisfy the BSDEs
\begin{equation}\label{eq:adjoint2}
\left\{ \begin{aligned}
\rd p^i_t & = -\sqbra{(\nx b^i_t)\tp p^i_t + \nx \Tr\parentheses{(\sigma^i_t)\tp q^i_t} - \nx r^i_t} \rd t + q^i_t ~ \rd W_t \\
p^i_T & = -\nx g(x^i_T).
\end{aligned} \right.
\end{equation}
By assumption \ref{assump:basic}, we have $\abs{p^i_t} \le K$ and $\abs{q^i_t} \le K^2$.
By \eqref{eq:diff_V0},
$$\int_0^T \parentheses{(p^1_t)\tp \sigma^1_t - (p^2_t)\tp \sigma^2_t} \rd W_t = g(x^1_T)-g(x^2_T) + \int_0^T (r^1_t-r^2_t) \rd t - \parentheses{V_{u_1}(0,x_0) - V_{u_2}(0,x_0)}.$$
Taking a square expectation and using a Cauchy inequality, we obtain
\begin{equation}\label{eq:diff_psigma_square}
\begin{aligned}
& \quad \EE \int_0^T \abs{(p^1_t)\tp \sigma^1_t - (p^2_t)\tp \sigma^2_t}^2 \rd t = \EE \parentheses{\int_0^T \parentheses{(p^1_t)\tp \sigma^1_t - (p^2_t)\tp \sigma^2_t} \rd W_t}^2 \\
& \le 3\EE \sqbra{\abs{g(x^1_T)-g(x^2_T)}^2 + T \int_0^T \abs{r^1_t-r^2_t}^2 \rd t + \abs{V_{u_1}(0,x_0) - V_{u_2}(0,x_0)}^2}\\
& \le 3L^2 \EE \abs{x^1_T - x^2_T}^2 + 3T ~ \EE \int_0^T \abs{r^1_t-r^2_t}^2 \rd t + 3 \norm{V_{u_1}(0,\cdot) - V_{u_2}(0,\cdot)}_{L^2}^2\\
& \le C \norm{u_1-u_2}_{L^2}^2,
\end{aligned}
\end{equation}
where the last inequality is because of the Gronwall inequality, estimate \eqref{eq:all1-2}, and the result in \emph{Step 1}. Also note that
$$(p^1_t)\tp \sigma^1_t - (p^2_t)\tp \sigma^2_t = (p^1_t)\tp \sigma^1_t - (p^2_t)\tp \sigma^1_t + (p^2_t)\tp \sigma^1_t - (p^2_t)\tp \sigma^2_t,$$
so
\begin{equation}\label{eq:diff_psigma1_square}
\abs{(p^1_t)\tp \sigma^1_t - (p^2_t)\tp \sigma^1_t}^2 \le 2\abs{(p^1_t)\tp \sigma^1_t - (p^2_t)\tp \sigma^2_t}^2 + 2K^2 \abs{\sigma^1_t - \sigma^2_t}^2.
\end{equation}
Next, we have
\begin{equation}\label{eq:diff_psquare}
\EE\int_0^T \abs{p^1_t-p^2_t}^2 \rd t \le \dfrac{1}{2 \sigma_0} \EE\int_0^T \abs{(p^1_t)\tp \sigma^1_t - (p^2_t)\tp \sigma^1_t}^2 \rd t \le C \norm{u_1-u_2}_{L^2}^2,
\end{equation}
where the first inequality is due to the uniform ellipticity assumption and the second is because of \eqref{eq:diff_psigma1_square}, \eqref{eq:diff_psigma_square}, and \eqref{eq:all1-2}.
Next, since
$$p^2_t - p^1_t = \nx V_{u_1}(t,x^1_t) - \nx V_{u_2}(t,x^1_t) + \nx V_{u_2}(t,x^1_t) - \nx V_{u_2}(t,x^2_t),$$
we have
\begin{equation}\label{eq:diff_nxV}
\begin{aligned}
& \quad \abs{\nx V_{u_1}(t,x^1_t) - \nx V_{u_2}(t,x^1_t)}^2 \\
&\le 2 \abs{p^1_t - p^2_t}^2 + 2 \abs{\nx V_{u_2}(t,x^1_t) - \nx V_{u_2}(t,x^2_t)}^2\\
&\le 2 \abs{p^1_t - p^2_t}^2 + 2 L^2 \abs{x^1_t-x^2_t}^2.
\end{aligned}
\end{equation}
Therefore,
\begin{equation*}
\begin{aligned}
& \quad \norm{\nx V_{u_1} - \nx V_{u_2}}_{L^2}^2 \le \dfrac{1}{\rho_0} \EE \int_0^T \abs{\nx V_{u_1}(t,x^1_t) - \nx V_{u_2}(t,x^1_t)}^2 \rd t \\
& \le C \EE \int_0^T \parentheses{\abs{p^1_t - p^2_t}^2 + \abs{x^1_t-x^2_t}^2} \rd t \le C \norm{u_1-u_2}_{L^2}^2,
\end{aligned}
\end{equation*}
where we have consecutively used: Proposition \ref{prop:rho}; equation \eqref{eq:diff_nxV}; equation \eqref{eq:diff_psquare} and Lemma \ref{lem:gronwall}. Therefore \eqref{eq:regularity_Vu} holds.
\end{proof}

\begin{lem}\label{lem:J_quadratic}
Under Assumption \ref{assump:basic}, there exists a positive constant $C_3$ s.t.
\begin{equation}\label{eq:J_quadratic}
J[u]- J[u^*] \le C_3 \norm{u-u^*}_{L^2}^2
\end{equation}
for any $u \in \mU$.
\end{lem}
\begin{proof}
Denote $\ve_0 = \norm{u-u^*}_{L^2}$ and define $\phi$ by $u = u^*+\ve_0\phi$, then $\norm{\phi}_{L^2}=1$. We denote $u^{\ve} = u^* + \ve \phi$. Denote the corresponding value function $V_{u^{\ve}}$ by $V^{\ve}$. Denote the corresponding density function by $\rho^{\ve}$, with initial condition $\rho^{\ve}(0,\cdot) \equiv 1$.
By Proposition \ref{prop:cost_derivative},
\begin{equation}\label{eq:dJdve2}
\begin{aligned}
& \quad \dfrac{\rd}{\rd \ve}J[u^{\ve}] = \inner{\fd{J}{u}[u^{\ve}]}{\phi}_{L^2} \\
& = -\int_0^T \int_{\mX} \rho^{\ve}(t,x) \inner{\nbu G\parentheses{t,x,u^{\ve}(t,x),-\nx V^{\ve}(t,x)}}{\phi(t,x)} \rd x \, \rd t.
\end{aligned}
\end{equation}

In order to show \eqref{eq:J_quadratic}, it is sufficient to show that $\dfrac{\rd}{\rd \ve} J[u^{\ve}] \le C\ve$ for some uniform constant $C$ (that does not depend on $\phi$), because
$$J[u^*+\ve_0\phi] - J[u^*] = \int_0^{\ve_0} \dfrac{\rd}{\rd \ve}J[u^{\ve}] \, \rd \ve.$$
We estimate $\nbu G$ in \eqref{eq:dJdve2} first. 
\begin{equation}\label{eq:nuGve}
\begin{aligned}
& \quad \abs{\nbu G \parentheses{t,x,u^{\ve}(t,x),-\nx V^{\ve}(t,x)}} \\
& = \abs{\nbu G \parentheses{t,x,u^{\ve}(t,x),-\nx V^{\ve}(t,x)} - \nbu G \parentheses{t,x,u^{*}(t,x),-\nx V^*(t,x)}}\\
& \le \abs{\nbu G \parentheses{t,x,u^{\ve}(t,x),-\nx V^{\ve}(t,x)} - \nbu G \parentheses{t,x,u^{*}(t,x),-\nx V^{\ve}(t,x)}}\\
& \quad + \abs{\nbu G \parentheses{t,x,u^{*}(t,x),-\nx V^{\ve}(t,x)} - \nbu G \parentheses{t,x,u^{*}(t,x),-\nx V^*(t,x)}}\\
& =: (\rom{1}) + (\rom{2}),
\end{aligned}
\end{equation}
where we used the maximum condition \eqref{eq:max1} in the first equality. Let us also denote $u^{\ve}(t,x)$ and $u^*(t,x)$ by $u^{\ve}$ and $u^*$ for simplicity. For $(\rom{1})$, we have
\begin{equation}\label{eq:lem3rom1}
\begin{aligned}
(\rom{1}) &\le \abs{\nbu r(x,u^{\ve}) - \nbu r(x,u^*)} +  \abs{\parentheses{\nbu b(x,u^{\ve}) - \nbu b(x,u^*)}\tp \nx V^{\ve}(t,x)}\\
& \le L \ve \abs{\phi(t,x)} + L \ve \abs{\phi(t,x)}K \le C \ve \abs{\phi(t,x)},
\end{aligned}
\end{equation}
where we have used the Lipschitz conditions in Assumption \ref{assump:basic} and boundedness of the value function's derivatives. For $(\rom{2})$, we have
\begin{equation}\label{eq:lem3rom2}
(\rom{2}) \le \abs{ \nbu b(x,u^*) \tp \parentheses{\nx V^{\ve}(t,x) - \nx V^*(t,x)} } \le K \abs{\nx V^{\ve}(t,x) - \nx V^*(t,x)}.
\end{equation}
Combining \eqref{eq:lem3rom1} and \eqref{eq:lem3rom2} into \eqref{eq:nuGve}, we obtain
\begin{equation}\label{eq:bound_nuGve}
\abs{\nbu G\parentheses{t,x,u^{\ve}(t,x),-\nx V^{\ve}(t,x)}} \le C \parentheses{\ve\abs{\phi(t,x)} + \abs{\nx V^{\ve}(t,x) - \nx V^*(t,x)}}.
\end{equation}
Therefore, \eqref{eq:dJdve2} has estimate
\begin{equation*}
\begin{aligned}
& \quad \abs{\dfrac{\rd}{\rd \ve}J[u^{\ve}]} \\
& \le \rho_1 \int_0^T\int_{\mX} \parentheses{\dfrac{1}{2\ve} \abs{\nbu G \parentheses{t,x,u^{\ve}(t,x),-\nx V^{\ve}(t,x)}}^2 + \dfrac{\ve}{2} \abs{\phi(t,x)}^2} \rd x \,\rd t \\
& \le C \parentheses{\ve \norm{\phi}_{L^2}^2 + \dfrac{1}{\ve} \norm{\nx V^{\ve} - \nx V^*}_{L^2}^2 + \ve \norm{\phi}_{L^2}^2} \le C\ve,
\end{aligned}
\end{equation*}
where we have consecutively used: Proposition \ref{prop:rho} and Cauchy's inequality; inequality \eqref{eq:bound_nuGve}; Lemma \ref{lem:regularity_Vu}. Therefore, \eqref{eq:J_quadratic} holds.
\end{proof}

\begin{lem}\label{lem:quadratic_Vu}
Under Assumption \ref{assump:basic}, there exists a positive constant $C_4$ s.t. for any control function $u \in \mU$, we have
\begin{equation}\label{eq:quadratic_Vu}
\norm{\nx V_{u} - \nx V^*}_{L^2} \le C_4 \norm{u-u^*}_{L^2}^{1+\alpha},
\end{equation}
with $\alpha = \frac{1}{n+2}$.
\end{lem}
% \begin{rmk}
% We believe that \eqref{eq:quadratic_Vu} holds with $\alpha=1$, but encountered some technical difficulty to prove it. We give the intuition here. Following the notation in the previous lemma, we denote $u^{\ve} = u^* + \ve \phi$. Since $V^{\ve}(t,x)$ reaches its minimum at $\ve=0$ for any $(t,x)$, we have
% $$\pve V^{\ve}(t,x) \big|_{\ve=0} =0.$$
% With sufficient regularity, we have
% $$\parentheses{\pve \nx V^{\ve}}|_{\ve=0} = \parentheses{\nx \pve V^{\ve}}|_{\ve=0} = \nx \parentheses{\pve V^{\ve}}|_{\ve=0} = 0,$$
% $$\parentheses{\pve \nx^2 V^{\ve}}|_{\ve=0} = \parentheses{\nx^2 \pve V^{\ve}}|_{\ve=0} = \nx^2 \parentheses{\pve V^{\ve}}|_{\ve=0} = 0.$$
% Making a local Taylor expansion w.r.t. $\ve$, we know that $\nx V^{\ve} - \nx V^*$ and $\nx^2 V^{\ve} - \nx^2 V^*$ are of order $\mO(\ve^2)$, which implies \eqref{eq:quadratic_Vu} holds with $\alpha=1$.
% \end{rmk}
\begin{proof}
We will inherit some notations from the previous lemma. Denote $\ve_0 = \norm{u-u^*}_{L^2}$ and let $u = u^*+\ve_0\phi$, then $\norm{\phi}_{L^2}=1$. We denote $u^{\ve} = u^* + \ve \phi$. Denote the corresponding value function $V_{u^{\ve}}$ by $V^{\ve}$. Denote the corresponding density function by $\rho^{\ve}$, with initial condition $\rho^{\ve}(0,\cdot) \equiv 1$. The key difficulty for the proof is that $\phi(t,x)$ may not lie in $\mU$ like $u^*(t,x)$ or $u^{\ve}(t,x)$, which has $K$ as a bound for itself and its derivatives. $\phi = (u^{\ve} - u^*)/ \ve$ do have some regularity, but the constant for the bounds have a factor of $2/\ve$.

% \emph{Step 1.} We want to show
% \begin{equation}\label{eq:lem4step1}
% \norm{V_u - V^*}_{L^2} \le C \norm{u - u^*}_{L^2}^{1+\alpha}.
% \end{equation}
% Note that $V_u \ge V^*$, so
% \begin{equation*}
% \begin{aligned}
% & \quad \int_{\mX} \abs{V_u(0,x)-V^*(0,x)} \rd x = \int_{\mX} \parentheses{V_u(0,x)-V^*(0,x)} \rd x \\
% & = J[u] - J[u^*] \le C_3 \norm{u-u^*}_{L^2}^2,
% \end{aligned}
% \end{equation*}
% where we used Lemma \ref{lem:J_quadratic} in the last inequality. i.e., $\norm{V^{\ve}(0,\cdot) - V^*(0,\cdot)}_{L^1} \le C_3\norm{u-u^*}_{L^2}^2$. A similar argument with $t \in (0,T)$ as the starting time gives us $\norm{V^{\ve}(t,\cdot) - V^*(t,\cdot)}_{L^1} \le C\norm{u-u^*}_{L^2}^2$. Therefore, we have
% $$\norm{V^{\ve} - V^*}_{L^1} \le C\norm{u-u^*}_{L^2}^2,$$
% which implies \eqref{eq:lem4step1} because $V^{\ve}$, $V^*$, $u$, $u^*$ are bounded.

% \emph{Step 2.} We want to show
% \begin{equation}\label{eq:lem4step2}
% \norm{\nx V_u - \nx V^*}_{L^2} \le C \norm{u - u^*}_{L^2}^{1+\alpha}.
% \end{equation}
It is sufficient to show the partial derivative in each dimension at $t=0$ satisfies the estimate in $L^1$ norm:
\begin{equation}\label{eq:lem4step2temp}
\norm{\partial_i V_u(0,\cdot) - \partial_i V^*(0,\cdot)}_{L^1} \le C \norm{u - u^*}_{L^2}^{1+\alpha}.
\end{equation}
This is because firstly we can repeat the argument in other dimensions and for other $t \in (0,T)$. Secondly, the derivatives of the value functions are bounded, which implies equivalence between $L^1$ and $L^2$ norm (recall we are on a torus).

\emph{Step 1.} We reformulate the problem using finite difference in this step. Let $x_1 \in \mX$ be a variable and denote $x_2 = x_1 + \delta e_i$ a perturbation. We assume $\delta >0$ without loss of generality. We have
\begin{equation}\label{eq:lem4_s2d1}
\begin{aligned}
& \quad \norm{\partial_i V_u(0,\cdot) - \partial_i V^*(0,\cdot)}_{L^1} = \int_{\mX} \abs{\partial_i V_u(0,x_1) - \partial_i V^*(0,x_1)} \rd x_1 \\
& = \int_{\mX} \abs{ \int_0^{\ve_0} \pve \partial_i V^{\ve}(0,x_1) \rd\ve } \rd x_1 = \int_{\mX} \abs{ \int_0^{\ve_0} \partial_i \pve  V^{\ve}(0,x_1) \rd\ve} \rd x_1 \\
&= \int_{\mX} \abs{ \int_0^{\ve_0} \lim_{\delta \to 0^+} \dfrac{1}{\delta} \parentheses{\pve  V^{\ve}(0,x_2) - \pve  V^{\ve}(0,x_1)} \rd\ve } \rd x_1 \\
& \le \liminf_{\delta \to 0^+} \dfrac{1}{\delta}  \int_0^{\ve_0} \int_{\mX} \abs{ \pve  V^{\ve}(0,x_2) - \pve  V^{\ve}(0,x_1)}\rd x_1 \, \rd\ve,
\end{aligned}
\end{equation}
where the last inequality is because of Fatou's lemma. Now, we denote $\xyet$ and $\xeet$ the state processes under control $u^{\ve}$ that start at $x^{1,\ve}_0=x_1$ and $x^{2,\ve}_0=x_2$. Here $\xyet$ and $\xeet$ share the same realization of Brownian motion. By Proposition \ref{prop:rho},
\begin{align*}
-\pve  V^{\ve}(0,x_1) &= \EE \sqbra{\int_0^T \inner{\nbu G \parentheses{t, \xyet, u^{\ve}(t,\xyet), -\nx V^{\ve}(t,\xyet) }}{\phi(t,x^{1,\ve}_t)} \rd t }\\
&=:  \EE \sqbra{\int_0^T \inner{\nbu G^{1,\ve}_t}{\phi^{1,\ve}_t} \rd t }.
\end{align*}
Similarly, $-\pve  V^{\ve}(0,x_2) = \EE \sqbra{\int_0^T \inner{\nbu G^{2,\ve}_t}{\phi^{2,\ve}_t} \rd t }$.
So
\begin{equation}\label{eq:lem4_s2d1a}
\begin{aligned}
& \quad \int_{\mX} \abs{  \pve  V^{\ve}(0,x_2) - \pve  V^{\ve}(0,x_1)}\rd x_1 \\
& = \int_{\mX} \abs{ \EE \sqbra{\int_0^T \inner{\nbu G^{1,\ve}_t}{\phi^{1,\ve}_t} - \inner{\nbu G^{2,\ve}_t}{\phi^{2,\ve}_t} \rd t } }\rd x_1
\end{aligned}
\end{equation}
Combining \eqref{eq:lem4_s2d1} and \eqref{eq:lem4_s2d1a}, in order to show \eqref{eq:lem4step2temp}, it is sufficient to show that
\begin{equation}\label{eq:lem4step2temp2}
\int_{\mX} \abs{ \EE \sqbra{\int_0^T   \inner{\nbu G^{1,\ve}_t}{\phi^{1,\ve}_t} - \inner{\nbu G^{2,\ve}_t}{\phi^{2,\ve}_t}   \rd t}}\rd x_1 \le C \delta \ve^{\alpha}
\end{equation}
for some uniform constant $C$. We only need to show \eqref{eq:lem4step2temp2} when $\delta \le \ve$ because $\nbu G^{j,\ve}_t|_{\ve=0}=0$ ($j=1,2$), which implies \eqref{eq:lem4step2temp2} when $\ve=0$.

\emph{Step 2.} We split into two sub-tasks to show \eqref{eq:lem4step2temp2} in this step. Let us denote $\rho^{1,\ve}(t,x)$ and $\rho^{2,\ve}(t,x)$ the density functions of $\xyet$ and $\xeet$, then $\rho^{j,\ve}(0,\cdot)=\delta_{x_j} \, (j=1,2)$ and an sufficient condition for \eqref{eq:lem4step2temp2} is
\begin{equation}\label{eq:lem4step2temp3}
\begin{aligned}
& \int_{\mX} \int_0^T \int_{\mX} \Big| \inner{\nbu G\parentheses{t,x,u^{\ve}(t,x), -\nx V^{\ve}(t,x)}}{\phi(t,x)} \\
& \hspace{0.8in} \parentheses{\rho^{1,\ve}(t,x) - \rho^{2,\ve}(t,x)} \Big| \, \rd x\, \rd t\,\rd x_1 \le C \delta \ve^{\alpha}.
\end{aligned}
\end{equation}
The idea to prove \eqref{eq:lem4step2temp2} is to decompose the time interval $[0,T]$ into two sub-intervals $[0,\ve^{2\alpha}]$ and $(\ve^{2\alpha},T]$ and prove \eqref{eq:lem4step2temp2} and \eqref{eq:lem4step2temp3} with $\int_0^T$ replaced by the corresponding intervals respectively. For the first part, we take advantage that $\ve^{2\alpha}$ is small, while for the second part, we use the fact that $\rho^{1,\ve}(t,x)$ and $\rho^{2,\ve}(t,x)$ are nicely mixed.

\emph{Step 3.} We estimate the integration in the interval $[0,\ve^{2\alpha}]$ in this step. We want to show
\begin{equation}\label{eq:lem4step2temp4}
\int_{\mX} \EE \sqbra{\int_0^{\ve^{2\alpha}} \abs{  \inner{\nbu G^{1,\ve}_t}{\phi^{1,\ve}_t} - \inner{\nbu G^{2,\ve}_t}{\phi^{2,\ve}_t}  } \rd t}\rd x_1 \le C \delta \ve^{\alpha}.
\end{equation}
Using the Lipschitz property and boundedness of $\nbu r$, $\nbu b$, and $\nx V^{\ve}$, we can show
$$\abs{\nbu G^{1,\ve}_t - \nbu G^{2,\ve}_t} \le C \abs{\xyet-\xeet}.$$
Also, we have $\abs{\pyet-\peet} \le \dfrac{2L}{\ve} \abs{\xyet-\xeet}$. Therefore,
\begin{equation*}
\begin{aligned}
& \quad \abs{ \inner{\nbu G^{1,\ve}_t}{\phi^{1,\ve}_t} - \inner{\nbu G^{2,\ve}_t}{\phi^{2,\ve}_t} }\\
& = \abs{ \inner{\nbu G^{1,\ve}_t - \nbu G^{2,\ve}_t}{\phi^{1,\ve}_t} + \inner{\nbu G^{2,\ve}_t}{\phi^{1,\ve}_t-\phi^{2,\ve}_t} }\\
& \le \abs{\nbu G^{1,\ve}_t - \nbu G^{2,\ve}_t} \, \abs{\phi^{1,\ve}_t} + \abs{\nbu G^{2,\ve}_t} \, \abs{\phi^{1,\ve}_t-\phi^{2,\ve}_t}\\
& \le C \left(\abs{\pyet} + \abs{\peet} + \dfrac{1}{\ve} \abs{\nx V^{\ve}(t,\xeet) - \nx V^*(t,\xeet)} \right) \abs{\xyet-\xeet}
\end{aligned}
\end{equation*}
where we have used \eqref{eq:bound_nuGve} to estimate $\abs{\nbu G^{2,\ve}_t}$ in the last inequality. Substituting the estimate above into \eqref{eq:lem4step2temp4} left, we obtain
\begin{equation}\label{eq:lem4_step23}
\begin{aligned}
& \quad \int_{\mX} \EE \sqbra{\int_0^{\ve^{2\alpha}} \abs{  \inner{\nbu G^{1,\ve}_t}{\phi^{1,\ve}_t} - \inner{\nbu G^{2,\ve}_t}{\phi^{2,\ve}_t}  } \rd t}\rd x_1 \\
& \le C \int_{\mX} \EE \left[ \int_0^{\ve^{2\alpha}} \left(  \delta \ve^{\alpha} \abs{\pyet}^2 + \delta \ve^{\alpha} \abs{\peet}^2 + \delta\ve^{\alpha-2} \abs{\nx V^{\ve}(t,\xeet) - \nx V^*(t,\xeet)}^2 \right. \right. \\
& \hspace{0.4in} \left. \left.  + \dfrac{1}{\delta \ve^{\alpha}} \abs{\xyet-\xeet}^2 \right) \rd t \right] \rd x_1 \\
& \le C \rho_1 \parentheses{ 2\delta \ve^{\alpha} \norm{\phi}_{L^2}^2 + \delta\ve^{\alpha-2} \norm{\nx V^{\ve} - \nx V^*}_{L^2}^2 } + C\int_{\mX} \dfrac{1}{\delta \ve^{\alpha}} \int_0^{\ve^{2\alpha}} \EE \abs{\xyet-\xeet}^2 \rd t \, \rd x_1 \\
& \le C \parentheses{ \delta \ve^{\alpha} + \delta\ve^{\alpha-2} \ve^2  } + C \int_{\mX} \dfrac{1}{\delta \ve^{\alpha}} \int_0^{\ve^{2\alpha}} C_1 \abs{x_1-x_2}^2 \rd t \, \rd x_1 \\
& \le C \delta \ve^{\alpha} + \dfrac{C}{\delta \ve^{\alpha}} \ve^{2\alpha} \delta^2 \le C \delta \ve^{\alpha}.
\end{aligned}
\end{equation}
Here, the first inequality is just Cauchy's inequality. For the third inequality, we used Lemma \ref{lem:regularity_Vu} and the Gronwall inequality \eqref{eq:Gronwall2}. In the fourth inequality, we used $\abs{x_1-x_2}=\delta$. We give an explanation of the second inequality in \eqref{eq:lem4_step23} next. After confirming this second inequality, we get \eqref{eq:lem4step2temp4}.

Although $x^{1,\ve}_0=x_1$ and $x^{2,\ve}_0=x_2$ are fixed points, we are integrating $x_1$ over $\mX$ (with $x_2-x_1 = \delta e_i$ fixed). So, we can define two new processes $\overline{x}^{1,\ve}_t$ and $\overline{x}^{2,\ve}_t$ that have the same dynamic as $x^{1,\ve}_t$ and $x^{2,\ve}_t$, but start at uniform distribution in $\mX$, with $\overline{x}^{2,\ve}_t - \overline{x}^{1,\ve}_t \equiv \delta e_i$. The densities for $\overline{x}^{2,\ve}_t$ and $\overline{x}^{1,\ve}_t$ (denoted by $\overline{\rho}^{1,\ve}(t,x)$ and $\overline{\rho}^{2,\ve}(t,x)$) satisfies the estimate in Proposition \ref{prop:rho}. Therefore,
\begin{equation}\label{eq:view_uniform}
\begin{aligned}
& \quad \int_{\mX} \EE \sqbra{\int_0^{\ve^{2\alpha}} \abs{\phi^{1,\ve}_t}^2 \rd t} \rd x_1 \equiv \int_{\mX} \EE \sqbra{\int_0^{\ve^{2\alpha}} \abs{\phi(t,\xyet)}^2 \rd t ~\Big|~ x^{1,\ve}_0 = x_1} \rd x_1\\
&= \EE_{\overline{x}^{1,\ve}_0 \sim \text{Unif}(\mX)} \EE \sqbra{\int_0^{\ve^{2\alpha}} \abs{\phi(t, \overline{x}^{1,\ve}_t)}^2 \rd t ~\Big|~ \overline{x}^{1,\ve}_0} = \EE \int_0^{\ve^{2\alpha}} \abs{\phi(t, \overline{x}^{1,\ve}_t)}^2 \rd t \\
& = \int_{\mX} \int_0^{\ve^{2\alpha}} \abs{\phi(t, x)}^2 \overline{\rho}^{1,\ve}(t,x) \, \rd t \, \rd x \le \rho_1 \int_{\mX} \int_0^{\ve^{2\alpha}} \abs{\phi(t, x)}^2 \rd t \, \rd x \le \rho_1 \norm{\phi}_{L^2}^2.
\end{aligned}
\end{equation}
$\phi^{2,\ve}_t$ satisfies the same inequality. The analysis for the $\nx V$ terms in the second line of \eqref{eq:lem4_step23} are exactly the same. Therefore, we can apply Proposition \ref{prop:rho}, and the second inequality in \eqref{eq:lem4_step23} holds. Hence, we confirm that \eqref{eq:lem4step2temp4} holds.

\emph{Step 4.} We estimate the integration in the interval $[\ve^{2\alpha},T]$ in this step. We want to show 
\begin{equation}\label{eq:lem4step2temp5}
\int_{\ve^{2\alpha}}^T \int_{\mX} \abs{  \inner{\nbu G\parentheses{t,x,u^{\ve}, -\nx V^{\ve}(t,x)}}{\phi(t,x)} \parentheses{\rho^{1,\ve}(t,x) - \rho^{2,\ve}(t,x)} }\rd x\, \rd t \le C \delta \ve^{\alpha}.
\end{equation}
We recall that $\rho$ is the solution of the Fokker Planck equation $\partial_t \rho = \mI_{\ve}^{\dagger} \rho$, where $\mI_{\ve}$ is the infinitesimal generator of the state process with control $u^{\ve}$ and $\mI_{\ve}^{\dagger}$ is its adjoint. Let us use $p^{\ve}(t,x;s,y)$ ($t \ge s$) to denote the fundamental solution of this PDE. Then, $\rho^{j,\ve}(t,x) = p^{\ve}(t,x;0,x_j)$ for $j=1,2$. The fundamental solution of linear parabolic PDE is well-studied, and a comprehensive description can be found in \cite{friedman2008partial}. A key observation of the fundamental solution $p^{\ve}$ is that $q^{\ve}(t,x;s,y) := p^{\ve}(s,y;t,x)$ is the fundamental solution of the backward Kolmogorov equation $\partial_t \psi + \mI_{\ve} \psi = 0$ \citep{ito1953fundamental}. Therefore, the regularity of $p^{\ve}(t,x;s,y)$ in $y$ (here $t \ge s$) is equivalent to the regularity of $q^{\ve}(t,x;s,y)$ in $x$ (here $s \ge t$).  \cite{aronson1959fundamental} proved (in Lemma 4.2) that
\begin{equation}\label{eq:lemma4.2}
\abs{\nx^k q^{\ve}(t,x;s,y)} \le C^{(k)} (s-t)^{-(n+k)/2}.
\end{equation}
Applying a standard mean value theorem and this lemma \eqref{eq:lemma4.2} with $k=1$ to $q^{\ve}$, we obtain
\begin{equation}\label{eq:diff_rho}
\begin{aligned}
& \quad \abs{\rho^{1,\ve}(t,x) - \rho^{2,\ve}(t,x)} = \abs{ q^{\ve}(0,x_1;t,x) - q^{\ve}(0,x_2;t,x) } \\
& = \abs{\inner{\nx q^{\ve}(0,(1-c)x_1 + cx_2;t,x)}{x_1 - x_2} } \le C t^{-(1+n)/2} \abs{x_1-x_2} = C t^{-(1+n)/2} \delta,
\end{aligned}
\end{equation}
where we clarify that $\nx$ is operated on the second (not fourth) argument on $q^{\ve}(t,x;s,y)$. Therefore,
\begin{equation*}
\begin{aligned}
& \quad \int_{\ve^{2\alpha}}^T \int_{\mX} \abs{  \inner{\nbu G\parentheses{t,x,u^{\ve}, -\nx V^{\ve}(t,x)}}{\phi(t,x)} \parentheses{\rho^{1,\ve}(t,x) - \rho^{2,\ve}(t,x)} }\rd x\, \rd t \\
& \le C \int_{\ve^{2\alpha}}^T \int_{\mX} \parentheses{\ve\abs{\phi(t,x)} + \abs{\nx V^{\ve}(t,x) - \nx V^*(t,x)} } \abs{\phi(t,x)} t^{-(1+n)/2} \delta \, \rd x\, \rd t \\
& \le C \delta \ve^{-\alpha-n\alpha} \int_{\ve^{2\alpha}}^T \int_{\mX} \parentheses{ 
\ve \abs{\phi(t,x)}^2 + \dfrac{1}{\ve} \abs{\nx V^{\ve}(t,x) - \nx V^*(t,x)}^2 }\rd x\, \rd t \\
& \le C \delta \ve^{-\alpha-n\alpha} \parentheses{\ve \norm{\phi}_{L^2}^2 + \dfrac{1}{\ve} \norm{\nx V^\ve - \nx V^*}_{L^2}^2} \le C \delta \ve^{1-\alpha-n\alpha} \le C \delta \ve^{\alpha}.
\end{aligned}
\end{equation*}
We used \eqref{eq:bound_nuGve} and \eqref{eq:diff_rho} in the first inequality, and used Lemma \ref{lem:regularity_Vu} in the thourth inequality. So, \eqref{eq:lem4step2temp5} holds.

To conclude, we combine \eqref{eq:lem4step2temp4} and \eqref{eq:lem4step2temp5} and recover \eqref{eq:lem4step2temp2}. Therefore, \eqref{eq:lem4step2temp} and hence \eqref{eq:quadratic_Vu} hold.
\end{proof}

\begin{lem}\label{lem:rho_speed}
Let Assumption \ref{assump:basic}, \ref{assump:u_smooth} hold, then there exist a constant $C_5$ such that
\begin{equation}\label{eq:rho_speed}
\dfrac{\rd}{\rd \tau} \rho^\tau(t,x) \le C_5 \alpha_a
\end{equation}
for all $(t,x) \in [0,T] \times \mX$ and $\tau \ge 0$.
\end{lem}
Recall that $\alpha_a$ is the actor speed and $\rho^\tau$ is short for $\rho^{u^\tau}$. This lemma informs us that the density is changing with a bounded speed.
\begin{proof}
Let $u^{\tau,\ve} = u^\tau + \ve \phi$ where
$$\phi(t,x) = \alpha_a \rho^\tau(t,x) \nbu G\parentheses{t,x,u^\tau(t,x),-\mG^\tau(t,x)}.$$
Here $\abs{\ve}$ is small such that $u^\ve \in \mU$. We denote $\rho^{u^{\tau,\ve}}$ by $\rho^{\tau,\ve}$. Then,
$$\dfrac{\rd}{\rd \tau} \rho^\tau(t,x) = \dfrac{\rd}{\rd \ve} \rho^{\tau,\ve}(t,x) \Big|_{\ve=0}.$$

We will derive a backward parabolic equation for $\dfrac{\rd}{\rd \tau} \rho^\tau(t,x)$ and use Feynman-Kac formula to represent $\partial_\ve \rho^{\tau,\ve}(t,x)$ and show its boundedness.

Taking derivative of the Fokker Planck w.r.t. $\ve$, we obtain
\begin{equation}\label{eq:dFP_dve}
\begin{aligned}
\pt \pve \rho^{\tau,\ve}(t,x) &= -\nx \cdot \sqbra{b(x, u^\ve(t,x)) \, \pve\rho^{\tau,\ve}(t,x)} + \sum_{i,j=1}^n \partial_i \partial_j \sqbra{D_{ij}(x) \,\pve\rho^{\tau,\ve}(t,x)}\\
& \quad - \nx \cdot \sqbra{\nbu b(x, u^\ve(t,x)) \,\phi(t,x) \, \rho^{\tau,\ve}(t,x)}
\end{aligned}
\end{equation}
In order to show the bound \eqref{eq:rho_speed}, we rewrite \eqref{eq:dFP_dve} into a backward parabolic equation for $\pve \rho^{\tau,\ve}(t,x)$. We denote $\rho_b^\ve(t,x) := \pve \rho^{\tau,\ve}(T-t,x)$, where the subscript ``$b$'' means backward. It satisfies $\pve \rho_b^\ve(T,x) = 0$ and
\begin{equation*}
\begin{aligned}
-\pt \rho_b^\ve(t,x) &= -\nx \cdot \sqbra{b(x, u^\ve(T-t,x)) \, \rho_b^\ve(t,x)} + \sum_{i,j=1}^n \partial_i \partial_j \sqbra{D_{ij}(x) \,\rho_b^\ve(t,x)}\\
& \quad - \nx \cdot \sqbra{\nbu b(x, u^\ve(T-t,x)) \,\phi(T-t,x) \, \rho^{\tau,\ve}(T-t,x)},
\end{aligned}
\end{equation*}
which simplifies to
\begin{equation}
\begin{aligned}
0 & = \pt \rho_b^\ve(t,x) + \Tr\sqbra{D(x) \nx^2 \rho_b^\ve(t,x)} \\
& \quad + \sum_{i,j=1}^n \parentheses{\partial_i D_{ij}(x) \partial_j \rho_b^\ve(t,x) + \partial_j D_{ij}(x) \partial_i \rho_b^\ve(t,x)} + \inner{b(x, u^\ve(T-t,x))}{\nx \rho_b^\ve(t,x)} \\
& \quad + \sqbra{\sum_{i,j=1}^n \partial_i \partial_j D_{i,j}(x) - \nx \cdot \parentheses{ b(x, u^\ve(T-t,x))}} \rho_b^\ve(t,x)\\
& \quad -\nx \cdot \sqbra{\nbu b(x,u^\ve(T-t,x)) \, \phi(T-t,x) \, \rho^{\tau,\ve}(T-t,x)} \\
& =: \pt \rho_b^\ve(t,x) + \Tr\sqbra{D(x) \nx^2 \rho_b^\ve(t,x)} + \inner{b^\ve_b(t,x)}{\nx \rho_b^\ve(t,x)} + V_b^\ve(x) \rho_b^\ve(t,x) + f_b^\ve(t,x).
\end{aligned}
\end{equation}
Therefore, if we define an SDE on $\mX$
$$\rd x^{b,\ve}_t = b_b^\ve(t,x^{b,\ve}_t) \,\rd t + \sigma(x^{b,\ve}_t) \,\rd W_t,$$
then, using the Feynman Kac formula, (and recall $\rho_b^\ve(T,x)=0$,) we know that
\begin{equation}\label{eq:FK_rho}
\rho_b^\ve(t,x) = \EE \sqbra{ \int_t^T \exp\parentheses{-\int_t^s V_b^\ve(t',x^{b,\ve}_{t'}) \,\rd t'} f^\ve_b(s,x^{b,\ve}_s) \,\rd s ~\bigg|~ x^{b,\ve}_t=x }.
\end{equation}
Next, we show boundedness of $V^\ve_b$ and $f^\ve_b$ to recover \eqref{eq:rho_speed}. Since Assumption \ref{assump:basic} holds and $u^\ve \in \mU$,
$$\abs{V^\ve_b(t,x)} = \abs{\sum_{i,j=1}^n \partial_i \partial_j D_{i,j}(x) - \nx \parentheses{b(x, u^\ve(T-t,x))}} \le nK + K + K^2.$$
Therefore, $V^\ve_b$ is bounded. For $f^\ve_b$,
\begin{equation}
\begin{aligned}
& \quad f^\ve_b(T-t,x) = -\nx \cdot \sqbra{\nbu b(x,u^\ve(t,x)) \, \phi(t,x) \, \rho^{\tau,\ve}(t,x)}\\
& = -\alpha_a\nx \cdot \sqbra{\nbu b(x,u^\ve(t,x)) ~~ \nbu G\parentheses{t,x,u^\tau(t,x),-\mG^\tau(t,x)} \, \rho^{\tau,\ve}(t,x)^2}\\
& = \alpha_a\nx \cdot \sqbra{\nbu b(x,u^\ve(t,x)) ~ \nbu b(x,u^\ve(t,x))\tp \mG^\tau(t,x) ~ \rho^{\tau,\ve}(t,x)^2} \\
& \quad + \alpha_a\nx \cdot \sqbra{\nbu b(x,u^\ve(t,x)) ~ \nbu r(x,u^\ve(t,x)) ~ \rho^{\tau,\ve}(t,x)^2}
% & = -\alpha_a \nx \cdot \parentheses{\nbu b(x,u^\ve(t,x))}\tp \nbu G\parentheses{t,x,u^\tau(t,x),-\mG^\tau(t,x)} ~~ \rho^{\tau,\ve}(t,x)^2 \\
% & \quad - \alpha_a \Tr\sqbra{\nbu b(x,u^\ve(t,x))\tp ~~ \nx \parentheses{\nbu G\parentheses{t,x,u^\tau(t,x),-\mG^\tau(t,x)}}} \rho^{\tau,\ve}(t,x)^2 \\
% & \quad - 2 \alpha_a \inner{\nbu b(x,u^\ve(t,x))\tp ~~ \nbu G\parentheses{t,x,u^\tau(t,x),-\mG^\tau(t,x)}}{\nx \rho^{\tau,\ve}(t,x)} \rho^{\tau,\ve}(t,x)
\end{aligned}
\end{equation}
Next, we show boundedness of each term above. $\abs{\nbu b}, \abs{\nbu r} \le K$ because of Assumption \ref{assump:basic}.  $\abs{\nx \nbu b(x,u^\ve(t,x))}, \abs{\nx \nbu r(x,u^\ve(t,x))} \le K + K^2$ because Assumption \ref{assump:basic} holds and $u^\ve \in \mU$. $\abs{\mG}, \abs{\nx \mG} \le K$ because of Assumption \ref{assump:u_smooth}. $\abs{\rho^{\tau,\ve}} \le \rho_1$ and $\abs{\nx \rho^{\tau,\ve}(t,x)} \le \rho_2$ because of Proposition \ref{prop:rho}. Concluding these bounds, we know that $\abs{f_b^\ve(t,x)} \le C \alpha_a$ for any $t\in[0,T]$ and $x\in \mX$.

Combining the bounds for $V^\ve_b$ and $f^\ve_b$ into \eqref{eq:FK_rho}, we obtain
$$\abs{\pve \rho^{\tau,\ve}(t,x)} = \abs{\rho_b^\ve(T-t,x)} \le C_5 \alpha_a.$$
Therefore, 
\begin{equation*}
\dfrac{\rd}{\rd \tau} \rho^\tau(t,x) \le C_5 \alpha_a. 
\end{equation*}
\end{proof}

\section{Proofs for the theorems}\label{sec:thms}

\begin{proof}[Proof for theorem \ref{thm:critic_improvement}]
As is explained in Section \ref{sec:critic} \eqref{eq:critic_loss2}, the critic loss $\mL_c^\tau$ is the sum of the $L^2$ error for $\mV_0$
$$\mL_0^\tau = \frac12 \int_\mX \parentheses{\mV_0^\tau(x) - V_{u^\tau}(0,x)}^2 \rd x$$
and the gradient error
$$\mL_1^\tau = \inttx{ \rho^\tau(t,x) \abs{\sigma(x)\tp \parentheses{\mG^\tau(t,x) - \nx V_{u^\tau}(t,x)}}^2 }.$$
Taking derivative w.r.t. $\tau$ and plug in the critic dynamic \eqref{eq:V0dynamic}, we obtain
\begin{equation}\label{eq:dL0_dtau}
\begin{aligned}
& \quad \dfrac{\rd}{\rd \tau} \mL_0^\tau = \int_\mX \parentheses{\mV_0^\tau(x) - V_{u^\tau}(0,x)} \parentheses{\dfrac{\rd}{\rd \tau}\mV_0^\tau(x) - \dfrac{\rd}{\rd \tau}V_{u^\tau}(0,x)} \rd x \\
& = - \alpha_c \int_\mX \parentheses{\mV_0^\tau(x) - V_{u^\tau}(0,x)}^2 \rd x - \int_\mX \parentheses{\mV_0^\tau(x) - V_{u^\tau}(0,x)} \dfrac{\rd}{\rd \tau} V_{u^\tau}(0,x) \rd x\\
& \le - 2 \alpha_c \mL_0^\tau + \frac12 \alpha_c \int_\mX \parentheses{\mV_0^\tau(x) - V_{u^\tau}(0,x)}^2 \rd x + \dfrac{1}{2\alpha_c} \int_\mX\parentheses{\dfrac{\rd}{\rd \tau} V_{u^\tau}(0,x)}^2 \rd x \\
& = - \alpha_c \mL_0^\tau + \dfrac{1}{2\alpha_c} \int_\mX\parentheses{\dfrac{\rd}{\rd \tau} V_{u^\tau}(0,x)}^2 \rd x,
\end{aligned}
\end{equation}
where the inequality is because of Cauchy Schwartz inequality. The last term in \eqref{eq:dL0_dtau} satisfies
\begin{equation}\label{eq:dVdtau}
\begin{aligned}
& \quad \int_\mX\parentheses{\dfrac{\rd}{\rd \tau} V_{u^\tau}(0,x)}^2 \rd x = \int_\mX 
\lim_{\Delta\tau \to 0} \dfrac{1}{\Delta\tau^2} \parentheses{V_{u^{\tau+\Delta\tau}}(0,x) - V_{u^{\tau}}(0,x)}^2 \rd x\\
& \le \liminf_{\Delta\tau \to 0} \dfrac{1}{\Delta\tau^2} \int_\mX \parentheses{V_{u^{\tau+\Delta\tau}}(0,x) - V_{u^{\tau}}(0,x)}^2 \rd x \\
& \le C_2 \liminf_{\Delta\tau \to 0} \dfrac{1}{\Delta\tau^2} \norm{u^{\tau+\Delta\tau} - u^\tau}^2_{L^2},
\end{aligned}
\end{equation}
where we used Fatou's Lemma and \emph{Step 1} in Lemma \ref{lem:regularity_Vu} in the two inequalities. Since
\begin{align*}
& \quad u^{\tau+\Delta\tau}(t,x) - u^\tau(t,x) = \int_\tau^{\tau+\Delta\tau} \dfrac{\rd}{\rd \tau'} u^{\tau'}(t,x) \rd x\\
& = \alpha_a \int_\tau^{\tau+\Delta\tau} \rho^{\tau'}(t,x) \nbu G\parentheses{t,x,u^{\tau'}(t,x), -\mG^{\tau'}(t,x)} \,\rd \tau',
\end{align*}
we have
\begin{equation}\label{eq:diff_unorm}
\begin{aligned}
& \quad \norm{u^{\tau+\Delta\tau} - u^\tau}_{L^2}^2 \\
& = \inttx{ \parentheses{\alpha_a \int_\tau^{\tau+\Delta\tau} \rho^{\tau'}(t,x) \abs{\nbu G\parentheses{t,x,u^{\tau'}(t,x), -\mG^{\tau'}(t,x)}} \,\rd \tau'}^2 } \\
& \le \alpha_a^2 \rho_1 \inttx{ \Delta\tau \int_\tau^{\tau+\Delta\tau} \rho^{\tau'}(t,x) \abs{\nbu G\parentheses{t,x,u^{\tau'}(t,x), -\mG^{\tau'}(t,x)}}^2 \,\rd \tau'}, 
\end{aligned}
\end{equation}
where we used Proposition \ref{prop:rho} and Cauchy Schwartz inequality. Substituting \eqref{eq:diff_unorm} into \eqref{eq:dVdtau}, we obtain
\begin{equation}\label{eq:dVdtau2}
\begin{aligned}
& \quad \int_\mX\parentheses{\dfrac{\rd}{\rd \tau} V_{u^\tau}(0,x)}^2 \rd x \\
& \le  C_2 \liminf_{\Delta\tau \to 0} \dfrac{1}{\Delta\tau} \alpha_a^2 \rho_1 \inttx{ \int_\tau^{\tau+\Delta\tau} \rho^{\tau'}(t,x) \abs{\nbu G\parentheses{t,x,u^{\tau'}(t,x), -\mG^{\tau'}(t,x)}}^2 \,\rd \tau'} \\
& = C_2 \liminf_{\Delta\tau \to 0} \dfrac{1}{\Delta\tau} \alpha_a^2 \rho_1 \int_\tau^{\tau+\Delta\tau} \inttx{ \rho^{\tau'}(t,x) \abs{\nbu G\parentheses{t,x,u^{\tau'}(t,x), -\mG^{\tau'}(t,x)}}^2 } \,\rd \tau' \\
& = C_2 \alpha_a^2 \rho_1 \inttx{\rho^{\tau}(t,x) \abs{\nbu G\parentheses{t,x,u^\tau(t,x), -\mG^\tau(t,x)}}^2 }.
\end{aligned}
\end{equation}
Combining \eqref{eq:dL0_dtau} and \eqref{eq:dVdtau2}, we recover 
\begin{equation}\label{eq:critic_improvement0}
\dfrac{\rd}{\rd \tau} \mL_0^\tau \le - \alpha_c \mL_0^\tau + \dfrac{\alpha_a^2}{\alpha_c} \dfrac{C_2  \rho_1}{2} \inttx{ \rho^{\tau}(t,x) \abs{\nbu G\parentheses{t,x,u^{\tau'}(t,x), -\mG^\tau(t,x)}}^2 }.
\end{equation}

Next, we consider $\mL_1^\tau$. Taking derivative w.r.t. $\tau$, we obtain
\begin{equation}\label{eq:dL1_dtau}
\begin{aligned}
& \quad \dfrac{\rd}{\rd \tau} \mL_1^\tau = \frac12 \dfrac{\rd}{\rd \tau} \inttx{ \rho^\tau(t,x) \abs{\sigma(x)\tp \parentheses{\mG^\tau(t,x) - \nx V_{u^\tau}(t,x)}}^2 } \\
& = \frac12 \inttx{ \dfrac{\rd}{\rd \tau} \rho^\tau(t,x) \abs{\sigma(x)\tp \parentheses{\mG^\tau(t,x) - \nx V_{u^\tau}(t,x)}}^2 }\\
& \quad + \inttx{ \rho^\tau(t,x) \parentheses{\mG^\tau(t,x) - \nx V_{u^\tau}(t,x)}\tp \sigma(x)\sigma(x)\tp \dfrac{\rd}{\rd \tau} \mG^\tau(t,x) \,} \\
& \quad + \inttx{ \rho^\tau(t,x) \parentheses{\nx V_{u^\tau}(t,x) - \mG^\tau(t,x)}\tp \sigma(x)\sigma(x)\tp \dfrac{\rd}{\rd \tau} \nx V_{u^\tau}(t,x) \,}\\
& =: (\rom{1}) + (\rom{2}) + (\rom{3}).
\end{aligned}
\end{equation}
We analyze these three terms next. The first term satisfies
\begin{equation}\label{eq:L1_term1}
(\rom{1}) \le \frac12 C_5 \alpha_a \inttx{ \abs{\sigma(x)\tp \parentheses{\mG^\tau(t,x) - \nx V_{u^\tau}(t,x)}}^2 } \le \dfrac{\alpha_a C_5}{\rho_0} \mL_1^\tau,
\end{equation}
where we used Lemma \ref{lem:rho_speed} and Proposition \ref{prop:rho} in the two inequalities.

The second term
\begin{equation}\label{eq:L1_term2}
\begin{aligned}
(\rom{2}) & = - \alpha_c \int_0^T \int_\mX \rho^\tau(t,x) \parentheses{\mG^\tau(t,x) - \nx V_{u^\tau}(t,x)}\tp \sigma(x) \sigma(x)\tp\\
& \hspace{1in} \rho^\tau(t,x) \,\sigma(x) \sigma(x)\tp \parentheses{\mG^\tau(t,x) - \nx V_{u^\tau}(t,x)} \, \rd x \, \rd t\\
& \le - 2\alpha_c\, \rho_0 \,\sigma_0 \inttx{ \rho^\tau(t,x) \abs{\sigma(x)\tp \parentheses{\mG^\tau(t,x) - \nx V_{u^\tau}(t,x)}}^2 \,}\\
& = - 4\alpha_c\, \rho_0\, \sigma_0 \,\mL_1^\tau
\end{aligned}
\end{equation}
where we used the critic dynamic \eqref{eq:Gdynamic} in the first equality, and Proposition \ref{prop:rho} and uniform ellipticity in the inequality.

We estimate $(\rom{3})$ next. By Cauchy Schwartz inequality,
\begin{equation}\label{eq:critic_term3}
\begin{aligned}
(\rom{3}) & \le \frac12 \alpha_c \, \rho_0 \, \sigma_0 \inttx{ \rho^\tau(t,x) \abs{\sigma(x)\tp \parentheses{\nx V_{u^\tau}(t,x) - \mG^\tau(t,x)} }^2 \,} \\
& \quad + \dfrac{1}{2 \alpha_c \, \rho_0 \, \sigma_0} \inttx{ \rho^\tau(t,x) \abs{\sigma(x)\tp \dfrac{\rd}{\rd \tau} \nx V_{u^\tau}(t,x) }^2 \,} \\
& = \alpha_c\, \rho_0\, \sigma_0 \,\mL_1^\tau + \dfrac{1}{2 \alpha_c \, \rho_0 \, \sigma_0} \inttx{ \rho^\tau(t,x) \abs{\sigma(x)\tp \dfrac{\rd}{\rd \tau} \nx V_{u^\tau}(t,x) }^2 \,}.
\end{aligned}
\end{equation}
For the last term in \eqref{eq:critic_term3}, we have 
\begin{equation}\label{eq:term3_temp}
\begin{aligned}
& \quad \inttx{ \rho^\tau(t,x) \abs{\sigma(x)\tp \dfrac{\rd}{\rd \tau} \nx V_{u^\tau}(t,x) }^2 \,} \\
& \le \rho_1 K^2 \inttx{ \abs{ \dfrac{\rd}{\rd \tau} \nx V_{u^\tau}(t,x) }^2 \,} \\
& = \rho_1 K^2 \inttx{ \abs{ \lim_{\Delta\tau \to 0} \dfrac{1}{\Delta \tau} \parentheses{\nx V_{u^{\tau+\Delta\tau}}(t,x) - \nx V_{u^\tau}(t,x)}  }^2 \,} \\
& \le \rho_1 K^2 \liminf_{\Delta\tau \to 0} \dfrac{1}{\Delta \tau^2} \inttx{ \abs{ \nx V_{u^{\tau+\Delta\tau}}(t,x) - \nx V_{u^\tau}(t,x) }^2 \,} \\
& \le \rho_1 K^2 \liminf_{\Delta\tau \to 0} \dfrac{C_2^2}{\Delta \tau^2} \norm{u^{\tau + \Delta\tau} - u^\tau}_{L^2}^2,
\end{aligned}
\end{equation}
where we have consecutively used: upper bound of $\rho(t,x)$ in Proposition \ref{prop:rho} and $\sigma(x)$ in Assumption \ref{assump:basic}; definition of derivative; Fatou's lemma; Lemma \ref{lem:regularity_Vu}. Substituting \eqref{eq:diff_unorm} into \eqref{eq:term3_temp}, we obtain
\begin{equation}\label{eq:term3_temp3}
\begin{aligned}
& \quad \inttx{ \rho^\tau(t,x) \abs{\sigma(x)\tp \dfrac{\rd}{\rd \tau} \nx V_{u^\tau}(t,x) }^2 \,} \\
& \le \rho_1^2 K^2 C_2^2 \liminf_{\Delta\tau \to 0} \dfrac{\alpha_a^2}{\Delta\tau} \inttx{\int_\tau^{\tau + \Delta\tau} \rho^{\tau'}(t,x) \abs{\nbu G\parentheses{t,x,u^{\tau'}(t,x), -\mG^{\tau'}(t,x)}}^2 \rd\tau'}\\
& = \rho_1^2 K^2 C_2^2 \liminf_{\Delta\tau \to 0} \dfrac{\alpha_a^2}{\Delta\tau} \int_\tau^{\tau + \Delta\tau} \inttx{\rho^{\tau'}(t,x) \abs{\nbu G\parentheses{t,x,u^{\tau'}(t,x), -\mG^{\tau'}(t,x)}}^2} \,\rd\tau'\\
& = \rho_1^2 K^2 C_2^2 \alpha_a^2 \inttx{\rho^{\tau}(t,x) \abs{\nbu G\parentheses{t,x,u^{\tau}(t,x), -\mG^{\tau}(t,x)}}^2}
\end{aligned}
\end{equation}
Combining \eqref{eq:term3_temp3} with \eqref{eq:critic_term3}, we get
\begin{equation}\label{eq:L1_term3}
\begin{aligned}
& \quad (\rom{3})  \\
& \le  \alpha_c\, \rho_0\, \sigma_0 \,\mL_1^\tau + \dfrac{1}{2 \alpha_c \, \rho_0 \, \sigma_0} \inttx{ \rho^\tau(t,x) \abs{\sigma(x)\tp \dfrac{\rd}{\rd \tau} \nx V_{u^\tau}(t,x) }^2 \,} \\
& \le \alpha_c\, \rho_0\, \sigma_0 \,\mL_1^\tau + \dfrac{\alpha_a^2 \rho_1^2 K^2 C_2^2}{2 \alpha_c \, \rho_0 \, \sigma_0} \inttx{\rho^{\tau}(t,x) \abs{\nbu G\parentheses{t,x,u^{\tau}(t,x), -\mG^{\tau}(t,x)}}^2}.
\end{aligned}
\end{equation}
Combining the estimation for $(\rom{1})-(\rom{3})$ in \eqref{eq:L1_term1}, \eqref{eq:L1_term2}, and \eqref{eq:L1_term3} into \eqref{eq:dL1_dtau}, we get 
\begin{equation}\label{eq:critic_improvement1}
\begin{aligned}
& \quad \dfrac{\rd}{\rd \tau} \mL_1^\tau \\
& \le \dfrac{\alpha_a C_5}{\rho_0} \mL_1^\tau - 4\alpha_c\, \rho_0\, \sigma_0 \,\mL_1^\tau + \alpha_c\, \rho_0\, \sigma_0 \,\mL_1^\tau \\
& \quad + \dfrac{\alpha_a^2 \rho_1^2 K^2 C_2^2}{2 \alpha_c \, \rho_0 \, \sigma_0} \inttx{\rho^{\tau}(t,x) \abs{\nbu G\parentheses{t,x,u^{\tau}(t,x), -\mG^{\tau}(t,x)}}^2}\\
& \le - 2\alpha_c\, \rho_0\, \sigma_0 \,\mL_1^\tau + \dfrac{\alpha_a^2}{\alpha_c} \dfrac{\rho_1^2 K^2 C_2^2}{2 \, \rho_0 \, \sigma_0} \inttx{\rho^{\tau}(t,x) \abs{\nbu G\parentheses{t,x,u^{\tau}(t,x), -\mG^{\tau}(t,x)}}^2},
\end{aligned}
\end{equation}
where we used \eqref{eq:speed_ratio} in the last inequality. Finally, combining \eqref{eq:critic_improvement0} and \eqref{eq:critic_improvement1}, we recover \eqref{eq:critic_improvement}.
\end{proof}

\begin{proof}[Proof for theorem \ref{thm:actor_improvement}]
By direct computation, we have
\begin{equation*}
\begin{aligned}
& \quad \dfrac{\rd}{\rd \tau} J[u^{\tau}] = \inner{\fd{J}{u}[u^{\tau}]}{\dfrac{\rd}{\rd \tau} u^{\tau}}_{L^2} \\
& = - \alpha_a \inttx{ \rho^{u^\tau}(t,x)^2 \inner{\nbu G\parentheses{t,x,u^\tau(t,x),-\nx V_{u^\tau}(t,x)}}{\nbu G\parentheses{t,x,u^\tau(t,x),-\mG^\tau(t,x)}}} \\
& = - \frac12 \alpha_a \inttx{\rho^{u^\tau}(t,x)^2 \abs{\nbu G\parentheses{t,x,u^\tau(t,x),-\nx V_{u^\tau}(t,x)}}^2} \\
& \quad - \frac12 \alpha_a \inttx{\rho^{u^\tau}(t,x)^2 \abs{\nbu G\parentheses{t,x,u^\tau(t,x),-\mG^\tau(t,x)}}^2} \\
& \quad + \dfrac{\alpha_a}{2} \inttx{\rho^{u^\tau}(t,x)^2 \abs{\nbu G\parentheses{t,x,u^\tau(t,x),-\nx V_{u^\tau}(t,x)} - \nbu G\parentheses{t,x,u^\tau(t,x),-\mG^\tau(t,x)}}^2}.
\end{aligned}
\end{equation*}
Therefore, we have
\begin{equation}\label{eq:chain_rule}
\begin{aligned}
& \quad \dfrac{\rd}{\rd \tau} J[u^{\tau}]\\
& \le - \frac12 \alpha_a \rho_0^2 \inttx{\abs{\nbu G\parentheses{t,x,u^\tau(t,x),-\nx V_{u^\tau}(t,x)}}^2} \\
& \quad - \frac12 \alpha_a \rho_0 \inttx{\rho^\tau(t,x) \abs{\nbu G\parentheses{t,x,u^\tau(t,x),-\mG^\tau(t,x)}}^2} \\
& \quad + \dfrac{\alpha_a}{2} \rho_1^2 \inttx{\abs{\nbu G\parentheses{t,x,u^\tau(t,x),-\nx V_{u^\tau}(t,x)} - \nbu G\parentheses{t,x,u^\tau(t,x),-\mG^\tau(t,x)}}^2}
\end{aligned}
\end{equation}
where we have used Proposition \ref{prop:rho}. Note that the last term in \eqref{eq:chain_rule} satisfies
\begin{equation}
\begin{aligned}
& \quad \inttx{\abs{\nbu G\parentheses{t,x,u^\tau(t,x),-\nx V_{u^\tau}(t,x)} - \nbu G\parentheses{t,x,u^\tau(t,x),-\mG^\tau(t,x)}}^2} \\
& = \inttx{\abs{\nbu b(x,u^\tau(t,x)) \parentheses{\nx V_{u^\tau}(t,x) - \mG^\tau(t,x)}}^2} \\
& \le K^2 \norm{\nx V - \mG}_{L^2}^2.
\end{aligned}
\end{equation}
This term comes from the critic error. Therefore, by \eqref{eq:chain_rule}, in order to prove \eqref{eq:actor_improvement}, is it sufficient to show
\begin{equation}\label{eq:actor_rate}
\inttx{\abs{\nbu G\parentheses{t,x,u^\tau(t,x),-\nx V_{u^\tau}(t,x)}}^2} \ge c \parentheses{J[u^\tau] - J[u^*]}
\end{equation}
for some constant $c>0$. This is very similar to the Polyak-{\L}ojasiewicz (PL) condition \citep{karimi2016linear}, which commonly appears in theoretical analysis for optimization. In order to seek for this condition, we make a technical definition. For any control function $u \in \mU$, we define its corresponding local optimal control function by
\begin{equation}\label{eq:udiamond}
u^{\diamond}(t,x) := \argmax_{u' \in \RR^{n'}} G(t,x,u',-\nx V_u(t,x), -\nx^2 V_u(t,x)).
\end{equation}
Since $G$ is strongly concave in $u$, $u^{\diamond}$ is well-defined.
Since the solution to the HJB equation is unique, $u \equiv u^{\diamond}$ if and only if $u$ is the optimal control.
By $\mu_G$-strong concavity of $G$ in $u$, we have
\begin{equation}\label{eq:strong_concave}
\begin{aligned}
& \quad \abs{\nbu G\parentheses{t,x,u(t,x),-\nx V_u(t,x)}} \\
& = \abs{\nbu G\parentheses{t,x,u(t,x),-\nx V_u(t,x)} - \nbu G\parentheses{t,x,u^{\diamond}(t,x),-\nx V_u(t,x)}} \\
& \ge \mu_G \abs{u(t,x) - u^{\diamond}(t,x)}.
\end{aligned}
\end{equation} 
With this definition, we state a crucial criterion for the PL condition: there exists positive constants $\mu_0$ and $\tau_0$ such that
\begin{equation}\label{eq:easy_case}
\norm{u^{\tau} - u^{\tau \diamond}}_{L^2} \ge \mu_0 \norm{u^{\tau} - u^*}_{L^2}
\end{equation}
for all $\tau \ge \tau_0$.
Under such condition, we have
\begin{equation}\label{eq:Polyak}
\begin{aligned}
& \quad \norm{\nbu G\parentheses{t,x,u^{\tau}(t,x),-\nx V_{u^\tau}(t,x)}}_{L^2}^2 \ge \mu_G^2 \norm{u^{\tau} - u^{\tau \diamond}}_{L^2}^2 \\
& \ge \, \mu_G^2 \, \mu_0^2 \norm{u^{\tau} - u^*}_{L^2}^2 \ge \mu_G^2 \, \mu_0^2 \frac{1}{C_3} \parentheses{J[u^{\tau}]- J[u^*]}
\end{aligned}
\end{equation}
when $\tau \ge \tau_0$, where we have consecutively used: estimate \eqref{eq:strong_concave}, condition \eqref{eq:easy_case}, and Lemma \ref{lem:J_quadratic}. This implies \eqref{eq:actor_rate} holds with $c= \mu_G^2 \, \mu_0^2 / C_3$, so the remaining task is to show \eqref{eq:easy_case}.

We will establish \eqref{eq:easy_case} by contradiction. If it does not hold, then there exists a sequence $\{\tau_k\}_{k=1}^\infty$, such that $\tau_k < \tau_{k+1}$, $\lim_{k \to \infty} \tau_k=\infty$, and
\begin{equation*}
\norm{u^{\tau_k} - u^{\tau_k\diamond}}_{L^2} \le \frac1k \norm{u^{\tau_k} - u^*}_{L^2}.
\end{equation*}
For simplicity, we denote $u^{\tau_k}$ by $u_k$ and the corresponding value function $V_{u^{\tau_k}}$ by $V_k$.
The above condition becomes
\begin{equation}\label{eq:hard_case}
\norm{u_k - u_k^{\diamond}}_{L^2} \le \frac1k \norm{u_k - u^*}_{L^2}.
\end{equation}
We denote
\begin{equation}\label{eq:V_infinity}
V_{\infty}(t,x) := \limsup_{k \to \infty} V_k(t,x).
\end{equation}
We claim that 
\begin{equation}\label{eq:claim}
V_{\infty}(t,x) \equiv V^*(t,x).
\end{equation}
The proof for this claim contains the most technical part, so we leave it to Lemma \ref{lem:claim} (which lies exactly after this proof) and focus on the rest part of the analysis first. The intuition for this lemma is that under \eqref{eq:hard_case}, when $k$ is large, $V_k$ is very close to the solution of the HJB equation in the sense that the maximal condition \eqref{eq:max1} is nearly satisfied (cf. \eqref{eq:udiamond}). Therefore, the idea for Lemma \ref{lem:claim} is to modify the proof for the uniqueness of the viscosity solution to the HJB equation and show $V_\infty$, as the limit of $V_k$, is the solution to the HJB equation.

Now, provided that the claim \eqref{eq:claim} holds, we know that
$$\limsup_{k \to \infty} V_k(t,x) = V^*(t,x) \le \liminf_{k \to \infty} V_k(t,x) ~~~~~~ \forall t,x,$$
where the last inequality is because of the optimality of $V^*$. Therefore,
$$\lim_{k \to \infty} V_k(t,x) = V^*(t,x) ~~~~~~ \forall t,x.$$
By the Lipschitz condition of the value function implied from Assumption \ref{assump:u_smooth} (as explained at the beginning of the appendix) and Arzel\'a--Ascoli theorem, we know that $V_k$ converges to $V^*$ uniformly. Therefore, using the definition
$$J[u_k] = \int_{\mX} \rho^{u_k}(0,x) V_k(0,x) \,\rd x = \int_{\mX} V_k(0,x) \,\rd x$$
and
$$J[u^*] = \int_{\mX} V^*(0,x)\, \rd x,$$
we know that 
$$\lim_{k \to \infty} J[u_k] = J[u^*].$$
By Assumption \ref{assump:actor_rate}, we further obtain
\begin{equation}\label{eq:uk_converge}
\lim_{k \to \infty} \norm{u_k - u^*}_{L^2} = 0.
\end{equation}

Next, we recall that we define the local optimal control function by \eqref{eq:udiamond}
\begin{equation*}
u^{\diamond}(t,x) := \argmax_{u' \in \RR^{n'}} G(t,x,u',-\nx V_u(t,x), -\nx^2 V_u(t,x)).
\end{equation*}
For fixed $(t,x)$, $u^{\diamond}(t,x)$ can be viewed as an implicit function of $-\nx V_u(t,x)$, given by the equation
$$\nbu G\parentheses{t,x,u^{\diamond}(t,x),-\nx V_u(t,x)} = 0.$$
Next, we want to show that this implicit function is Lipschitz. Computing the Jacobian of $\nbu G(t,x,u^{\diamond},p) = 0$ w.r.t. $p \in \RR^n$, we obtain
\begin{equation*}
0 = \nbu^2 G(t,x,u^{\diamond},p) \cdot \pd{u^{\diamond}}{p}\tp + \nbu b(x,u).
\end{equation*}
So
\begin{equation}\label{eq:Jacobian}
\pd{u^{\diamond}}{p}\tp = - \parentheses{\nbu^2 G(t,x,u^{\diamond},p)}^{-1} \nbu b(x,u).
\end{equation}
By assumption \ref{assump:basic}, $\abs{\nbu b(x,u)} \le K$. Since $G$ is $\mu_G$-strongly concave in $u$, $-(\nbu^2 G)^{-1}$ is positive definite with spectrum norm less than $1/\mu_G$. Therefore, \eqref{eq:Jacobian} implies
$$\abs{\pd{u^{\diamond}}{p}} \le \dfrac{K}{\mu_G}.$$
So the implicit function $p \mapsto u^\diamond$ induced by $\nbu G(t,x,u^\diamond,p)=0$ is Lipschitz. Therefore,
$$\abs{u_k^{\diamond}(t,x) - u^*(t,x)} \le \dfrac{K}{\mu_G} \abs{\nx V_k(t,x) - \nx V^*(t,x)},$$
hence
$$\norm{u_k^{\diamond} - u^*}_{L^2} \le \dfrac{K}{\mu_G}\norm{\nx V_k - \nx V^*}_{L^2} \le \dfrac{K}{\mu_G} C_4 \norm{u_k-u^*}_{L^2}^{1+\alpha}$$
where we used Lemma \ref{lem:quadratic_Vu} in the last inequality. Using triangle inequality, we get
$$\norm{u_k-u^*}_{L^2} \le \norm{u_k - u_k^{\diamond}}_{L^2} + \norm{u_k^{\diamond} - u^*}_{L^2} \le \frac1k \norm{u_k-u^*}_{L^2} + C\norm{u_k-u^*}_{L^2}^{1+\alpha}.$$
This gives a contradiction when $k$ is sufficiently large (which means $\norm{u_k-u^*}_{L^2}$ is sufficiently small because of \eqref{eq:uk_converge}) unless $u_k = u^*$. This contradiction comes from assumption \eqref{eq:hard_case}, which is the negation of \eqref{eq:easy_case}. Therefore, \eqref{eq:easy_case} must hold, and the theorem is proved.
\end{proof}

\begin{lem}[claim \eqref{eq:claim}]\label{lem:claim}
Let \eqref{eq:hard_case} and all the assumptions in Theorem \ref{thm:actor_improvement} hold. Then, the claim \eqref{eq:claim} holds, i.e.
$$\limsup_{k \to \infty} V_k(t,x) =: V_{\infty}(t,x) \equiv V^*(t,x).$$
\end{lem}
\begin{proof}
We will prove \eqref{eq:claim} by contradiction. We assume to the contrary that there exists $\btx \in [0,T] \times \mX$ s.t. $V_{\infty}\btx - V^*\btx \ge 2\eta >0$.

First, we claim that we can find $(\bt,\bx) \in (0,T] \times \mX$. This is because otherwise $\bt=0$ and $V_\infty(t,x) = V^*(t,x)$ for all $t>0$ and $x\in\mX$, then we can pick $\delta t = \eta / (3L)$ and obtain
\begin{equation}\label{eq:exclude0}
\begin{aligned}
& \quad \abs{V_\infty(0,\bx) - V^*(0,\bx)} \\
& \le \abs{V_\infty(0,\bx) - V_k(0,\bx)} + \abs{V_k(0,\bx) - V_k(\delta t,\bx)} \\
& \quad + \abs{V_k(\delta t,\bx) - V^*(\delta t,\bx)} + \abs{V^*(\delta t,\bx) - V^*(0,\bx)} %\\
%& < \frac12 \eta + L \,\delta t + \frac12 \eta + L \,\delta t < 2 \eta.
\end{aligned}
\end{equation}
for any $k$. We bound each of the term in \eqref{eq:exclude0} next. The first term is bounded by $\frac12 \eta$ because
$$\limsup_{k \to \infty} V_k(0 ,\bx) = V_\infty (0 ,\bx),$$
and for any $K_0>0$ we can find $k>K_0$ s.t. $\abs{V_k(0 ,\bx) - V_\infty (0 ,\bx)} < \frac12 \eta$. The second and forth terms are bounded by $L \,\delta t$ because of the Lipschitz condition induced by Assumption \ref{assump:basic}. For the third term, since
$$\limsup_{k \to \infty} V_k(t,x) = V_\infty (t,x) = V^*(t,x) \le \liminf_{k \to \infty} V_k(t,x)$$
in $[\frac12 \delta t, T] \times \mX$, we have $V_k$ converges to $V_\infty$ pointwise (and hence uniformly by Arzel\'a--Ascoli theorem) in $[\frac12 \delta t, T] \times \mX$. So the third term $\abs{V_k(\delta t,\bx) - V^*(\delta t,\bx)}$ is less than $\frac12 \eta$ as long as $k$ is sufficiently large. Combining the estimations for the four terms in \eqref{eq:exclude0}, we obtain
$$\abs{V_\infty(0,\bx) - V^*(0,\bx)} < \frac12 \eta + L \,\delta t + \frac12 \eta + L \,\delta t < 2 \eta,$$
which contradicts to $V_\infty(0,\bx) - V^*(0,\bx) \ge 2\eta$. Therefore, we can assume that $\bt > 0$.

Without loss of generality, we assume that
\begin{equation}\label{eq:contrary}
V_k\btx - V^*\btx \ge \eta >0 ~~~ \forall k.
\end{equation}
If necessary, we can extract a subsequence (without changing notations) to achieve \eqref{eq:contrary}, where \eqref{eq:hard_case} is still satisfied. Using the same trick, we may assume that
\begin{equation}\label{eq:limit_equal}
\lim_{k \to \infty} V_{k}\btx = V_\infty\btx.
\end{equation}

Next, for any $\ve,\,\delta,\,\lambda \in (0,1)$, we define two continuous functions on $(0,T] \times \mX \times (0,T] \times \mX$
\begin{equation}\label{eq:varphi}
\varphi(t,x,s,y) := \dfrac{1}{2\ve} \abs{t-s}^2 + \dfrac{1}{2\delta}\abs{x-y}^2 + \dfrac{\lambda}{t} + \dfrac{\lambda}{s}
\end{equation}
and
\begin{equation}\label{eq:Phi_k}
\Phi_k(t,x,s,y) := V_k(t,x) - V^*(s,y) - \varphi(t,x,s,y).
\end{equation}
Since $\lim_{t\wedge s \to 0+} \Phi_k(t,x,s,y) = - \infty$, and $\Phi_k$ is continuous, $\Phi_k(t,x,s,y)$ achieves its maximum at some point $(t_k, x_k, s_k, y_k) \in (0,T] \times \mX \times (0,T] \times \mX$. This $(t_k, x_k, s_k, y_k)$ depends on $\ve,\,\delta,\,\lambda$, and $k$. By this optimality, we have
$$2\Phi_k(t_k, x_k, s_k, y_k) \ge \Phi_k(t_k, x_k, t_k, x_k) + \Phi_k(s_k, y_k, s_k, y_k),$$
which simplifies to
\begin{equation*}
\begin{aligned}
\dfrac{1}{\ve} \abs{t_k-s_k}^2 + \dfrac{1}{\delta}\abs{x_k-y_k}^2 &\le V_k(t_k,x_k) - V_k(s_k,y_k) + V^*(t_k,x_k) - V^*(s_k,y_k) \\
&\le 2L \abs{(t_k,x_k) - (s_k,y_k)},
\end{aligned}
\end{equation*}
where we used the Lipschitz condition of $V_k$ and $V^*$ in the second inequality. So,
\begin{equation*}
\frac{1}{\ve + \delta}\abs{(t_k,x_k) - (s_k,y_k)}^2 = \frac{1}{\ve + \delta} \parentheses{\abs{t_k-s_k}^2 + \abs{x_k-y_k}^2} \le 2L \abs{(t_k,x_k) - (s_k,y_k)}.
\end{equation*}
This implies
\begin{equation}\label{eq:bound1}
\abs{(t_k,x_k) - (s_k,y_k)} \le 2L(\ve + \delta)
\end{equation}
and
\begin{equation}\label{eq:bound2}
\dfrac{1}{\ve} \abs{t_k-s_k}^2 + \dfrac{1}{\delta}\abs{x_k-y_k}^2 \le 4L^2(\ve + \delta).
\end{equation}
Also, $\abs{t_k-s_k},\,\abs{x_k-y_k} \to 0$ as $\ve,\,\delta \to 0$.

Using the optimality of $(t_k, x_k, s_k, y_k)$, we also have
\begin{align*}
& \quad V_k\btx - V^*\btx - \varphi(\bt,\bx,\bt,\bx) = \Phi_k(\bt,\bx,\bt,\bx)\\
& \le \Phi_k(t_k, x_k, s_k, y_k) = V_k(t_k, x_k) - V^*(s_k,y_k) - \varphi(t_k, x_k, s_k, y_k),
\end{align*}
which implies
\begin{equation}\label{eq:ineq1}
\quad V_k\btx - V^*\btx - \dfrac{2\lambda}{\bt}
\le V_k(t_k, x_k) - V^*(s_k,y_k) - \dfrac{\lambda}{t_k} - \dfrac{\lambda}{s_k}.
\end{equation}

Next, we separate into two cases. When $t_k$ or $s_k$ are close to $T$, we use the fact that $V_{\infty}(T,\cdot) = V^*(T,\cdot) = g(\cdot)$ to derive a contradiction. Conversely, when $t_k$ and $s_k$ are not close to $T$, we use positivity of $\lambda$ to derive a contradiction.

\emph{Case 1.} For any $K_0$ and $\alpha_0>0$, we can find $k \ge K_0$ and $\ve,\delta,\lambda < \alpha_0$ s.t. $t_k \vee s_k \ge T - \frac{\eta}{3L}$. Under this assumption, we can find a sequence $\{k_i,\ve_i,\delta_i,\lambda_i\}_{i=1}^{\infty}$ s.t. $k_i$ increases to infinity, $(\ve_i,\delta_i,\lambda_i)$ decrease to $0$s, and the corresponding $t_{k_i}, x_{k_i}, s_{k_i}, y_{k_i}$ satisfies $t_{k_i} \vee s_{k_i} \ge T - \frac{\eta}{3L}$ for all $i$.
Since $[0,T] \times \mX$ is bounded, we can pick a subsequence of $\{k_i, \ve_i, \delta_i, \lambda_i\}_{i=1}^{\infty}$ (without changing notations) such that $t_{k_i}, x_{k_i}, s_{k_i}, y_{k_i}$ all converge. \eqref{eq:ineq1} becomes 
\begin{equation}\label{eq:ineq1i}
\quad V_{k_i}\btx - V^*\btx - \dfrac{2\lambda_i}{\bt}
\le V_{k_i}(t_{k_i}, x_{k_i}) - V^*(s_{k_i},y_{k_i}) - \dfrac{\lambda_i}{t_{k_i}} - \dfrac{\lambda_i}{s_{k_i}}.
\end{equation}

Let $i \to \infty$, then $V_{k_i}\btx \to V_\infty\btx$ by \eqref{eq:limit_equal} and $\lambda_i \to 0$. Also, \eqref{eq:bound1} implies $s_{k_i},t_{k_i}$ converge to some same limit $t_{\infty} \ge T - \frac{\eta}{3L}$ and $x_{k_i}, y_{k_i}$ converge to some same limit $x_{\infty}$. For the term $V_{k_i}(t_{k_i}, x_{k_i})$, we have
\begin{align*}
& \quad \limsup_{i \to \infty} V_{k_i}(t_{k_i}, x_{k_i}) \le \limsup_{i \to \infty} \sqbra{ V_{k_i}(t_\infty, x_\infty) + L \parentheses{\abs{t_{k_i} - t_\infty} + \abs{x_{k_i} - x_\infty}}} \\
& \le V_\infty(t_\infty, x_\infty) + L \cdot 0 =  V_\infty(t_\infty, x_\infty)
\end{align*}
Also, $\lim_{i\to\infty} V^*(s_{k_i},y_{k_i}) = V^*(t_\infty,x_\infty)$ because $V^*$ is Lipschitz continuous.
Combining the analysis above and let $i \to \infty$, \eqref{eq:ineq1i} becomes
\begin{equation*}
\begin{aligned}
& \quad V_{\infty}\btx - V^*\btx 
\le V_{\infty}(t_{\infty}, x_{\infty}) - V^*(t_{\infty}, x_{\infty}) \\
& = V_{\infty}(t_{\infty}, x_{\infty}) - g(x_{\infty}) + g(x_{\infty}) - V^*(t_{\infty}, x_{\infty}) \\
& = V_{\infty}(t_{\infty}, x_{\infty}) - V_{\infty}(T, x_{\infty}) + V^*(T, x_{\infty}) - V^*(t_{\infty}, x_{\infty}) \\
& \le L \abs{T - t_{\infty}} + L \abs{T - t_{\infty}} \le 2L \dfrac{\eta}{3L} = \dfrac{2}{3} \eta < \eta,
\end{aligned}
\end{equation*}
which contradicts to \eqref{eq:contrary}. This contradiction comes from the assumption $t_k \vee s_k \ge T - \frac{\eta}{3L}$ in \emph{Case 1}.

\emph{Case 2.} There exist $K_0$ and $\alpha_0>0$ s.t. for any $k\ge K_0$ and $\ve,\delta,\lambda < \alpha_0$, we have $t_k \vee s_k < T - \frac{\eta}{3L}$. In this second case, we will only focus on the situation when $k\ge K_0$ and $\ve,\delta,\lambda < \alpha_0$. Without loss of generality, we assume $K_0 \ge 1$ and $\alpha_0 \le 1$. We fix $\lambda < \alpha_0$ and let $k, \ve, \delta$ vary. Define $M := 4K+2T+2/\bt$ and $r_0 := \min\{ \frac{\lambda}{M (M+1)}, \frac{\eta}{6L} \}$.
Note that $\lambda$ is fixed, so $r_0$ is an absolute constant. We also define
$$Q_0 := \{ (t,x,s,y) ~|~ \lambda/M \le t,s \le T-2r_0 \}.$$
Then $\Phi_k$ achieves its maximum $(t_k, x_k, s_k, y_k)$ in $Q_0$ because \eqref{eq:ineq1} cannot hold if $t_k \le \lambda/M$ or $s_k \le \lambda/M$. Next, we define
\begin{equation}\label{eq:Q_set}
Q := \{ (t,x,s,y) ~|~ \lambda/(M+1) < t,s < T-r_0 \}.
\end{equation}
We find $Q_0 \subset Q$ and
\begin{equation*}%\label{eq:tk_estimate}
\min\{t_k, s_k\} - \dfrac{\lambda}{M+1} \ge \dfrac{\lambda}{M} - \dfrac{\lambda}{M+1} = \dfrac{\lambda}{M(M+1)} \ge r_0,
\end{equation*}
and $T - r_0 - \max\{t_k, s_k\} \ge r_0$. Restricted in $Q$, $\Phi_k$ has bounded derivatives. 

Next, we will make some perturbation on $\Phi_k$ and verify that its maximum still lies in $Q$. We define $\mu>0$ by
\begin{equation}\label{eq:mu_def}
2\mu(K+1) + \mu K^2 = \dfrac{\lambda}{2T^2}.
\end{equation}
Define $r_1 := \mu r_0 / 4$. Let $(q,p,\hq,\hp) \in \RR^{1+n+1+n}$ s.t. $\abs{p},\abs{q},\abs{\hp},\abs{\hq} \le r_1$. Then we define a new function
\begin{equation}\label{eq:hPhik}
\begin{aligned}
    \hPhik(t,x,s,y) =& \Phi_k(t,x,s,y) - \dfrac{\mu}{2}\parentheses{\abs{t-t_k}^2 + \abs{x-x_k}^2 + \abs{s-s_k}^2 + \abs{y-y_k}^2}\\
    &+ q(t-t_k) + \inner{p}{x-x_k} + \hq(s-s_k) + \inner{\hp}{y-y_k}.
\end{aligned}
\end{equation}
The second term in the RHS of \eqref{eq:hPhik}, which starts with $\frac{\mu}{2}$, ensures that $(t_k, x_k, s_k, y_k)$ becomes a strict maximum, and the rest of (linear) terms in the second line can be viewed as a linear perturbation. 
So, $\hPhik$ achieves a maximum at some other point in $\RR^{1+n+1+n}$, denoted by $(\htk,\hxk,\hsk,\hyk)$. By optimality, $(\htk,\hxk,\hsk,\hyk)$ must lie in the set
\begin{align*}
& \left\{ (t,x,s,y) ~\Big|~ \dfrac{\mu}{2}\parentheses{\abs{t-t_k}^2 + \abs{x-x_k}^2 + \abs{s-s_k}^2 + \abs{y-y_k}^2} \right.\\
& \hspace{0.8in} \left. \le q(t-t_k) + \inner{p}{x-x_k} + \hq(s-s_k) + \inner{\hp}{y-y_k} \right\},
\end{align*}
which implies
$$\dfrac{\mu}{2} \abs{ (\htk,\hxk,\hsk,\hyk) - (t_k, x_k, s_k, y_k) }^2 \le \abs{ (\htk,\hxk,\hsk,\hyk) - (t_k, x_k, s_k, y_k) } \cdot \abs{(q,p,\hq,\hp) }.$$
Therefore,
$$\abs{ (\htk,\hxk,\hsk,\hyk) - (t_k, x_k, s_k, y_k) } \le \dfrac{2}{\mu} \abs{(q,p,\hq,\hp) } \le \dfrac{2}{\mu} ~ 2r_1 = r_0.$$
So $\abs{\htk - t_k}, \abs{\hsk - s_k} \le r_0$, which implies $(\htk,\hxk,\hsk,\hyk) \in Q$. More importantly, $(\htk,\hxk,\hsk,\hyk)$ lies in the interior of $(0,T] \times \mX \times (0,T] \times \mX$. So, by the optimality of $(\htk,\hxk,\hsk,\hyk)$, we have
\begin{equation}\label{eq:optimality_hPhik}
\left\{ \begin{aligned}
&0 = \pt\, \hPhik(\htk,\hxk,\hsk,\hyk) = \pt V_k(\htk,\hxk) - \pt\,  \varphi(\htk,\hxk,\hsk,\hyk) - \mu(\htk-t_k) + q\\
&0 = \ps\, \hPhik(\htk,\hxk,\hsk,\hyk) = -\ps V^*(\hsk,\hyk) - \ps\,  \varphi(\htk,\hxk,\hsk,\hyk) - \mu(\hsk-s_k) + \hq\\
&0 = \nx \hPhik(\htk,\hxk,\hsk,\hyk) = \nx V_k(\htk,\hxk) - \nx  \varphi(\htk,\hxk,\hsk,\hyk) - \mu(\hxk-x_k) + p\\
&0 = \ny \hPhik(\htk,\hxk,\hsk,\hyk) = -\ny V^*(\hsk,\hyk) - \ny  \varphi(\htk,\hxk,\hsk,\hyk) - \mu(\hyk-y_k) + \hp\\
&\begin{pmatrix} \nx^2V_k(\htk,\hxk) & 0 \\0 & -\ny^2V^*(\hsk,\hyk) \end{pmatrix} \le %\begin{pmatrix} \nx^2\varphi & \nx\ny\varphi \\ \ny\nx\varphi  & \ny^2\varphi \end{pmatrix} 
\nabla_{x,y}^2 \varphi
+ \mu I_{2n} = \dfrac{1}{\delta}\begin{pmatrix} I_n & -I_n \\ -I_n  & I_n \end{pmatrix} + \mu I_{2n}
\end{aligned} \right.
\end{equation}
as first and second order necessary conditions. Note that $(\htk,\hxk,\hsk,\hyk)$ depend on $\ve$, $\delta$, $\lambda$, $q$, $p$, $\hq$, $\hp$, $\mu$, and $k$, where $\lambda$ and $\mu$ are fixed.

For given $\ve$, $\delta$ and $k$, we can view $(\htk,\hxk,\hsk,\hyk)$ as an implicit function of $(q,p,\hq,\hp)$, given by the equation $\nabla \hPhik(\htk,\hxk,\hsk,\hyk) = 0$, i.e.,
\begin{equation}\label{eq:implicit_eqn}
(q,p,\hq,\hp) = -\nabla \Phi_k(\htk,\hxk,\hsk,\hyk) + \mu(\htk-t_k, \hxk-x_k, \hsk-s_k, \hyk-y_k).
\end{equation}
Here, the gradient is taken w.r.t. $(t,x,s,y)$.
We claim that the inverse also holds: we can view $(q,p,\hq,\hp)$ as an implicit function of $(\htk,\hxk,\hsk,\hyk)$ locally, also given by \eqref{eq:implicit_eqn}. 
The Jacobian of this implicit function is
\begin{equation}\label{eq:Ak}
A_k := \pd{(q,p,\hq,\hp)}{(\htk,\hxk,\hsk,\hyk)} = \mu I_{2n+2} - \nabla^2 \Phi_k(\htk,\hxk,\hsk,\hyk).
\end{equation}
We will show this claim by proving that $A_k$ is nonsingular locally. Let us restrict
\begin{equation}\label{eq:range_hat}
\abs{\htk-t_k}, \abs{\hxk-x_k}, \abs{\hsk-s_k}, \abs{\hyk-y_k} < r_2
\end{equation}
where
\begin{equation}\label{eq:range_r}
0 < r_2 < \mu / \sqbra{8\parentheses{L + 3(M+1)^4 / \lambda^3}}.
\end{equation}
Later, this $r_2$ will change according to $\ve$ and $\delta$, but will be independent with $k$. We give an estimate of the Hessian next.
\begin{equation}\label{eq:diff_d2Phi}
\begin{aligned}
& \quad \abs{ \nabla^2 \Phi_k(t_k, x_k, s_k, y_k) - \nabla^2 \Phi_k(\htk,\hxk,\hsk,\hyk) } \\
& \le \abs{\nabla^2 V_k(\htk, \hxk) - \nabla^2 V_k(t_k, x_k)} + \abs{\nabla^2 V^*(\hsk, \hyk) - \nabla^2 V^*(s_k, y_k)} \\
& \quad + 2\lambda \parentheses{ \abs{\htk^{-3} - t_k^{-3}} + \abs{\hsk^{-3} - s_k^{-3}} } \\
& \le L \abs{(\htk, \hxk) - (t_k, x_k)} + L \abs{(\hsk, \hyk) - (s_k, y_k)} \\
& \quad + 2\lambda \parentheses{ 3\abs{\htk - t_k} \parentheses{\min\{ \htk, t_k \}}^{-4} + 3\abs{\hsk - s_k} \parentheses{\min\{ \hsk, s_k \}}^{-4}}\\
& \le 4Lr_2 + 12\lambda r_2 \parentheses{ \lambda / (M+1) }^{-4} \le \dfrac{1}{2} \mu.
\end{aligned}
\end{equation}
Here, we used the the definition of $\varphi$ \eqref{eq:varphi} and $\Phi_k$ \eqref{eq:Phi_k} in the first inequality.
The third inequality is because \eqref{eq:range_hat} holds and the range of $\htk$, $\hsk$, $t_k$, $s_k$ are given by \eqref{eq:Q_set}. The fourth inequality comes from the range for $r_2$ in \eqref{eq:range_r}. In the second inequality, we used the Lipschitz condition for the derivatives of the value functions (including $\pt^2 V_k$ and $\ps^2 V^*(s,y)$) and mean value theorem. We remark that this estimate \eqref{eq:diff_d2Phi} makes the analysis hard to generalize to the viscosity solution of the HJB equation, which does not have sufficient regularity in general.
Therefore,
\begin{equation*}%\label{eq:Ak_lower}
\begin{aligned}
& \quad A_k = \mu I_{2n+2} - \nabla^2 \Phi_k(\htk, \hxk, \hsk, \hyk) \\
& = \mu I_{2n+2} - \nabla^2 \Phi_k(t_k, x_k, s_k, y_k) + \parentheses{\nabla^2 \Phi_k(t_k, x_k, s_k, y_k) - \nabla^2 \Phi_k(\htk, \hxk, \hsk, \hyk)} \\
& \ge \mu I_{2n+2} - \abs{\nabla^2 \Phi_k(t_k, x_k, s_k, y_k) - \nabla^2 \Phi_k(\htk,\hxk,\hsk,\hyk)} \cdot I_{2n+2} \ge \dfrac{1}{2} \mu \, I_{2n+2}.
\end{aligned}
\end{equation*}
Here, the inequality $\ge$ between two symmetric matrix means that their difference is positive semi-definite. In the first inequality, we use the fact that $\nabla^2 \Phi_k(t_k, x_k, s_k, y_k) \le 0$, coming from the optimality of $(t_k, x_k, s_k, y_k)$. In the second equality, we use the estimate \eqref{eq:diff_d2Phi}. Therefore, the Jacobian $A_k$ in \eqref{eq:Ak} always nonsingular when \eqref{eq:range_r} holds and we confirm the claim after \eqref{eq:implicit_eqn}.

Next, we also want to derive an upper bound for the Jacobian $A_k$. A direct calculation from \eqref{eq:Ak} gives us
\begin{equation*}
\begin{aligned}
\norm{A_k}_2 \le & \mu + \norm{\nabla^2 V_k(\htk,\hxk)}_2 + \norm{\nabla^2 V^*(\hsk,\hyk)}_2\\
& + \dfrac{1}{\ve} \norm{ \begin{pmatrix} 1 & -1 \\-1 & 1 \end{pmatrix}}_2 + \dfrac{1}{\delta} \norm{\begin{pmatrix} I_n & -I_n \\-I_n & I_n \end{pmatrix}}_2 + \dfrac{4\lambda}{(\lambda/M)^2}\\
\le & \mu + 2K + \dfrac{2}{\ve} + \dfrac{2}{\delta} + \dfrac{4M^2}{\lambda} \le C \parentheses{\dfrac{1}{\ve} + \dfrac{1}{\delta}}.
\end{aligned}
\end{equation*}
Note that the notation $\norm{\cdot}_2$ is the $l_2$ operator norm for a matrix (instead of the Frobenius norm). The first inequality above is by definition of $\varphi$ in \eqref{eq:varphi} and $\Phi_k$ in \eqref{eq:Phi_k}. The second is by boundedness of the derivatives of the value functions. The third is because $\lambda$ is fixed while $\ve$ and $\delta$ are small and are going to $0$s later. This estimation for the Jacobian gives us a bound for the implicit function
\begin{equation}\label{eq:regularity_p}
\abs{(q,p,\hq,\hp)} \le C \parentheses{\dfrac{1}{\ve} + \dfrac{1}{\delta}} \abs{(\htk-t_k, \hxk-x_k, \hsk-s_k, \hyk-y_k)}.
\end{equation}
Therefore, we also require that 
\begin{equation}\label{eq:range_r2}
r_2 \le \parentheses{\dfrac{1}{\ve} + \dfrac{1}{\delta}}^{-1} \dfrac{r_1}{2C}
\end{equation}
where the $C$ in \eqref{eq:range_r2} is the same as the $C$ in \eqref{eq:regularity_p}, in order to guarantee 
$\abs{(q,p,\hq,\hp)} \le r_1$. Now we can see that the $r_2$ depends on $\ve$ and $\delta$, but it is independent of $k$.

Next, we consider the quantity
\begin{equation*}%\label{eq:Bk}
B_k := \ps V^*(\hsk,\hyk) - \pt V_k(\htk,\hxk),
\end{equation*}
which depends on $\ve,\,\delta,\,\lambda,\,q,\,p,\,\hq,\,\hp,\,\mu,$ and $k$. On the one hand, by the optimality condition \eqref{eq:optimality_hPhik},
\begin{equation}\label{eq:one_hand}
\begin{aligned}
B_k &= -\ps \varphi(\htk,\hxk,\hsk,\hyk) -\pt \varphi(\htk,\hxk,\hsk,\hyk) - \mu \sqbra{(\hsk-s_k) + (\htk-t_k)} + q+\hq \\
& = \lambda/\htk^{~2} + \lambda/\hsk^{~2} - \mu \sqbra{(\hsk-s_k) + (\htk-t_k)} + q+\hq \\
& \ge 2\lambda/T^2 - \mu \parentheses{\abs{\hsk-s_k} + \abs{\htk-t_k}} + q+\hq,
\end{aligned}
\end{equation}
where the terms with $\ve$ in $\ps \varphi$ and $\pt \varphi$ cancel each other.

On the other hand, using the HJ equations that $V^*$ and $V_k$ satisfy, we have
\begin{equation*}
\begin{aligned}
B_k &= G(\hsk,\hyk,u^*(\hsk,\hyk),-\ny V^*,-\ny^2 V^*) - G(\htk,\hxk,u_k(\htk,\hxk),-\nx V_k,-\nx^2 V_k) \\
& = \sup_{u \in \RR^{n'}} G(\hsk,\hyk,u,-\ny V^*,-\ny^2 V^*) - G(\htk,\hxk,u_k(\htk,\hxk),-\nx V_k,-\nx^2 V_k) \\
& \le \sup_{u \in \RR^{n'}} G(\hsk,\hyk,u,-\ny V^*,-\ny^2 V^*) - \sup_{u \in \RR^{n'}} G(\htk,\hxk,u,-\nx V_k,-\nx^2 V_k) \\
& \quad + L \abs{u_k(\htk,\hxk) - u_k^{\diamond}(\htk,\hxk)}\\
& \le \sup_{u \in \RR^{n'}} \sqbra{G(\hsk,\hyk,u,-\ny V^*,-\ny^2 V^*) - G(\htk,\hxk,u,-\nx V_k,-\nx^2 V_k)}\\
& \quad + L \abs{u_k(\htk,\hxk) - u_k^{\diamond}(\htk,\hxk)},
\end{aligned}
\end{equation*}
where we have consecutively used: HJ equations for $V^*$ and $V_k$; the optimality condition \eqref{eq:max1} for $u^*$; the definition of $u_k^{\diamond}$ in \eqref{eq:udiamond} and the Lipschitz condition of $G$ in $u$; a simple property for supremum. Note that we omit the input $\hsk, \hyk$ for $\ny V^*, \ny^2 V^*$ and $\htk, \hxk$ for $\nx V_k, \nx^2 V_k$ for notational simplicity. Therefore, by the definition of $G$ in \eqref{eq:generalized_Hamiltonian}, 
\begin{equation}\label{eq:other_hand}
\begin{aligned}
B_k & \le \sup_{u \in \RR^{n'}} \left\{ \frac12 \Tr \sqbra{ \nx^2V_k(\htk,\hxk)\sigma\sigma\tp(\hxk) - \ny^2 V^*(\hsk,\hyk) \sigma\sigma\tp(\hyk)} \right.\\
& \quad +\sqbra{ \inner{\nx V_k(\htk,\hxk)}{b(\hxk,u)} - \inner{\ny V^*(\hsk,\hyk)}{b(\hyk,u)} } \\
& \quad + r(\hxk,u) - r(\hyk,u) \bigg\} + L \abs{u_k(\htk,\hxk) - u_k^{\diamond}(\htk,\hxk)}\\
& =: \sup_{u \in \RR^{n'}} \curlybra{(\rom{1}) + (\rom{2}) + (\rom{3})} + L \abs{u_k(\htk,\hxk) - u_k^{\diamond}(\htk,\hxk)}.
\end{aligned}
\end{equation}

Next, we bound the three terms in \eqref{eq:other_hand}. Using the estimates \eqref{eq:bound1} and \eqref{eq:bound2}, we can easily show that
\begin{equation}\label{eq:bound3}
\abs{\hxk-\hyk} \le \abs{\hxk-x_k} + \abs{\hyk-y_k} + \abs{x_k - y_k} \le 2r_2 + 2L(\ve + \delta)
\end{equation}
and
\begin{equation}\label{eq:bound4}
\dfrac{1}{\delta}\abs{\hxk-\hyk}^2 \le \dfrac{1}{\delta} \parentheses{8r_2^2 + 2\abs{x_k - y_k}^2} \le \dfrac{8r_2^2}{\delta} + 8L^2(\ve + \delta)
\end{equation}

For $(\rom{3})$, we have
\begin{equation}\label{eq:term3}
(\rom{3}) = r(\hxk,u) - r(\hyk,u) \le L \abs{\hxk -\hyk} \le 2Lr_2 + 2L^2(\ve+\delta),
\end{equation}
where we used Lipschitz condition of $r$ in Assumption \ref{assump:basic} and \eqref{eq:bound3}.

For $(\rom{2})$, we have
\begin{equation}\label{eq:term2}
\begin{aligned}
(\rom{2}) &= \inner{\nx \varphi(\htk,\hxk,\hsk,\hyk) + \mu(\hxk-x_k) -p}{b(\hxk,u)}\\
& \quad + \inner{\ny \varphi(\htk,\hxk,\hsk,\hyk) + \mu(\hyk-y_k) -\hp}{b(\hyk,u)}\\
& = \inner{\dfrac{1}{\delta} (\hxk-\hyk) + \mu(\hxk-x_k) -p}{b(\hxk,u)} \\
& \quad + \inner{ \dfrac{1}{\delta} (\hyk-\hxk) + \mu(\hyk-y_k) -\hp}{b(\hyk,u)}\\
& \le \dfrac{L}{\delta} \abs{\hxk-\hyk}^2 + \mu K(\abs{\hxk-x_k} + \abs{\hyk-y_k}) + K(\abs{p}+\abs{\hp})\\
& \le 8r_2^2L/\delta + 8L^3(\ve+\delta) + \mu K(\abs{\hxk-x_k} + \abs{\hyk-y_k}) + K(\abs{p}+\abs{\hp}),
\end{aligned}
\end{equation}
where we have consecutively used: the optimality condition \eqref{eq:optimality_hPhik}; the definition of $\varphi$ in \eqref{eq:varphi}; boundedness and Lipschitz condition of $b$ in Assumption \ref{assump:basic}; the bound \eqref{eq:bound4}.

For $(\rom{1})$, we have
\begin{equation}\label{eq:term1}
\begin{aligned}
(\rom{1}) &= \dfrac12 \Tr \sqbra{ \begin{pmatrix}  \sigma(\hxk) \\ \sigma(\hyk) \end{pmatrix}\tp \begin{pmatrix} \nx^2V_k(\htk,\hxk) &0 \\ 0& -\ny^2V^*(\hsk,\hyk) \end{pmatrix} \begin{pmatrix}  \sigma(\hxk) \\ \sigma(\hyk) \end{pmatrix} }\\
&\le \dfrac12 \Tr \sqbra{ \begin{pmatrix}  \sigma(\hxk) \\ \sigma(\hyk) \end{pmatrix}\tp \parentheses{\dfrac{1}{\delta}\begin{pmatrix} I_n & -I_n \\ -I_n & I_n \end{pmatrix} + \mu I_{2n}} \begin{pmatrix}  \sigma(\hxk) \\ \sigma(\hyk) \end{pmatrix} }\\
&= \dfrac{1}{2\delta} \abs{\sigma(\hxk) - \sigma(\hyk)}^2 + \dfrac{\mu}{2} \parentheses{\abs{\sigma(\hxk)}^2 + \abs{\sigma(\hyk)}^2}\\
& \le \dfrac{L^2}{2\delta} \abs{\hxk-\hyk}^2 + \mu K^2 \le 4r_2^2L^2/\delta + 4L^4(\ve + \delta) + \mu K^2,
\end{aligned}
\end{equation}
where we have consecutively used: a simple transform in linear algebra; the second order optimality condition in \eqref{eq:optimality_hPhik}; a simple calculation; boundedness and Lipschitz condition of $\sigma$ in Assumption \ref{assump:basic}; the bound \eqref{eq:bound4}.

Combining \eqref{eq:one_hand}, \eqref{eq:other_hand}, \eqref{eq:term3}, \eqref{eq:term2}, and \eqref{eq:term1}, we obtain
\begin{equation*}
\begin{aligned}
& \quad 2\lambda/T^2 - \mu \parentheses{\abs{\hsk-s_k} + \abs{\htk-t_k}} + q+\hq \\
& \le 4r_2^2L^2/\delta + 4L^4(\ve + \delta) + \mu K^2 +8r_2^2L/\delta + 8L^3(\ve+\delta) + \mu K\parentheses{\abs{\hxk-x_k} + \abs{\hyk-y_k}} \\
& \quad + K(\abs{p}+\abs{\hp}) + 2Lr_2 + 2L^2(\ve+\delta) + L \abs{u_k(\htk,\hxk) - u_k^{\diamond}(\htk,\hxk)}.
\end{aligned}
\end{equation*}
Using the bounds \eqref{eq:regularity_p} and \eqref{eq:range_hat}, we get
\begin{equation}\label{eq:ineq2}
\begin{aligned}
2\lambda/T^2  \le 2\mu(K+1)r_2 + \mu K^2 + (4L^4+8L^3+2L^2)(\ve+\delta) + (4r_2^2L^2+8r_2^2L)/\delta \\
+ (2K+2)C\parentheses{\dfrac{1}{\ve} + \dfrac{1}{\delta}}r_2 + L \abs{u_k(\htk,\hxk) - u_k^{\diamond}(\htk,\hxk)}.
\end{aligned}
\end{equation}
Next, we pick a box in $\RR^{2n+2}$ that centered at $(t_k, x_k, s_k, y_k)$ and have side length $r_3 = 2r_2/\sqrt{n}$. Then $\abs{\hxk - x_k}, \abs{\hyk - y_k} \le \sqrt{n}r_3/2 = r_2$, so that all the estimates before hold in this box. If we integrate \eqref{eq:ineq2} over the box w.r.t. $(\htk,\hxk,\hsk,\hyk)$ and divided it by $r_3^{2n+2}$, we obtain
\begin{equation}\label{eq:ineq3}
\begin{aligned}
2\lambda/T^2  &\le 2\mu(K+1)r_2 + \mu K^2 + (4L^4+8L^3+2L^2)(\ve+\delta) + (4r_2^2L^2+8r_2^2L)/\delta\\
& \quad + (2K+2)C\parentheses{\dfrac{1}{\ve} + \dfrac{1}{\delta}}r_2 +  (2r_2/\sqrt{n})^{-2n-2} L \norm{u_k- u_k^{\diamond}}_{L^1}.
\end{aligned}
\end{equation}
We recall that $\lambda$ is fixed at first. We also recall that the definition of $\mu$ in \eqref{eq:mu_def} ensures that
\begin{equation}\label{eq:contra1}
2\mu(K+1)r_2 + \mu K^2 \le \dfrac{\lambda}{2T^2}.
\end{equation}
Therefore, if we firstly set $\ve$ and $\delta$ to be small such that 
\begin{equation}\label{eq:contra2}
(4L^4+8L^3+2L^2)(\ve + \delta) < \dfrac{\lambda}{2T^2}.
\end{equation}
Then we set $r_2$ to be small such that
\begin{equation}\label{eq:contra3}
(4r_2^2L^2+8r_2^2L)/\delta + (2K+2)C\parentheses{\dfrac{1}{\ve} + \dfrac{1}{\delta}}r_2 < \dfrac{\lambda}{2T^2},
\end{equation}
where the $C$ in \eqref{eq:contra3} is the same as the $C$ in \eqref{eq:ineq3}. Next, note that $\norm{u_k- u_k^{\diamond}}_{L^1} \le \sqrt{T} \norm{u_k- u_k^{\diamond}}_{L^2}$. So, by \eqref{eq:hard_case}, we can set $k$ to be large enough such that 
\begin{equation}\label{eq:contra4}
(2r_2/\sqrt{n})^{-2n-2} L \norm{u_k- u_k^{\diamond}}_{L^1} < \dfrac{\lambda}{2T^2},
\end{equation}
where we used a key argument that $r_2$ is independent of $k$, given in \eqref{eq:range_r2}.
Finally, substituting \eqref{eq:contra1}, \eqref{eq:contra2}, \eqref{eq:contra3}, and \eqref{eq:contra4} into \eqref{eq:ineq3}, we obtain an contradiction, so Lemma \ref{lem:claim} is proved.
\end{proof}

\end{appendices}
\bibliography{ref}
\end{document}